\documentclass{article}%
\usepackage{amsfonts}
\usepackage{amsmath}
\usepackage{amssymb}
\usepackage{float}
\usepackage{makecell}
\usepackage{comment}
\usepackage{caption}
\usepackage{hyperref}
\hypersetup{linktoc = all}
\usepackage{tikz}
\usetikzlibrary{positioning}
\usetikzlibrary{arrows,decorations.markings}
\usetikzlibrary{shapes.geometric, arrows}
\usepackage{arydshln}
\usepackage{graphicx}
\usepackage[T1]{fontenc}
\usepackage{float}
\providecommand{\U}[1]{\protect\rule{.1in}{.1in}}

\newtheorem{theorem}{Theorem}

\newtheorem{corollary}[theorem]{Corollary}

\newtheorem{definition}[theorem]{Definition}
\newtheorem{example}[theorem]{Example}

\newtheorem{lemma}[theorem]{Lemma}

\newtheorem{proposition}[theorem]{Proposition}
\newtheorem{remark}[theorem]{Remark}

\newenvironment{proof}[1][Proof]{\noindent\textbf{#1.} }{\ \rule{0.5em}{0.5em}}

\usepackage{titlesec}
\titleformat{\paragraph}
{\normalfont\normalsize\bfseries}{\theparagraph}{1em}{}
\titlespacing*{\paragraph}
{0pt}{3.25ex plus 1ex minus .2ex}{1.5ex plus .2ex}
\definecolor{AnthonyComments}{HTML}{00BFFF}

\begin{document}
\sloppy
\captionsetup[figure]{labelfont={bf},name={Fig.},labelsep=none}

\title{Extension of the Bessmertny\u{\i} Realization Theorem for Rational Functions
of Several Complex Variables}
\author{Anthony Stefan\footnote{Email: astefan2015@my.fit.edu}\; and Aaron Welters\footnote{Email: awelters@fit.edu}\\Florida Institute of Technology\\Department of Mathematical Sciences\\Melbourne, FL, USA}
\date{\today}
\maketitle

\begin{abstract}
We prove a realization theorem for rational functions of several complex
variables which extends the main theorem of M. Bessmertny\u{\i},
\textquotedblleft On realizations of rational matrix functions of several
complex variables," in Vol. 134 of \textit{Oper. Theory Adv. Appl.}, pp.
157-185, Birkh\"{a}user Verlag, Basel, 2002. In contrast to Bessmertny\u{\i}'s
approach of solving large systems of linear equations, we use an operator theoretical approach based on the theory of Schur complements. This leads to a simpler and more \textquotedblleft natural" construction to solving the realization problem as we need only apply elementary algebraic operations to Schur complements such as sums, products, inverses, and compositions. A novelty of our approach is the use of Kronecker product as opposed to the matrix product in the realization problem. As such our synthetic approach leads to a solution of the realization problem that has potential for further extensions and applications within multidimensional systems theory especially for those linear models associated with electric circuits, networks, and composites.
\end{abstract}

\section{Introduction}

As part of M. Bessmertny\u{\i}'s 1982 Ph.\ D.\ thesis (in Russian) \cite{82MB}, he proved that every rational matrix-valued function of several variables could be written as the Schur complement of a linear matrix pencil (i.e., a \textit{Bessmertny\u{\i} long resolvent representation}). As mentioned in \cite{04KV}, this theorem of his (a.k.a, the \textit{Bessmertny\u{\i} realizability theorem}), was unknown to Western readers until parts of it were translated into English beginning in 2002 with \cite{02MB}. Given the potential applications of this important theorem to realization problems in multivariate systems theory, electric network theory, and the theory of composites, we want to consider this theorem, its proof and its extensions, from a different viewpoint and method of approach than that of M. Bessmertny\u{\i}. In fact, we were very inspired and motivated by the abstract theory of composites approach to a similar realization problem developed by G.\ Milton in \cite{16GM7}, as well as the approach of R.\ Duffin in \cite{55RD} on synthesis problems in electrical network theory that can be solved using elementary algebraic operations.

In this paper, we are interested in giving an alternative construction [in comparison to the approach of M. Bessmertny\u{\i} \cite{02MB} by solving large systems of linear equations (see Subsec. \ref{SubsecRelevantWork} for details) or that which can be achieved using the methods in D. Alpay and C. Dubi \cite{03AD} based on Gleason's problem] that solves the following version of the \textit{Bessmertny\u{\i} realization problem}: Given a rational $\mathbb{C}^{k\times k}$-valued matrix function $f(z)=f(z_1,\ldots, z_n)$ of $n$ complex variables $z_1,\ldots, z_n$ [$z=(z_1,\ldots, z_n)$], find a linear matrix pencil 
\begin{align}\label{TheLinearPencil}
A(z)=A_0+z_1A_1+\cdots+z_nA_n
\end{align}
such that $f(z)$ is representable as the Schur complement 
\begin{align}\label{TheBessmLongResolvRepr}
 f(z)=A_{11}(z)-A_{12}(z)A_{22}(z)^{-1}A_{21}(z)
\end{align}
of a $2\times 2$ block matrix form $A(z)=[A_{ij}(z)]_{i,j=1,2}$ with respect to its $(2,2)$-block $A_{22}(z)$ [denoted by $A(z)/A_{22}(z)$]. If this is possible, we say that $f(z)$ is \textit{realizable} (or can be \textit{realized}) and has a \textit{realization}.

The main theorem of M.\ Bessmertny\u{\i} in \cite[Theorem 1.1]{02MB} solves this realization problem, and his construction of the linear matrix pencil $A(z)=A_0+z_1A_1+\cdots+z_nA_n$ from $f(z)$ involves solving large systems of constrained linear equations in such a way that $A(z)$ inherits certain real, symmetric, or homogeneity properties from $f(z)$. 

The main theorem of our paper, namely, Theorem \ref{ThmBessmRealiz}, also solves this realization problem, but our construction of the linear matrix pencil $A(z)=A_0+z_1A_1+\cdots+z_nA_n$ from $f(z)$ uses a theory of Schur complement/realization algebra and operations that we develop in Section \ref{SecSchurComplAlgebraAndOps}. In fact, we approach and prove our Bessmertny\u{\i}
Realization Theorem (i.e., Theorem \ref{ThmBessmRealiz}) in a systematic way using the following steps (i)-(v) to realize $f(z)=\frac{P(z)}{q(z)}$, where $P(z)$ is a matrix polynomial and $q(z)$ is a scalar polynomial not identically equal to $0$: \begin{itemize}
    \item[(i)] The degree-1 scalar monomials $z_j$ ($j=1,\ldots,n$) are realizable.
    \item[(ii)] A (constant) scalar multiple of a realizable function is realizable.
    \item[(iii)] Sums of realizable rational matrix functions (of fixed square size) are realizable.
    \item[(iv)] Kronecker products of realizable functions are realizable (based on realizability of simple products). Thus, monomials of arbitrary degree are realizable and hence matrix polynomials are realizable.
    \item[(v)] If $Q$ is a matrix polynomial with $\det Q(z)\not\equiv 0$, then $Q(z)^{-1}$ is realizable. Using this and (iv), we get a realization of $$f(z)=\frac{P(z)}{q(z)}=[q(z)]^{-1}\otimes P(z).$$
\end{itemize}
In addition, our realization method via steps (i)-(v) (using the Kronecker product instead of matrix product, although matrix product is likely a viable approach) has the added advantage that it naturally preserves symmetries and as such we need only make minor adjustments to our proof to construct symmetric realizations. In particular, our approach allows us to extend M.\ Bessmertny\u{\i}'s theorem \cite[Theorem 1.1]{02MB} with statements (d) and (e) in our Theorem \ref{ThmBessmRealiz} in which we can construct a Hermitian matrix pencil $A(z)=A_0+z_1A_1+\cdots+z_nA_n$ (i.e., all the matrices $A_j$ are Hermitian matrices) if $f(z)$ is ``Hermitian" [i.e., has the functional property $f(\overline {z})^*=f(z)$] or a matrix pencil with any combination of the symmetries (real, symmetric, Hermitian, or degree-one homogeneous) if $f(z)$ also has the same combination of symmetries.

Another merit of our paper is the thorough catalog of Schur complement/realization algebra and operations in Sec. \ref{SecSchurComplAlgebraAndOps}. We show that for all elementary algebraic-functional operations (including ones we did not need for solving the realization problem) when applied to Schur complements (such as linear combinations, products, inversion, and composition) is equal to another Schur complement of a block matrix and we give explicit formulas to compute it. It is expected that these results will find application in multidimensional systems theory such as in electrical network theory and the theory of composites as well as for other realization problems involving the Schur complement. In fact, the proof of the Bessmertny\u{\i} realization theorem can be considered a good example of how the results of Sec. \ref{SecSchurComplAlgebraAndOps} could be used in applications (for more on the applications see \cite{21AS}).

Another instance in which Sec. \ref{SecSchurComplAlgebraAndOps} and other methods in this paper could be useful is in extensions of the Bessmernty\u{\i} realization theorem to realizations of rational matrix functions having additional symmetries. For example, in electric network theory an important class of functions known as multivariate reactance functions that have the functional symmetries $-f(-z)^T=f(z)$ and $\overline{f(\overline{z})}=f(z)$ \cite{74BN}. Realizations of these types of functions along with some other symmetries that arise naturally in the theory of composites and electric network theory, will be considered in future work.

The main motivation and long-term goal of these results and extensions is to try to make some progress toward solving several open problems on realizability from the theory of composites (for the relevant mathematical theory of composites and the many open problems in it, see \cite{16YG, 02GM, 16GM}) as well as in electrical network theory (see, for instance, \cite{11JB} and references within). For these open problems, the class of functions which are of interest are those $f(z)$ that are rational positive-real functions of $n$-variables and are also homogeneous degree-one functions. For such functions, the main question we are interested in answering is whether or not there exists a homogeneous linear matrix pencil \begin{align}\label{HomoLinPen}A(z)=z_1A_1+\cdots+z_nA_n\end{align} that also has these functional properties and gives a realization of $f(z)$ (the class of such functions that have such a realization are known as the rational \textit{Bessmertny\u{\i} class of functions}, see \cite{11JB, 04KV}). For $n=1,2$, it is known that such a realization is possible, but it is still an open problem for $n\geq 3$ (see \cite{11JB, 04KV}). Of particular importance in the theory of composites (see \cite{16GM10, 87GM1, 87GM2}, \cite[Chap. 29]{02GM}, and \cite{16GM7}) are such functions $f(z)$ that also satisfy the normalization condition $f(1,\ldots, 1)=I_k$ ($I_k$ is the $k\times k$ identity matrix) and the associated subclass of the Bessmertny\u{\i} class of functions (the \textit{Milton class of functions}) that can be realized with the homogeneous linear matrix pencil (\ref{HomoLinPen}) also satisfying the normalization condition, i.e., $A_1+\cdots+A_n=I_m$. For more information on these open problems see \cite[Chap. 7]{21AS}. 

The rest of the paper will proceed as follows. In the remainder of this section we will give a review of relevant work. Then, in Sec.\ \ref{SecPreliminaries} we establish the notation, conventions, and preliminary results used in the paper. In Sec.\ \ref{SecMainResultBessmRealiz} we state and prove the main theorem of this paper, namely, the Bessmernty\u{\i} Realization Theorem (Theorem \ref{ThmBessmRealiz}). In Section \ref{SecExamplesOfBessRealizThm} we give several examples that show how to work out the realization of some rational functions using our methods. In Sec.\ \ref{SecSchurComplAlgebraAndOps} we develop a theory for the elementary algebraic-functional operations on Schur complements that occur in Bessmertny\u{\i} realization problems. This includes linear combinations, products (including Kronecker products), matrix inversion, and composition of Schur complements as well as several others. Finally, we conclude with Section \ref{SecTransformsForAlternativeRealizations} on additional transformations of matrices associated with the Schur complement and how they can be used to give alternative realization theorems for rational matrix functions in conjunction with the Bessmernty\u{\i} Realization Theorem (Theorem \ref{ThmBessmRealiz}). The main focus there is on the \textit{principal pivot transform} (PPT) which is an important transformation in the context of network synthesis problems (see, for instance, \cite{65DHM, 66DHM, 00MT}) and may also play an important role in the theory of composites.

\subsection{Relevant work}\label{SubsecRelevantWork}
The realizability theory that we are interested in is commutative (as opposed to noncommutative) multivariate rational matrix-valued functions (especially with symmetries) over the complex field. An analogy that one could draw from our work to the single-variable setting is that of descriptor form realizations of rational functions as transfer functions of descriptor systems which is in contrast to the standard state space realizations usually associated with R. Kalman (i.e., Kalman-type realizations). We will briefly elaborate on this below and the relevance of the previous work to ours in the single variable as well as in the multivariate setting.

The state space realizability theory for rational matrix
functions of one variable which are analytic at infinity (i.e., proper functions) is well developed. In this setting, any such rational function $f(z)$ of a single complex variable $z$ has the Kalman-type realization
\begin{align}\label{KalmanTypeRealization}
    f(z)=D+C(zI-A)^{-1}B, 
\end{align}
for some complex matrices $A,B,C,D$ where $A$ is a square matrix and $I$ is the identity matrix the same size as $A$. The size of the matrix $A$ is called the ``dimension of the state space" (or ``dimension of the realization") and a ``minimal realization" of $f(z)$ is any realization of the form \eqref{KalmanTypeRealization} with smallest possible dimension of the state space. In this context, using the following linear matrix pencil in $2\times 2$ block matrix form
\begin{align}
    A(z)=\left[\begin{array}{c;{2pt/2pt}c}
    A_{11}(z) & A_{12}(z)  \\ \hdashline 
    A_{21}(z) & A_{22}(z)
    \end{array} \right]=\left[\begin{array}{c;{2pt/2pt}c}
    D & C  \\ \hdashline 
    B & A-zI
    \end{array} \right]=\left[\begin{array}{c;{2pt/2pt}c}
    D & C  \\ \hdashline 
    B & A
    \end{array} \right]+z\left[\begin{array}{c;{2pt/2pt}c}
    0 & 0  \\ \hdashline 
    0 & -I
    \end{array} \right],
\end{align}
the function $f(z)$ is the Schur complement of $A(z)=[A_{ij}(z)]_{i,j=1,2}$ with respect to the $(2,2)$-block $A_{22}(z)=A-zI$, that is,
\begin{align}
    f(z)=D+C(zI-A)^{-1}B=A(z)/A_{22}(z).
\end{align}
In particular, this is just a very special case of a Bessmertny\u{\i} realization of $f(z)$ in one complex variable $z$ whenever $f(z)$ is also a square matrix-valued rational function.

One of the main developers of the state space realization theory was R. Kalman based on a major result he proved in 1963 \cite{63RK} (and elaborated on in his 1965 paper \cite{65RK}) that bridged linear control theory and the concepts of controllability and observability with minimal realizability and the construction of minimal realizations of a given rational function from nonminimal realizations of it. He also proved the important `state space similarity theorem' which describes how two minimal realizations of the same rational function are related by similarity (see \cite[Theorem 8]{63RK}, \cite[Proposition 2]{65RK}, \cite[Theorem 3.1]{79BGK}, and \cite[Theorem 7.7]{08BGKR}). Shortly thereafter, in 1965 \cite{65RK} he showed how these concepts were related to the notion of the McMillan degree of a rational matrix-valued function (also called the McMillan-Duffin-Hazony degree as it was, according to R. Kalman \cite{65RK}, first introduced by B. McMillan in 1952 \cite{52BM} and further studied by R. Duffin and D. Hazony in 1963 \cite{63DH}, but credit for an equivalent definition of degree and its usefulness in network realization theory actually seems to belong to B. Tellegen based on his 1948 paper, see \cite{48BT,68AN}). In 1966 \cite{66HK}, R. Kalman and his Ph.D. student B. Ho showed how this was connected to the theory of Hankel matrices associated with the power series expansion of a rational function about infinity (related to Pad\'{e} approximation theory, continued fractions, Markov parameters, and the theory of moments, see \cite{86BB}) with an elegant realization algorithm to construct a minimal realization from this theory.

There are a couple relevant things we want to point out in this regard. First, R. Kalman was motivated by some special cases considered by E. Gilbert in \cite{63EG} and this paper already has the formulas for sum and matrix products of Schur complements \{cf. \cite[Fig. 2, Eq. (11), Theorem 3]{63EG} and \cite[Fig. 3, Eq. (12), Theorem 4]{63EG}, respectively\} that we have listed (Propositions \ref{PropSumOfSchurComps} and \ref{PropMatrixMultipliationOfTwoSchurComplements}) which were derived for the realization of the sum and product of transfer functions based on the analogy of the parallel and cascade connections of electrical networks. Another clear instance of the formulas for product and inverses of Schur complements can be found in \cite[pp. 6-8]{79BGK} (see also \cite[Secs. 2.1-2.3]{08BGKR}), which play an important role in minimal factorization problems for matrices and operators using the `state space method.' The method itself was motivated by theory of operator nodes (or colligations) and characteristic functions of linear operators based on the pioneering work of M. S. Liv\v{s}ic starting in the middle of the 1940s (see \cite{72JH}, \cite{76JH}, and \cite[Sec. 5]{01VK} for more on this as well as \cite{73ML}). Second, using that method, D. Alpay and I. Gohberg \cite{88AG} and D. Alpay, J. Ball, I. Gohberg, and L. Rodman \cite{90ABGR}, \cite{92ABGR}, \cite{94aABGR}, \cite{94bABGR} have very effectively studied the realization problem for proper rational functions with symmetries in state space form with emphasis on the minimal realization part of the theory.

In order to treat the realization problem for non-proper rational matrix functions in the single variable case, descriptor representations were introduced in the 1970's. In this setting, any rational function $f(z)$ of a single complex variable $z$ has the descriptor realization
\begin{align}\label{DescriptorRealization}
    f(z)=D+C(zE-A)^{-1}B, 
\end{align}
for some complex matrices $A,B,C,D$ where $A$ and $E$ are square matrices of the same size. According to \cite{89JS}, it was D. Luenberger in 1977-1978 (\cite{77DL} and \cite{78DL}) who was the first to make an extensive study of this representation in systems theory. In contrast to the standard state space representation, i.e., the Kalman-type realization (\ref{KalmanTypeRealization}), the descriptor form (\ref{DescriptorRealization}) is capable of representing systems having a non-proper transfer function, i.e., $f(z)$ need not be analytic at infinity (also called `descriptor systems' or `singular systems').

Thus, as the descriptor form is a universal form for representing any rational matrix function of one variable, the analogy of this with the Bessmertny\u{\i} long resolvent representation in one variable, i.e., (\ref{TheLinearPencil}) and (\ref{TheBessmLongResolvRepr}), is more clear since it is a universal form for any commutative multivariate rational square matrix-valued function. This analogy is further justified since, in this context, the descriptor form (\ref{DescriptorRealization}) is again a special case of a Bessmertny\u{\i} realization in the square matrix case of one variable. Indeed, in this context using the following linear matrix pencil in $2\times 2$ block matrix form
\begin{align}
    A(z)=\left[\begin{array}{c;{2pt/2pt}c}
    A_{11}(z) & A_{12}(z)  \\ \hdashline 
    A_{21}(z) & A_{22}(z)
    \end{array} \right]=\left[\begin{array}{c;{2pt/2pt}c}
    D & C  \\ \hdashline
    B & A-zE
    \end{array} \right]=\left[\begin{array}{c;{2pt/2pt}c}
    D & C  \\ \hdashline 
    B & A
    \end{array} \right]+z\left[\begin{array}{c;{2pt/2pt}c}
    0 & 0  \\ \hdashline 
    0 & -E
    \end{array} \right],
\end{align}
the function $f(z)$ is the Schur complement of $A(z)=[A_{ij}(z)]_{i,j=1,2}$ with respect to the $(2,2)$-block $A_{22}(z)=A-zE$, that is,
\begin{align}
    f(z)=D+C(zE-A)^{-1}B=A(z)/A_{22}(z).
\end{align}

Of course, there are a multitude of different forms of realizations for rational matrix-valued functions in the one variable setting (see \cite{89JS} for a survey), but the Kalman-type and descriptor form realizations give a clear idea of the most basic and prominent relevant results.

In the $n$-variable setting (with $n\geq 2$), starting in the 1970s, models of multidimensional systems having transfer functions equal to $n$-variable rational matrix functions were introduced and research began on converse realization theorems \cite{72GR, 73GR, 75RR, 76FM, 77KLMK, 78FM, 78ES} (see also \cite{05BGM} and \cite{18HMS}). It became clear that the noncommutative and commutative cases are quite distinct regarding techniques, theorems, and open problems. We will briefly elaborate on this below and the relevance of the previous work to ours in the multivariate setting.

In the noncommutative setting, for a $n$-D rational matrix-valued function regular at zero there are analogous results from the classical $1$-D case of the relation between observability and controllability to minimal realizations along with a corresponding state space similarity theorem and construction of minimal realizations from nonminimal relations. The main reason for this analogy stems from the relation between noncommutative formal power series representations of such rational functions at zero, minimal realizations, and their associated Hankel matrices/operators (\cite{74MF}, \cite{08CR}, \cite{10BR}), for instance, the rank of the Hankel matrix is the minimal possible dimension of a realization in analogy to the $1$-D case. In fact, an application of this was a realization theorem in 1978 by E. Fornasini and G. Marchesini \cite[Proposition 1]{78FM} which used in the proof the noncommutative realization theory of M. Fliess from \cite{74MF} to get a commutative realization in the $2$-D case, called a `Fornasini-Marchesini realization' for a commutative `Fornasini-Marchesini system' and is a multivariate analogy of a Kalman-type realization. They also showed that the models previously investigated, i.e., the Givone-Roesser model \cite{72GR, 73GR, 75RR}, and the Kung-L\'{e}vy-Morf-Kailath model \cite{77KLMK} could be embedded in their state space model, the Fornasini-Marchesini system \cite[(1)]{78FM}. 

The study of these systems and their generalization to `structured noncommutative linear systems,' was developed extensively in 2005 by J. Ball, G. Groenewald, and T. Malakorn \cite{05BGM} and includes results on their standard system-theoretical properties that are analogous to $1$-D Kalman state space realization theory (e.g., the operations of cascade/parallel connection and inverses, controllability, observability, Kalman decomposition, state space similarity theorem, minimal state space realizations, Hankel operators, realization theory). Also in 2005, D. Alpay and D. Kalyuzhny\u{\i}-Verbovetzki\u{\i} \cite{05AK} treat the realization problem with symmetries for noncommutative rational formal power series in the $n$-D noncommutative Givone-Roesser realization form, which is an extension of work of D. Alpay and I. Gohberg \cite{88AG} in the single variable case as discussed above. Without going in to precise details and focusing on matrices instead of operators, a relevant result is that all these structured systems have transfer functions which are rational functions in noncommutative variables (indeterminates) $z=(z_1,\ldots, z_n)$ which are regular at zero and have the form
\begin{align}\label{NoncommRegAt0RationalRepr}
    f(z)=D+C(I-Z(z)A)^{-1}Z(z)B,
\end{align}
for complex matrices $A,B,C,D$, where $A$ is a matrix, $I$ is an identity matrix, and $Z(z)$ is a linear pencil in $z$ of the form
\begin{align}\label{NoncommRegAt0RationalReprLinearPencil}
    Z(z)=Z_1z_1+\cdots+Z_nz_n,
\end{align}
for certain square matrices $Z_j$ having entries in $\{0,1\}$ only \{see \cite[Eq. (1.7) and Sec. 3]{05BGM}\}. The converse of this result \cite[Corollary 12.4]{05BGM} essentially says that every noncommutative $n$-variable rational matrix function $f(z)$ which is regular at zero can be represented in the form (\ref{NoncommRegAt0RationalRepr}) for some structured noncommutative linear system. In this context, using the following linear matrix pencil in $2\times 2$ block matrix form
\begin{align}
    A(z)&=\left[\begin{array}{c;{2pt/2pt}c}
    A_{11}(z) & A_{12}(z)  \\ \hdashline 
    A_{21}(z) & A_{22}(z)
    \end{array} \right]=\left[\begin{array}{c;{2pt/2pt}c}
    D & C  \\ \hdashline 
    Z(z)B & Z(z)A-I
    \end{array} \right]\\
    &=A_0+A_1z_1+\cdots+A_nz_n,\\
    A_0&=\left[\begin{array}{c;{2pt/2pt}c}
    D & C  \\ \hdashline 
    0 & -I
    \end{array} \right],\;\;A_l=\left[\begin{array}{c;{2pt/2pt}c}
    0 & 0  \\ \hdashline 
    Z_lB & Z_lA
    \end{array} \right],\;\;\text{for }l=1,\ldots,n,
\end{align}
the function $f(z)$ is the Schur complement of $A(z)=[A_{ij}(z)]_{i,j=1,2}$ with respect to the $(2,2)$-block $A_{22}(z)=Z(z)A-I$, that is,
\begin{align}\label{NoncommSchurCompRealization}
    f(z)=D+C(I-Z(z)A)^{-1}Z(z)B=A(z)/A_{22}(z).
\end{align}
In particular, restricting to $n$ complex variables $z=(z_1,\ldots,z_n)$, this is just a very special case of a Bessmertny\u{\i} realization of $f(z)$ whenever $f(z)$ is also a square matrix-valued rational function. For instance, the well-known Fornasini-Marchesini and Givone-Roesser realizations (see \cite{78FM, 72GR, 73GR, 75RR, 05BGM}) are of this form, see Table \ref{BessmerntyiSepcialCasesTable} for more on these realizations (where $I_{\mathcal{H}}$ can be a different sized identity matrix then $I$).

One relevant issue to note here is that the $n$-D structured realizations (\ref{NoncommRegAt0RationalRepr}) are similar to the $1$-D Kalman-type realizations (\ref{KalmanTypeRealization}) in that the form of the realization depends on the choice of a point of regularity (e.g., the point zero in the former and infinity in the latter). This is also true of the realization results in D. Alpay and C. Dubi \cite[Theorem 1.1]{03AD}, they require regularity at zero for the realization of a rational matrix function of several complex variables (and they show that their form of realization is a special case of a Bessmertny\u{\i} realization \cite[p. 226]{03AD}). As such, this does not cover all possible rational matrix functions unless one adjusts the form of the realization. In contrast, the Bessmertny\u{\i} realization uniformly treats all rational square matrix functions with the same form of realization regardless of the regularity at a given point.

In addition, the Bessmertny\u{\i} realization (\ref{TheLinearPencil}) and (\ref{TheBessmLongResolvRepr}) is a more suitable form of realization for certain models of interest (cf., \cite[Chap. 3]{21AS}) especially when the desired form of the linear matrix pencil is homogeneous, i.e., of the form (\ref{HomoLinPen}). For example, when the rational matrix function is modeled as an impedance matrix of an electrical network associated with a finite linear graph/structure (see \cite{82MB}, \cite{05MB}, \cite[Chap. 3, Secs. 3.1-3.3]{21AS}) or is modeled as an effective tensor in the theory of composites (see \cite{02GM}, \cite{16GM2, 16GM7, 16GM10}, \cite[Chap. 3, Sec. 3.4]{21AS}). In fact, the latter model was the primary motivator for our paper, whereas the former model motivated M. Bessmertny\u{\i} to introduce such realizations in this 1982 Ph.D. thesis (\cite{82MB}, see also \cite{05MB}) in order to consider the multivariate analogy of the single variable analysis of some earlier work by his Ph.D. advisor V. P. Potapov with A. V. Efimov \cite{73EP} on realization theory in electric circuit theory (in which itself was based on results of M. S. Liv\v{s}ic, see \cite[p. 75]{73EP} and \cite{73ML}).

One of the main points of this section is to discuss how we treat the realization problem for multivariate rational matrix functions with symmetries using a `synthetic' approach, as opposed to other approaches such as that of Bessmertny\u{\i} who treats the problem by solving systems of equations \cite{82MB}, \cite{02MB}. We will briefly describe Bessmertny\u{\i}'s procedure, contrast it with ours, and then compare it to relevant work that uses a different synthetic approach.

Given a rational $k\times k$ matrix-valued function $f(z)$ of $n$-complex variables $z=(z_1,\ldots,z_n)$, Bessmertny\u{\i} \cite[Sec. 1]{02MB} would seek a linear matrix pencil in $2\times 2$ block matrix form $A(z)=[A_{ij}(z)]_{i,j=1,2}=A_0+z_1A_1+\cdots+z_nA_n$ with $\det A_{22} \not \equiv 0$ that solves the linear system of equations
\begin{align}\left[\begin{array}{c;{2pt/2pt}c}\label{BessLinearEqs}
    A_{11}(z) & A_{12}(z)  \\ \hdashline
     A_{21}(z) & A_{22}(z)
\end{array}\right]\left[\begin{array}{c}
    I_{k} \\ \hdashline
\Phi(z)
\end{array}\right] = \left[\begin{array}{c} 
     f(z)  \\  \hdashline
     0
\end{array}\right],\end{align}
from which it would follow that $f(z)=A(z)/A_{22}(z)$, i.e., $f(z)$ has a Bessmertny\u{\i} realization. Essentially, Bessmertny\u{\i} solves the realization problem by showing that one can solve these system of equations (\ref{BessLinearEqs}) in the following order of increasing complexity: ratios of monomials \cite[Lemma 1.1, Corollary 1.1]{05MB}, ratios of scalar polynomials \cite[Lemma 1.4]{05MB}, and ratios of the form $\frac{P(z)}{q(z)}$, where $q(z)$ is a scalar polynomial and $P(z)$ is a matrix polynomial \cite[Theorem 1.1, see proof in Sec. 1.6]{05MB}. Moreover, he shows that the matrix pencils $A(z)$ that arise in his procedure for realizing a rational matrix function $f(z)$ with symmetries can also have desired symmetries (as mentioned above in the introduction).

In contrast, our synthetic approach is based on the idea that rational functions are built up from monomials as the elementary building blocks using a finite number of operations, namely, the binary operations of addition, scalar multiplication, and products, and the unary operation of inversion. As such, we need only consider how the Schur complement (as a function on $2\times 2$ block matrices) interacts with those building blocks and operations to produce another Schur complement which preserves symmetries (see Section \ref{SecSchurComplAlgebraAndOps}). By doing so, we are able to solve the realization problem with symmetries following the steps (i)-(v) outlined in the introduction. In particular, in step (iv) our use of \textit{Kronecker products} as the binary operation for products instead of matrix products is unconventional. Our motivation for using this algebraic approach with the Kronecker product is mainly based on the following two points. First, it is more natural in realization problems for the effective tensor in the theory of composites (see \cite{87GM1, 87GM2}, \cite[Chap. 29]{02GM}, and \cite[Chap. 7]{16GM7}. Second, as briefly mentioned in the introduction above, the Kronecker product $\otimes$ is a binary operation that naturally preserves symmetry compared to the matrix product. Indeed, since for any two complex matrices $A,B$ we have  $\overline{(A \otimes B)} = \overline{A} \otimes \overline{B}$, $(A \otimes B)^T = A^T \otimes B^T$, and $(A \otimes B)^* = A^* \otimes B^*$ so that if, for example, $A,B$ are symmetric, i.e., $A^T=A$ and $B^T=B$, then their Kronecker product is symmetric, i.e., $(A \otimes B)^T = A\otimes B$, but in general their matrix product $AB$ (when their sizes are compatible for their product to be defined) is not symmetric if the matrices $A$ and $B$ do not commute, i.e., $(AB)^T=BA\not=AB$.




Now in comparison, the realization problem with symmetries can be solved synthetically in a different way. The conventional approach (going at least as far back as the work of D. Alpay and I. Gohberg \cite{88AG} in the single variable case with an emphasis on minimal realizations, an approach that was also used in \cite{90ABGR}, \cite{92ABGR}, \cite{94bABGR}, \cite{94aABGR} and, for the multivariate free noncommutative power-series case, in \cite{05AK}) is to first solve the realization problem using \textit{matrix products} to get a realization, not necessarily with the desired symmetries, for a given rational matrix function $f(z)$ having symmetries and then apply the symmetry operation to realize it (possibly after some additional manipulations) into the desired realization form with symmetries \{e.g., the symmetry operation applied to a function $f(z)$ satisfying the Hermitian symmetry $f(z)=f(\overline{z})^*$ would be $\frac{1}{2}[f(z)+f(z)^*]$\}, for instance, see \cite[Theorem 4.2 and Theorem 4.9]{18HMS} and their proofs. In particular, this is the approach taken in 2018 by J. Helton, T. Mai, and R. Speicher \cite{18HMS} in treating the realization problem for a given rational matrix function $f(z)$ of the noncommutative variables $z=(z_1,\ldots,z_n)$, which is not necessarily regular at zero, in the formal linear representation form
\begin{align}\label{FormalLinearRepr}
   f(z) = -uQ(z)^{-1}v,
\end{align} where $u$ and $v$ are complex matrices and $Q(z)$ is a linear matrix pencil of the form 
\begin{align}\label{FormalLinearReprQPencil}
    Q(z)=Q^{(0)}+Q^{(1)}z_1+\cdots+Q^{(n)}z_n,
\end{align} 
with $Q^{(j)}$ square complex matrices of the same size for $j=0,\ldots, n$. Their synthetic approach (using sums, matrix product, and inverses, see \cite[Algorithm 4.3 and 4.11]{18HMS}), proves that every such function $f(z)$ has a formal linear representation form \cite[Theorem 4.2 and 4.12]{18HMS} and in the case of the Hermitian symmetry, has a self-adjoint formal linear representation form \cite[Theorem 4.9 and 4.14]{18HMS} in which in (\ref{FormalLinearRepr}) and (\ref{FormalLinearReprQPencil}), $u=v^*$ and $Q^{(j)}={Q^{(j)}}^*$, i.e., are Hermitian matrices for each $j=0,\ldots,n$. In this context, using the following linear matrix pencil in $2\times 2$ block matrix form \begin{align}
    A(z)&=\left[\begin{array}{c;{2pt/2pt}c}
    A_{11}(z) & A_{12}(z)  \\ \hdashline 
    A_{21}(z) & A_{22}(z)
    \end{array} \right]=\left[\begin{array}{c;{2pt/2pt}c}
    0 & u  \\ \hdashline 
    v & Q(z)
    \end{array} \right] =A_0+A_1z_1+\cdots+A_nz_n,\\ A_0&=\left[\begin{array}{c;{2pt/2pt}c}
    0 & u  \\ \hdashline 
    v & Q^{(0)}
    \end{array} \right],\;\;A_l=\left[\begin{array}{c;{2pt/2pt}c}
    0 & 0  \\ \hdashline 
    0 & Q^{(j)}
    \end{array} \right],\;\;\text{for }l=1,\ldots,n,
\end{align} 
the function $f(z)$ in (\ref{FormalLinearRepr}) is the Schur complement of $A(z)=[A_{ij}(z)]_{i,j=1,2}$ with respect to the $(2,2)$-block $A_{22}(z)=Q(z)$, i.e., \begin{align}
   f(z) = -uQ(z)^{-1}v = A(z)/A_{22}(z).
\end{align}
In particular, restricting to $n$ complex variables $z=(z_1,\ldots,z_n)$, this is just a special case of a Bessmertny\u{\i} realization of $f(z)$ whenever $f(z)$ is also a square matrix-valued rational function.

It should be pointed out here that the existence of a formal linear representation (\ref{FormalLinearRepr}) and (\ref{FormalLinearReprQPencil}) for any \textit{scalar} noncommutative rational function is well known, see \cite[pp. 4 and 5]{18HMS}, \cite[p. 614]{17KPV}, \cite{94CR}, \cite{99CR}, and the references within. And in this context, the application of the synthetic approach for realization using the operations of sums, matrix products, and inverses also appears, for instance, see \cite[p. 312, Corollary 1.3]{99CR}, and a good minimal realization theory exists, see \cite[Theorem 1.4, Corollary 1.6, and Theorem 1.7]{99CR}. 

Although we have discussed several special cases of the Bessmertny\u{\i} realization (i.e., Kalman-type to formal linear realizations), there are a few others worth mentioning, which are summarized in Table \ref{BessmerntyiSepcialCasesTable}. The columns in this table are organized in the following order: name of the realization and its regularity, the form of that realization, the structure of the partitioned $2\times 2$ block matrix whose Schur complement with respect to the $(2,2)$-block is that form, the linear pencil equal to that block matrix, and reference(s) to that realization. The rows are collected into three groups. The first group is the Bessmertny\u{\i} realization (the $1$st row), whereas the second and third group are the single- and multi-variable special cases, respectively. One thing to point out is the realization that breaks the pattern in the $2$nd column, that is not obviously a Schur complement, is the butterfly realization (in the last row). This is due to the fact that the butterfly realization is actually the following sum of two Schur complements \begin{align}
    \left.\left[\begin{array}{c; {2pt/2pt} c }
r(z) & \Lambda(z) \\\hdashline[2pt/2pt]
\Lambda(z)^T & A(z)-J
\end{array}\right]
\right/
\begin{bmatrix}
A(z)-J
\end{bmatrix} +  \left.\left[\begin{array}{c; {2pt/2pt} c }
0 & \ell(z) \\\hdashline[2pt/2pt]
\ell(z)^T & -I
\end{array}\right]
\right/
\begin{bmatrix}
-I
\end{bmatrix},
\end{align}
which is, by Proposition \ref{PropSumOfSchurComps}, the Schur complement of the block matrix in the $3$rd column.

In conclusion, our synthetic approach to the Bessmertny\u{\i} realization problem which naturally incorporate symmetries by using the Kronecker product appears to be new in realizability theory for multivariate rational matrix functions. Likely, it can be generalized and used in the commutative and noncommutative setting for rational matrix functions of any size. Furthermore, our results in Section \ref{SecSchurComplAlgebraAndOps} may be useful in other ways besides using it to prove the Bessmertny\u{\i} realization theorem (Theorem \ref{ThmBessmRealiz}). For instance, in \cite{21SW} (see also \cite{21AS}) the authors used the results in Section \ref{SecSchurComplAlgebraAndOps} to give a short and elementary proof of the symmetric determinantal representation for multivariate polynomials over the reals and, more generally, over an arbitrary field with characteristic $\not =2$. There is one other potential application worth mentioning (besides that which we already mentioned on realizability problems in the theory of composites and the Bessmertny\u{\i} class of functions). One major problem in realization theory for multivariate rational matrix functions (in the nonregular case) is that of constructing minimal or reduced dimensional realizations, see \cite[pp. 28 and 31]{18HMS}, \cite[p. 1478]{05BGM}, \cite{14MD}, \cite{08NB}, and \cite{01KG}, for instance. In particular, there are no general minimal realization constructions for Bessmertny\u{\i} realizations nor are there general bounds for the dimension of realization in current realization algorithms. Our methods in this paper along with the results in Sec. \ref{SecSchurComplAlgebraAndOps} may help in regard to this problem.
\begin{table}
\vspace{-3cm}
\hspace{-3.34cm}
\begin{tabular}{|>{\centering\arraybackslash}m{1in}| >{\centering\arraybackslash}m{1.8in}|>{\centering\arraybackslash}m{1.8in} |>{\centering\arraybackslash}m{1.8in} |>{\centering\arraybackslash}m{0.25in}|}
\hline
Name & Form & Block Matrix & Linear Pencil & Refs. \\ 
\hline
\hline
\shortstack{\\\textbf{\textit{Bessmertny\u{\i}}}\\ \textbf{\textit{realization}} \\ \scriptsize{(\textit{\textbf{Universal}})} \\ \scriptsize{$z=(z_1,\ldots, z_n)$} }  & $A_{11}(z)-A_{12}(z)A_{22}^{-1}(z)A_{21}(z)$ & \shortstack{\makecell[c]{$\left[\begin{array}{c;{2pt/2pt}c}A_{11}(z)&  A_{12}(z)\\ \hdashline A_{21}(z) & A_{22}(z)\end{array} \right]$\\ } } & $A_{0}+z_1A_{1}+\cdots + z_nA_{n}$ & \shortstack{\cite{82MB}\\\cite{02MB}\\ \cite{21AS}} \\ 
\hline
\hline
 \shortstack{\\\textit{Kalman-type} \\\scriptsize{(\textit{Regular at $z_1=\infty$})}} & $D+C(z_1I-A)^{-1}B$ & \makecell[c]{$\left[\begin{array}{c;{2pt/2pt}c}D&  C\\ \hdashline B & A-z_1I\end{array} \right]$} & $\begin{bmatrix} D & C\\B & A \end{bmatrix} + z_1\begin{bmatrix} 0 & 0\\0 & -I \end{bmatrix}$ & \shortstack{\\\cite{65RK}\\\cite{63RK}\\\cite{08BGKR}}
\\ \hline 
 \shortstack{\\ \textit{Descriptor} \\ \textit{realization} \\
 \scriptsize{(\textit{Nonregular at}}\\ \scriptsize{\textit{$z_1=\infty$ allowed})}} & $D+C(z_1E-A)^{-1}B$ & \shortstack{\\\makecell[c]{$\left[\begin{array}{c;{2pt/2pt}c}D& C \\ \hdashline B & A-z_1E\end{array} \right]$}\\ \textit{``regular:"} $\det E\not= 0$, \\ \textit{``non-regular:"} \\ $\det E= 0$ \textit{allowed}} & $\begin{bmatrix} D & C\\B & A \end{bmatrix} + z_1\begin{bmatrix} 0 & 0\\0 & -E \end{bmatrix}$ & \shortstack{\\ \cite{77DL} \\ \cite{78DL}}
\\ \hline \hline
 \shortstack{\textit{Fornasini-} \\ \textit{Marchesini} \\ \textit{realization} \\ (\scriptsize{\textit{Regular at $z=0$}})} & \shortstack{\\ \makecell[c]{$D+C(I-Z(z)A)^{-1}Z(z)B$} \\$A=\begin{bmatrix}
 A_i
 \end{bmatrix}_{i=1,\ldots,n}$\\ $B=\begin{bmatrix}
 B_i
 \end{bmatrix}_{i=1,\ldots,n}$} & \shortstack{\\ \makecell[c]{$\left[\begin{array}{c;{2pt/2pt}c}D& C \\ \hdashline Z(z)B & Z(z)A-I \end{array} \right]$} \\ $Z(z)=\begin{bmatrix}z_1I_{\mathcal{H}}& \cdots & z_nI_{\mathcal{H}}\end{bmatrix}$} & \shortstack{ \makecell[c]{$\begin{bmatrix} D & C\\0 & -I \end{bmatrix} + \sum\limits_{j=1}^n z_j\begin{bmatrix} 0 & 0 \\ B_j & A_j \end{bmatrix}$}} & \shortstack{\\\cite{05BGM} \\\cite{78FM}}
\\  \hline 
 \shortstack{\textit{Givone-} \\ \textit{Roesser} \\ \textit{realization} \\ (\scriptsize{\textit{Regular at $z=0$}})} & \shortstack{ \makecell[c]{$D+C(I-Z(z)A)^{-1}Z(z)B$} \\$A=\begin{bmatrix}A_{ij}\end{bmatrix}_{i,j=1,\ldots,n}$\\$B=\begin{bmatrix}B_{ij}\end{bmatrix}_{i,j=1,\ldots,n}$ } & \shortstack{\\ \makecell[c]{$\left[\begin{array}{c;{2pt/2pt}c}D& C \\ \hdashline Z(z)B & Z(z)A-I \end{array} \right]$} \\ \makecell[c]{$Z(z)=\begin{bmatrix}
 z_1I_{\mathcal{H}} && \\& \ddots &\\&&z_nI_{\mathcal{H}}
 \end{bmatrix}$}} & \shortstack{\\\makecell[c]{\vspace{0.2cm}$\begin{bmatrix} D & C\\0 & -I \end{bmatrix} + \sum\limits_{j=1}^n z_j\begin{bmatrix} 0 & 0\\ B_j & A_j \end{bmatrix}$}} & \shortstack{\\\cite{05BGM}\\ \cite{72GR}\\ \cite{73GR}\\\cite{75RR}}
\\  \hline 
\shortstack{\\\textit{Formal} \\ \textit{linear} \\ \textit{realization} \\  \scriptsize{(\textit{Universal})}} & \shortstack{$-uQ(z)^{-1}v$\\$Q(z)=Q^{(0)}+\sum\limits_{j=1}^n z_jQ^{(j)}$ } &  \makecell[c]{$\left[\begin{array}{c;{2pt/2pt}c}0&  u\\ \hdashline v & Q(z) \end{array} \right]$} & \shortstack{\\ $\begin{bmatrix} 0 & u\\v & Q^{(0)} \end{bmatrix} +\sum\limits_{j=1}^n z_j\begin{bmatrix} 0 & 0\\0 & Q^{(j)} \end{bmatrix}$} & \shortstack{\\\cite{18HMS}\\\cite{17KPV}\\ \cite{94CR}\\ \cite{99CR}}
\\\hline 
\shortstack{\textit{Recognizable}\\ {\textit{realization}}  \\\scriptsize{(\textit{Regular at $z=0$})}} & \shortstack{\\$C(I-A(z))^{-1}B$ \\ $A(z)=\sum\limits_{j=1}^n z_jA_j$} & \makecell[c]{$\left[\begin{array}{c;{2pt/2pt}c}0& C \\ \hdashline B & A(z)-I \end{array} \right]$} & \shortstack{\\$\begin{bmatrix} 0 & C\\ B & -I \end{bmatrix} +\sum\limits_{j=1}^n z_j\begin{bmatrix} 0 & 0\\0 & A_j \end{bmatrix}$} &  \shortstack{\cite{05BGM}\\\cite{10BR}}
\\\hline 
 \shortstack{\textit{Descriptor} \\ \textit{realization} \\  \scriptsize{(\textit{Nonregular at}}\\ \scriptsize{\textit{$z=0$ allowed})} } & \shortstack{$D+C(E-A(z))^{-1}B$ \\ $A(z)=\sum\limits_{j=1}^n z_jA_j$} & \shortstack{\\\makecell[c]{$\left[\begin{array}{c;{2pt/2pt}c}D& C \\ \hdashline B & A(z)-E \end{array} \right]$}\\ \textit{``regular:"} $\det E\not= 0$, \\ \textit{``non-regular:"} \\ $\det E= 0$ \textit{allowed}} &  \shortstack{$\begin{bmatrix} D & C\\ B & -E \end{bmatrix} + \sum\limits_{j=1}^n z_j\begin{bmatrix} 0 & 0\\0 & A_j \end{bmatrix}$} &\cite{18HMS} \\\hline 
 \shortstack{\textit{Realization}\\ \textit{centered at $0$} \\ \\\scriptsize{(\textit{Regular at $z=0$})}}& \shortstack{\\$D+C(I- A(z))^{-1}(B(z))$ \\ $A(z)= \sum\limits_{j=1}^n z_jA_j$\\ $B(z)=\sum\limits_{j=1}^n z_jB_j$}& \makecell[c]{$\left[\begin{array}{c;{2pt/2pt}c} D & C \\\hdashline B(z)   & A(z) - I \end{array} \right]$} & \shortstack{\\$\begin{bmatrix} D & C\\ 0 & -I \end{bmatrix} + \sum\limits_{j=1}^n z_j\begin{bmatrix} 0 & 0\\B_j & A_j \end{bmatrix}$} &\cite{05BGM} \\\hline 
 \shortstack{\textit{Butterfly} \\ \textit{realization} \\\scriptsize{(\textit{Regular at $z=0$})}} & \makecell[c]{\\\shortstack{$r(z)+\ell(z)\ell(z)^T$\\$\phantom{A}+\Lambda(z)(J- A(z))^{-1}\Lambda(z)^T$\\$r(z)=r_0+r_1z_1$\\$\ell(z)= \sum\limits_{j=1}^n z_j\ell_j$\\ $\Lambda(z)=\Lambda_0+\sum\limits_{j=1}^nz_j\Lambda_j$\\ $J^2=I$, $J^T=J$\\$A(z)= \sum\limits_{j=1}^n z_jA_j$, $A_j=A_j^T$}} &
 \makecell[c]{\shortstack{$\left[\begin{array}{c;{2pt/2pt}cc} r(z) & \Lambda(z) & \ell(z) \\ \hdashline \Lambda(z)^T & A(z)-J & 0 \\ \ell(z)^T & 0  & -I\end{array} \right]$ \\ \textit{``monic:"} $J=I$,\\\textit{``pure:"} $\Lambda_0 =0$}} & \makecell[c]{$
     \begin{bmatrix} r_0 & \Lambda_0 & 0\\ \Lambda_0^T & -J & 0 \\ 0&0&-I\end{bmatrix}$ \\ $+ z_1\begin{bmatrix} r_1 & \Lambda_1 & \ell_1\\ \Lambda_1^T & A_1 &0  \\ \ell_1 &0 &0\end{bmatrix} $\\ $+ \sum\limits_{j=2}^n z_j\begin{bmatrix} 0 & \Lambda_j & \ell_j\\  \Lambda_j^T & A_j&0\\\ell_j&0&0 \end{bmatrix}$} & \shortstack{\cite{06HMV}}
\\\hline
\end{tabular}\caption{Special cases of Bessmertny\u{\i} realizations.}\label{BessmerntyiSepcialCasesTable}
\end{table} 
\section{Preliminaries}\label{SecPreliminaries}

Let $\mathbb{C}$ denote the field of complex numbers, the set of $n$-tuples of complex numbers by $\mathbb{C}^n$, where $z = (z_1, \ldots, z_n)$ denotes a point in $\mathbb{C}^n$, and $\mathbb{C}^{m\times n}$ denotes the set of all $m\times n$ matrices with entries in $\mathbb{C}$. Complex conjugate of a complex number $c$ will be denoted as $\overline{c}$. The transpose and conjugate transpose of a matrix $A$ will be denoted by $A^T$ and $A^*$ (i.e., $A^*=\overline{A}^T$), respectively, and the inverse of an invertible matrix $A$ will be denoted by $A^{-1}$. The symbol $\det A$ will denote the determinant of a square matrix $A$. The $m\times m$ identity matrix and the $m\times m$ zero matrix will be denoted by $I_m$ and $0_m$, respectively.

We will denote any matrix $A$ that is partitioned in $2\times2$ block matrix form as 
\[A = [A_{ij}]_{i,j=1,2} = \begin{bmatrix}
A_{11} &A_{12} \\ A_{21} &A_{22}
\end{bmatrix}=\left[\begin{array}{c; {2pt/2pt} c }
A_{11} & A_{12} \\\hdashline[2pt/2pt]
A_{21} & A_{22}
\end{array}\right],\] 
where the matrix $A_{ij}$ is called the $(i,j)$-block of $A$. The direct sum $A\oplus B$ of two matrices $A$ and $B$ is defined to be the $2\times2$ block matrix
\[
A\oplus B=%
\begin{bmatrix}
A & 0\\
0 & B
\end{bmatrix}
.
\]
The Schur complement of a matrix $A=[A_{ij}]_{i,j=1,2}$ with respect to $A_{22}$ [i.e., with respect to its $(2,2)$-block $A_{22}$], will be denoted as $A/A_{22}$ and defined by 
\[ A /A_{22} = A_{11} - A_{12} A_{22}^{-1} A_{21},\]
when $A_{22}^{-1}$ exists.

Some key elementary properties of the Schur complement, under the assumption that $A_{22}^{-1}$ exists, are:
\begin{align}
\left(\lambda A \right) / \left(\lambda A \right)_{22} &= \lambda \left( A/A_{22}\right),\label{ElemPropSchurComplementHomogeneousDegree1}\\
\overline{\left( A/A_{22} \right)} &= \left( \overline{A} \right) / \left( \overline{A} \right)_{22},\label{ElemPropSchurComplementConjugationPreserving}
\\ \left( A/A_{22} \right)^{T} &= \left( A^{T} \right) / \left( A^{T} \right)_{22},\label{ElemPropSchurComplementTransposePreserving}
\\ \left( A/A_{22} \right)^{\ast} &= \left( A^{\ast} \right) / \left( A^{\ast} \right)_{22},\label{ElemPropSchurComplementConjugationTransposePreserving}
\end{align}
for every $\lambda\in \mathbb{C}\setminus\{0\}$.

The Kronecker product of the two matrices $A = [a_{ij}]^{m,n}_{i,j = 1}$ $\in \mathbb{C}^{m\times n}$ and $B \in \mathbb{C}^{p\times q}$ is the matrix  $A \otimes B\in \mathbb{C}^{mp\times qn}$ defined as \[A \otimes B = [a_{ij}B]^{m,n}_{i,j = 1} \in \mathbb{C}^{mp\times nq}.\]

A linear matrix pencil is a $\mathbb{C}^{m\times m}$-valued matrix function of the form \[A(z) = A_0 + z_1A_1 + \cdots + z_nA_n,\] where $A_0, A_1, \ldots, A_n\in \mathbb{C}^{m\times m}$. If, in addition, the matrices $A_{k}$ are all $2\times 2$ block matrices partitioned conformally, then we can partition the matrix function $A(z)$ conformally as a $2\times 2$ block matrix and denote this by
\[
A(z) =[A_{ij}(z)]_{i,j=1,2}= \begin{bmatrix}
A_{11}(z) & A_{12}(z)\\
A_{21}(z) & A_{22}(z)
\end{bmatrix}=\left[\begin{array}{c; {2pt/2pt} c }
A_{11}(z) & A_{12}(z) \\\hdashline[2pt/2pt]
A_{21}(z) & A_{22}(z)
\end{array}\right],
\]
in other words, this block structure of $A(z)$ is independent of the variable $z$. In this case, we may write the Schur complement
\begin{align}
    A(z)/A_{22}(z)=A_{11}(z)-A_{12}(z)A_{22}(z)^{-1}A_{21}(z),
\end{align}
whenever $A_{22}(z)$ is an invertible matrix, and treat it as rational matrix-valued function of $z$. For any rational $\mathbb{C}^{m\times m}$-valued function $A(z)$, we will write $\det A(z)\not \equiv 0$ whenever $\det A(z)$ is not identically equal to $0$ as a rational function. 

For simplicity sake, in the statement of Bessmertny\u{\i}
Realization Theorem (Theorem \ref{ThmBessmRealiz}) below, we will abuse this notation (just as M. Bessmertny\u{\i} does, see \cite[p. 170, last para.\ in Sec.\ 1]{02MB}) and treat a linear matrix pencil $A(z)=A_{11}(z)=A_0 + z_1A_1 + \cdots + z_nA_n$ as a degenerate case of Schur complement $A(z)=A(z)/A_{22}(z)$, i.e., the matrix $A_{22}(z)$ is a $0\times 0$ matrix, in which case we can ignore the statement $\det A_{22}(z)\not \equiv 0$. But throughout the rest of the paper, we will not abuse this notation of the Schur complement in order to avoid confusion and because the resulting statements we want to prove often are quite different in their proofs in the degenerate case vs. the non-degenerate case (as can be seen in Sec.\ \ref{SecSchurComplAlgebraAndOps}).

\section{Extension of the Bessmertny\u{\i}
Realization Theorem\label{SecMainResultBessmRealiz}}

The following is our extension of the main theorem of M.\ Bessmertny\u{\i} \cite[Theorem 1.1]{02MB}. The proof of statements (a)-(c) are originally due to M.\ Bessmertny\u{\i} \cite{82MB, 02MB}. The statements (d) and (e) are new and extends his results (our modest contributions). Our constructive proof of the theorem below uses the results in Section \ref{SecSchurComplAlgebraAndOps} (see Fig. \ref{fig-FlowDiagPrfThmBessmRealiz}) and in general the approach we take may be of independent interest in the areas of linear algebra, operator theory, and multidimensional systems theory. To illustrate our approach of this theorem, we give two examples of realizations in Section \ref{SecExamplesOfBessRealizThm}. Before we state our main theorem though, we want to make a remark regarding statement (a).

\begin{remark}\label{RemSomeConfusionRegardingStatementAOfBessmertnyi}
It appears that there is an error or at least some confusion that needs to be cleared up in regard to one of the statements in the main theorem of M.\ Bessmertny\u{\i} \cite[Theorem 1.1.c)]{02MB} which seems to have propagated in the literature (see \cite[p. 256]{04KV}). Namely, if a rational $\mathbb{C}^{k\times k}$-valued matrix function $f(z_1,\ldots, z_n)$ of $n$-variables satisfies $f(\lambda z)=\lambda f(z)$, i.e., is a homogeneous degree one function, which can be represented as a Schur complement $f(z)=A(z)/A_{22}(z)$ of a linear matrix pencil $A(z)=[A_{ij}(z)]_{i,j=1,2}=A_0+z_1A_1+\cdots+z_nA_n,$ where $\det A_{22}(z)\not\equiv 0$, then it need not be the case that $A_0=0$. To see this consider the following simple example:
\begin{align}
f(z)=[z_1]=\left.\left[\begin{array}{c; {2pt/2pt} c }
z_1 & 0 \\\hdashline[2pt/2pt]
0 & 1
\end{array}\right]
\right/
\begin{bmatrix}
1
\end{bmatrix}=A(z)/A_{22}(z),
\end{align}
where
\begin{gather}
A(z)=A_0+z_1A_1=\left[\begin{array}{c; {2pt/2pt} c }
A_{11}(z) & A_{12}(z) \\\hdashline[2pt/2pt]
A_{21}(z) & A_{22}(z)
\end{array}\right]=\left[\begin{array}{c; {2pt/2pt} c }
z_1 & 0 \\\hdashline[2pt/2pt]
0 & 1
\end{array}\right],\\
A_0=\left[\begin{array}{c; {2pt/2pt} c }
0 & 0 \\\hdashline[2pt/2pt]
0 & 1
\end{array}\right],\;\;A_1=\left[\begin{array}{c; {2pt/2pt} c }
1 & 0 \\\hdashline[2pt/2pt]
0 & 0
\end{array}\right].
\end{gather}
This example and our statement below in Theorem \ref{ThmBessmRealiz}.(a) below should now help to clear up any confusion regarding \cite[Theorem 1.1.c)]{02MB}.
\end{remark}

\begin{theorem}[Bessmertny\u{\i} Realization Theorem]\label{ThmBessmRealiz}
Every rational $\mathbb{C}^{k\times k}$-valued matrix function $f(z_1,\ldots, z_n)$ of $n$-variables can be represented as a Schur complement $f(z)=A(z)/A_{22}(z)$ of a linear matrix pencil $A(z)=[A_{ij}(z)]_{i,j=1,2}=A_0+z_1A_1+\cdots+z_nA_n,$ where $\det A_{22}(z)\not\equiv 0.$ Moreover, the following functional relations are true:
\begin{itemize}
\item[(a)] $f(\lambda z)=\lambda f(z)$ (i.e., $f$ is a homogeneous degree-one function) if and only if one can choose $A_0=0$.
\item[(b)] $f(z)=\overline {f(\overline{z})}$ if and only if one can choose $A_j=\overline{A_j}$, for all $j=0,\ldots n$.
\item[(c)] $f(z)=f(z)^T$ if and only if one can choose $A_j=A_j^T$, for all $j=0,\ldots n$.
\item[(d)] $f(z)=f(\overline{z})^*$ if and only if one can choose $A_j=A_j^*$, for all $j=0,\ldots n$.
\item[(e)] $f$ satisfies any combination of the (a)-(d) if and only if one can choose the $A_j$ to have simultaneously the associated properties.
\end{itemize}
\end{theorem}

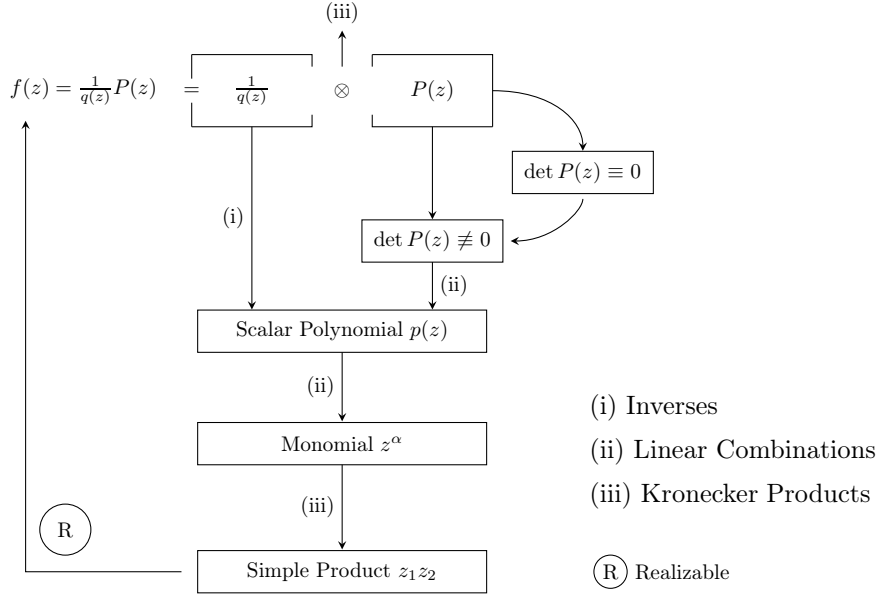
\begin{figure}[h]
\begin{center}
\begin{tikzpicture}[scale=0.8, transform shape,
blockone/.style ={rectangle, draw, text width=6em,align=center, minimum height=2em}%
,
blocktwo/.style ={rectangle, draw, text width=13em,align=center, minimum height=2em}%
,
R/.style ={draw, circle, inner sep=5pt},
Rk/.style ={draw, circle, inner sep=2pt}
]
\path(0,0) node(x) {$=$}
(-1.8,0) node[](fz) {$f(z) = \frac{1}{q(z)}P(z)$}
(1,0) node[very thick](q/z) {$\frac{1}{q(z)}$}
(2.5,0) node[](kp) {$\otimes$}
(4,0) node[](Pz) {$P(z)$};
\draw(2.5,1) coordinate (iii) node[above] {(iii)};
\draw[->,>=stealth] (2.5,.4) -- (2.5,1);
\draw(0.0,0.2) coordinate (1);
\draw(0.0,0.6) coordinate (2);
\draw(2,0.4) coordinate (3);
\draw(1) -- (2);
\draw(2) -| (3);
\draw(0.0,-0.2) coordinate (1);
\draw(0.0,-0.6) coordinate (2);
\draw(2,-0.4) coordinate (3);
\draw(1) -- (2);
\draw(2) -| (3);
\draw(3,0.4) coordinate (4);
\draw(3,0.6) coordinate (5);
\draw(5,-0.6) coordinate (6);
\draw(3,-0.4) coordinate (7);
\draw(4) -- (5);
\draw(5) -| (6);
\draw(6) -| (7);
\draw(1,-0.6) coordinate (8);
\draw(1,-3.64) coordinate (9);
\draw[->,>=stealth] (8) -- node[left] {(i)} (9);
\draw(6.5,-1.36) node[blockone] (detzero) {$\det P(z) \equiv0$};
\draw(4,-2.5) node[blockone] (detnzero) {$\det P(z) \not\equiv0$};
\draw[->,>=stealth] (5,0) .. controls (6,0) and (6.5,-0.5) .. (6.5,-1);
\draw
[->,>=stealth] (6.5,-1.8) .. controls (6.5,-2) and (6,-2.5) .. (5.3,-2.5);
\draw(4,-.6) coordinate (10);
\draw(4,-2.14) coordinate (11);
\draw[->,>=stealth] (10) -- (11);
\draw(4,-2.85) coordinate (12);
\draw(4,-3.64) coordinate (13);
\draw[->,>=stealth] (12) --  node[right] {(ii)} (13);
\draw(2.5,-4) node[blocktwo] (sp) {Scalar Polynomial $p(z)$};
\draw(2.5,-4.35) coordinate (14);
\draw(2.5,-5.5) coordinate (15);
\draw[->,>=stealth] (14) -- node[left] {(ii)} (15);
\draw(2.5,-5.87) node[blocktwo] (m) {Monomial $z^{\alpha}$};
\draw(2.5,-6.22) coordinate (16);
\draw(2.5,-7.63) coordinate (17);
\draw[->,>=stealth] (16) -- node[left] {(iii)} (17);
\draw(2.5,-8) node[blocktwo] (z1z2) {Simple Product $z_1z_2$};

\draw(-2.76,-.5) coordinate (18);
\draw(-.17,-8) coordinate (19);
\draw[->,>=stealth] (19) -| (18);
\draw(-2.1,-7.3) node[R] (R) {R};
\matrix[right] at (6,-6) {
\node[label=right:(i) Inverses] {}; \\
\node[label=right:(ii) Linear Combinations] {}; \\
\node[label=right:(iii) Kronecker Products] {}; \\
};
\draw(6.96,-8.0) node[Rk] [label=right: Realizable] {R};
\end{tikzpicture}
\caption{\;\;Flow diagram for the proof of the Bessmertny\u{\i}
Realization Theorem. \qquad\enspace\enspace\qquad\qquad}\label
{fig-FlowDiagPrfThmBessmRealiz}
\end{center}
\end{figure}

\begin{proof}\label{ProofMainResultBessmRealiz}
The theorem will be proved in a series of steps that reduce the complexity of the problem into simpler realization problems, as is illustrated in the flow diagram of Figure \ref{fig-FlowDiagPrfThmBessmRealiz}. Let $f(z)=f(z_1,\ldots, z_n)$ be a rational $\mathbb{C}^{k\times k}$-valued matrix function of $n$ complex variables $z_1,\ldots,z_n$. Then there exists a nonzero scalar polynomial $q(z)$ and a polynomial $\mathbb{C}^{k\times k}$-valued matrix function $P(z)$ of these $n$-variables such that \[
f\left(z\right)  =\frac{1}{q\left(z\right)} P\left(z\right).
\] Using the Kronecker product $\otimes$ [and treating $q(z)$ as a  polynomial $\mathbb{C}^{1\times 1}$-valued matrix function], we can rewrite this as \[
f\left(z\right)  =\frac{1}{q\left(z\right)}\otimes P\left(z\right).
\] Before we can proceed further, there are two cases we must consider corresponding to whether $\det P\left(  z\right)  \not \equiv 0$ or not.

First, consider the case that $\det P\left(  z\right)  \not \equiv 0.$ By Proposition \ref{PropRealizOfKroneckerProdOfRealiz} (as well as Lemmas \ref{LemRealizOfKroneckerProdsPart1} and \ref{LemRealizOfKroneckerProdsPart2}), we can realize $\frac
{1}{q\left(z\right)}\otimes P\left(z\right)$ if both $\frac
{1}{q\left(z\right)}$ and $P\left(z\right)$ are realizable. By Proposition \ref{PropInvOfASchurCompl}, we can realize $\frac
{1}{q\left(z\right)}$ if $q\left(z\right)$ is realizable.

Second, consider the case $\det P\left(  z\right) \equiv 0$. As the theorem is obviously true if $ P\left(z\right) \equiv 0$, we may assume $P\left(  z\right)  \not \equiv 0$. Then there exists $z_0\in\mathbb{C}^n$ such that $P\left(  z_{0}\right)\not=0$. Fix a nonzero real number $\lambda_{0}$ that is not an eigenvalue of $P\left(  z_{0}\right)$ and consider the two matrix polynomials $P_1\left(z\right)
=P\left(  z\right)  +\lambda_{0}I_{k}-P\left(  z_{0}\right)  $ and $P_2\left(z\right)
=P\left(  z_{0}\right)-\lambda_{0}I_{k}  $. They both satisfy $\det P_j\left(  z\right)  \not \equiv 0,$ for $j=1,2$ and $f(z)=\frac{1}{q\left(  z\right)  }P\left(  z\right)=\frac{1}{q\left(  z\right)  }P_1\left(  z\right)+\frac{1}{q\left(  z\right)  }P_2\left(  z\right)$. Hence, by Lemma \ref{LemSumSchurComplPlusMatrix} and Proposition \ref{PropSumOfSchurComps}, $f(z)$ is realizable if both $\frac{1}{q\left(  z\right)  }P_1\left(  z\right)$ and $\frac{1}{q\left(  z\right)  }P_2\left(  z\right)$ are realizable. Thus, we are back to the first case again.

From considering both of the cases above, it now becomes clear that we just need to be able to realize any arbitrary matrix polynomial $P\left(z\right)$ and scalar polynomial $q\left(z\right)$. We will begin by investigating the realizability of the former and show it reduces to the realizability of the latter.

Suppose $P(z)$ is a polynomial $\mathbb{C}^{k\times k}$-valued matrix function of the $n$ complex variables $z_1,\ldots,z_n$. Then, there exists scalar polynomial
functions $P_{ij}\left(  z\right)  $, for $i$, $j=1,\ldots,k$ such that%
\[
P\left(  z\right)  =\left[  P_{ij}\left(  z\right)  \right]  _{i,j=1}^{k}=
{\textstyle\sum\nolimits_{i=1}^{k}}
{\textstyle\sum\nolimits_{j=1}^{k}}
P_{ij}\left(  z\right)  E_{ij},
\]
where $E_{ij}$, for $i$, $j=1,\ldots,k$, are the standard basis vectors for $
\mathbb{C}^{k\times k}$ (i.e., $E_{ij}$ is the $k\times k$ matrix whose entry in $i$th row, $j$th column is $1$ and the remaining entries are all $0$). 

Therefore, by Lemma \ref{LemSumSchurComplPlusMatrix}, Proposition \ref{PropSumOfSchurComps}, and Proposition \ref{PropScalarProdOfASchurCompl}, $P\left(
z\right)$ is realizable if each scalar polynomial functions
$P_{ij}\left(  z\right)  $, for $i$, $j=1,\ldots,k$ are realizable.
Thus, we have reduced our problem to realizing an arbitrary scalar polynomial $p\left(z\right)$.

Suppose that $p\left( z\right)$ is an arbitrary scalar polynomial of the $n$ complex variables $z_1,\ldots,z_n$ [e.g., $q(z)$ or one of the $P_{ij}(z)$ above]. Then $p\left(z\right)$ can be written uniquely as a linear combination of monomials,
\[
p\left(  z\right)  =
{\textstyle\sum\nolimits_{j=0}^{m}}
a_{j}z^{\alpha_{j}},
\] where $a_j$ are scalars and $z^{\alpha_{j}}$ are monomials, for $j=0,\ldots,m$. Hence, by Lemma \ref{LemScalarMultiSchurCompl}, Lemma \ref{LemSumSchurComplPlusMatrix}, and Proposition \ref{PropSumOfSchurComps}, it follows that $p\left(z\right)$ is realizable if each monomial
$z^{\alpha_{j}}$ is realizable. Thus, we have reduced our problem
to realizing an arbitrary monomial $z^{\alpha}$.

Suppose that $z^{\alpha}$ is a monomial. Then it is realizable by Proposition \ref{PropRealizMonomials}. Let us explain the reason why. The monomial can be written uniquely as products of powers of the independent variables, $z^{\alpha}=\prod_{j=1}^{n}z_j^{\alpha_j}$. As this can be written as the Kronecker product $\prod_{j=1}^{n}z_j^{\alpha_j}=z_1^{\alpha_1}\otimes\cdots\otimes z_n^{\alpha_n}$, then by Proposition \ref{PropRealizOfKroneckerProdOfRealiz} we can realize the monomial $z^{\alpha}$ if we can realize the product $w_{1}w_{2}$ of two independent complex variables $w_1$ and $w_2$, which we can by Lemma \ref{LemRealizProdTwoIndepVar}.

Therefore, we have proven that the rational $\mathbb{C}^{k\times k}$-valued matrix function $f(z)=f(z_1,\ldots, z_n)$ of the $n$ complex variables $z_1,\ldots,z_n$ is realizable.

In the second part of this theorem, we will prove statements a)-e) are true for the rational function $f(z)$. We will achieve this by modifying the proof of the first part of the theorem above, when and where necessary, for each statements (a)-(e). 

First, we will prove statement (a). Suppose that $f(\lambda z)=\lambda f(z)$, i.e., $f(z)$ is also a homogeneous function of degree one. Then
\[
f(z)=z_1f\left(\frac{z}{z_1}\right)=z_1f\left(1,\frac{z_2}{z_1},\ldots,\frac{z_n}{z_1}\right).
\] 
Hence, 
\[
g(w)=f(1,w_2,\ldots,w_{n}),\;\; w=(w_2,\ldots,w_n),
\]
satisfies the hypotheses of the first part of the theorem, and has a realization
\[
g(w)=B(w)/B_{22}(w),\;\;B(w)=A_1+w_2A_2+\cdots+w_nA_n,
\]
implying $f(z)$ has the realization
\[
f(z)=A(z)/A_{22}(z),\;\;A(z)=z_1A_1+z_2A_2+\cdots+z_nA_n,
\]
since
\[
A(z)=z_1A\left(\frac{z}{z_1}\right)=z_1B\left(\frac{z_2}{z_1},\ldots,\frac{z_n}{z_1}\right)
\]
and, by property (\ref{ElemPropSchurComplementHomogeneousDegree1}),
\begin{gather*}
f(z)=z_1f\left(\frac{z}{z_1}\right)=z_1g\left(\frac{z_2}{z_1},\ldots,\frac{z_n}{z_1}\right)\\
=z_1\left[B\left(\frac{z_2}{z_1},\ldots,\frac{z_n}{z_1}\right)/B_{22}\left(\frac{z_2}{z_1},\ldots,\frac{z_n}{z_1}\right)\right]\\
=\left[z_1B\left(\frac{z_2}{z_1},\ldots,\frac{z_n}{z_1}\right)\right]/\left[z_1B_{22}\left(\frac{z_2}{z_1},\ldots,\frac{z_n}{z_1}\right)\right]\\
=\left[z_1A\left(\frac{z}{z_1}\right)\right]/\left[z_1A_{22}\left(\frac{z}{z_1}\right)\right]\\
=A(z)/A_{22}(z).
\end{gather*}
Conversely, suppose that $f(z)$ has a realization \[
f(z)=A(z)/A_{22}(z),\;\;A(z)=z_1A_1+z_2A_2+\cdots+z_nA_n.
\]
Then, since
\[
A(\lambda z)=\lambda A(z),
\] 
it follows by property (\ref{ElemPropSchurComplementHomogeneousDegree1}) that
\[
f(\lambda z)=A(\lambda z)/A_{22}(\lambda z)=\left[\lambda A(z)\right]/\left[\lambda A_{22}(z)\right]=\lambda\left[A(z)/A_{22}(z)\right]=\lambda f(z).
\]
Therefore, statement (a) is true.

Next, we will prove statement (b). Suppose that $f(z)=\overline {f(\overline{z})}$, [i.e., $f(z)$ is a $k\times k$ matrix whose entries are real rational scalar functions of $z$]. Then in the proof above in which we constructed a realization for $f(z)$ from the factorization $f(z)=\frac{1}{q(z)}P(z),$ we may assume that the nonzero scalar polynomial $q(z)$ is a real polynomial and the polynomial $\mathbb{C}^{k\times k}$-valued matrix function $P(z)$ is a real matrix polynomial [i.e.,  $P(z)$ is a $k\times k$ matrix whose entries are real polynomial scalar functions]. In this case, the proof of the realization of such an input $f(z)=\frac{1}{q(z)}P(z)$, as shown by the flow diagram in Fig.\ \ref{fig-FlowDiagPrfThmBessmRealiz}, would output a realization of $f(z)$ with real matrices (i.e., a Bessmertny\u{\i} realization in which the matrices in the linear matrix pencil are all real matrices). The converse of statement (b) is obviously true by property (\ref{ElemPropSchurComplementConjugationPreserving}). Therefore, we have proven statement (b).

Next, we will prove statement (c). Suppose that $f(z)=f(z)^T$. Then in the proof above, in which we constructed a realization for $f(z)$ from the factorization $f(z)=\frac{1}{q(z)}P(z),$ we may assume that the polynomial $\mathbb{C}^{k\times k}$-valued matrix function $P(z)$ satisfies $P(z)=P(z)^T$. In this case, the proof of the realization of such an input $f(z)=\frac{1}{q(z)}P(z)$, as shown by the flow diagram in Fig.\ \ref{fig-FlowDiagPrfThmBessmRealiz}, would output a symmetric realization of $f(z)$ (i.e., a Bessmertny\u{\i} realization in which the matrices in the linear matrix pencil are all symmetric matrices) provided we can prove that $P(z)$ has a symmetric realization. To prove this, we need only make one slight modification of our proof using the fact that since $P(z)=P(z)^T$ then $P_{ji}(z)=P_{ij}(z)$ for all $i,j=1,\ldots, k$ and $P_{ii}(z)E_{ii}$ and $P_{ij}(z)(E_{ij}+E_{ji})$ are symmetric for all $i,j=1,\ldots, k$ so that by Lemma \ref{LemScalarMultiSchurCompl}, Lemma \ref{LemSumSchurComplPlusMatrix}, Proposition \ref{PropSumOfSchurComps}, Proposition \ref{PropScalarProdOfASchurCompl}, and Proposition \ref{PropRealizMonomials} they have symmetric realizations which implies by Lemma \ref{LemSumSchurComplPlusMatrix} and Proposition \ref{PropSumOfSchurComps} that their sum
\begin{align*}
{\textstyle\sum\nolimits_{i=1}^{k}}P_{ii}\left(z\right)E_{ii}+{\textstyle\sum\nolimits_{1\leq i<j\leq k}}P_{ij}\left(z\right)\left(E_{ij}+E_{ji}\right) & = {\textstyle\sum\nolimits_{i=1}^{k}}
    {\textstyle\sum\nolimits_{j=1}^{k}}
    P_{ij}\left(z\right)  E_{ij} \\
    & = P\left(z\right),
\end{align*}
has a symmetric realization. The converse of statement (c) is obviously true by property (\ref{ElemPropSchurComplementTransposePreserving}). Therefore, we have proven statement (c).

Next, we will prove statement (d). Suppose that $f(z)=f(\overline{z})^*$. Then in the proof above, in which we constructed a realization for $f(z)$ from the factorization $f(z)=\frac{1}{q(z)}P(z),$ we may assume that the nonzero scalar polynomial $q(z)$ is a real polynomial [i.e., $q(z)=\overline{q(\overline{z})}$] and the polynomial $\mathbb{C}^{k\times k}$-valued matrix function $P(z)$ satisfies $P\left(  z\right)  =P\left(  \overline{z}\right)  ^{\ast}$. In this case, the proof of the realization of such an input $f(z)=\frac{1}{q(z)}P(z)$, as shown by the flow diagram in Fig.\ \ref{fig-FlowDiagPrfThmBessmRealiz}, would output a Hermitian realization of $f(z)$ (i.e., a Bessmertny\u{\i} realization in which the matrices in the linear matrix pencil are all Hermitian matrices) provided we can prove that $P(z)$ has a Hermitian realization. To prove this, we need only make one slight modification to our proof of part (c). We separate $P(z)$ into its symmetric $Q_s(z)$ and skew-symmetric $Q_a(z)$ parts, i.e.,
\begin{align*}
P\left(z\right)   &=
{\textstyle\sum\nolimits_{i=1}^{k}}
{\textstyle\sum\nolimits_{j=1}^{k}}
P_{ij}\left(z\right)  E_{ij}\\
&= Q_{s}\left(z\right)  + Q_{a}\left(  z\right),
\end{align*}
where
\begin{align*}
Q_{s}\left(  z\right)  & =
{\textstyle\sum\nolimits_{i=1}^{k}}
P_{ii}\left(  z\right)  E_{ii}+
{\textstyle\sum\nolimits_{1\leq i<j\leq k}}
\left[  \frac{P_{ij}\left(  z\right)  +\overline{P_{ij}\left(  \overline
{z}\right)  }}{2}\right]  \left(  E_{ij}+E_{ji}\right)  ,\\
Q_{a}\left(  z\right)  
& =
{\textstyle\sum\nolimits_{1\leq i<j\leq k}}
\left[  \frac{P_{ij}\left(  z\right)  -\overline{P_{ij}\left(  \overline
{z}\right)  }}{2}\right]  \left(  E_{ij}-E_{ji}\right) \\
&  =
{\textstyle\sum\nolimits_{1\leq i<j\leq k}}
\left[  \frac{P_{ij}\left(  z\right)  -\overline{P_{ij}\left(  \overline
{z}\right)  }}{2i}\right]  \left[  i\left(  E_{ij}-E_{ji}\right)  \right].
\end{align*}
Notice that for all $i,j=1,\ldots,k,$ the scalar polynomials 
\begin{align*}
\frac{P_{ij}\left(  z\right)  +\overline{P_{ij}\left(  \overline
{z}\right)  }}{2},\;\;\frac{P_{ij}\left(  z\right)  -\overline{P_{ij}\left(  \overline
{z}\right)  }}{2i}
\end{align*}
are all real polynomials, the matrices 
\begin{align*}
E_{ij}+E_{ji}
\end{align*}
are all real and symmetric (hence Hermitian), and the matrices
\begin{align*}
i\left(  E_{ij}-E_{ji}\right)
\end{align*}
are all Hermitian. Thus, it follows by Proposition \ref{PropScalarProdOfASchurCompl} that for any real scalar polynomial $p(z)$, if $B$ is real and symmetric then $p(z)B$ has a real symmetric realization (i.e., a Bessmertny\u{\i} realization in which each matrix in the linear matrix pencil is a real and symmetric matrix) and, if instead $B$ is a Hermitian matrix then $p(z)B$ has a Hermitian realization. From these facts and Lemma \ref{LemSumSchurComplPlusMatrix} and Proposition \ref{PropSumOfSchurComps} on realizations of sums, it follows that $Q_s(z)$ has a real symmetric realization (which is a Hermitian realization) and $Q_a(z)$ has a Hermitian realization, and thus, Lemma \ref{LemSumSchurComplPlusMatrix} and Proposition \ref{PropSumOfSchurComps} implies their sum $Q_s(z)+Q_a(z)=P(z)$ has a Hermitian realization. The converse of statement (d) is obviously true by property (\ref{ElemPropSchurComplementConjugationTransposePreserving}). This proves statement (d).

Finally, we will prove statement (e). Suppose $f(z)$ has any combination of two of the functional properties in (b), (c), or (d). Then $f(z)$ must satisfy $f(z)=f(z)^T$ and $f\left(  z\right)  =f\left(  \overline{z}\right)  ^{\ast}$ and hence we can proceed as in the proof of (d), in which case this we can assume that the nonzero scalar polynomial $q(z)$ is a real polynomial and $P\left(  \overline{z}\right)  ^{\ast}=P(z)^T=P(z)=Q_s(z)+Q_a(z)$ implying $Q_a(z)$ is the zero matrix and hence $P(z)=Q_s(z)$ has a real symmetric realization from which we conclude that in the proof of the realization of such an input $f(z)=\frac{1}{q(z)}P(z)$, as shown by the flow diagram in Fig.\ \ref{fig-FlowDiagPrfThmBessmRealiz}, would output a real symmetric realization of $f(z)$ which is automatically also a Hermitian realization.  Now suppose that $f(z)$ has any combination of functional properties in (a)-(d). To complete the proof of statement (e), we need only prove the statement now in the case one of these functional properties is (a) [which we do by slightly modifying the proof of statement (a)]. By our proof of (a), it follows that the function $g(w)=f(1,w_2,\ldots,w_n)$ inherits the same combination of functional properties (b)-(d) that $f(z)$ has. From our proof of statements (b)-(d) and the first part of our proof of (e) above, it follows that $g(w)$ has a real realization if (b) is true, a symmetric realization if (c) is true, a Hermitian realization if (d) is true, and a real symmetric realization if it has any combination of two of the functional properties in (b), (c), or (d). From this and the proof of statement (a) using such a realization for $g(w)$ as the choice of the linear matrix pencil $B(w)=A_1+w_2A_1+\cdots+ w_nA_n$ in the proof of (a), it follows that $f(z)$ can be realized with the linear matrix pencil $A(z)= z_1A_1+z_2A_1+\cdots+ z_nA_n$ which has the desired properties. The converse of statement (e) is obviously true by the elementary properties (\ref{ElemPropSchurComplementHomogeneousDegree1})-(\ref{ElemPropSchurComplementConjugationTransposePreserving}) of Schur complements. This proves statement (e) and completes the proof of the theorem.
\end{proof}

\begin{remark}\label{RemAltApproach}
Before we move on to examples of our approach to the Bessmertny\u{\i} realization theorem, we want to point out another application using our results in Sec. \ref{SecSchurComplAlgebraAndOps} that gives an alternative, more conventional approach to the proof of the first part of Theorem \ref{ThmBessmRealiz}, i.e., without symmetry considerations, using matrix products instead of Kronecker products (which the reviewer kindly outlined). To do this, we would do the following steps [in contrast to our steps (i)-(v) in the introduction] to give a Bessmertny\u{\i} realization of a $k\times k$ rational matrix function $f(z)$:
\begin{itemize}
    \item[(i)] The degree-1 scalar monomials $z_j\; (j = 1,\ldots ,n)$ are realizable (treated as linear $1\times 1$ matrix pencils).
    \item[(ii)] A scalar multiple of a realizable function is realizable (by Lemma \ref{LemScalarMultiSchurCompl}).
    \item[(iii)] Sums of realizable rational matrix functions (of fixed square size) are realizable (by Proposition \ref{PropSumOfSchurComps}).
    \item[(iv)] Matrix products of realizable functions are realizable. Hence, using the above steps (i)-(iv), scalar polynomials are realizable. Next, show the product $q(z)I$ is realizable (note this is not a matrix product), where $q(z)$ is a scalar polynomial and $I$ is any size identity matrix (here we would realize it using Proposition \ref{PropScalarProdOfASchurCompl} since $q(z)I=[q(z)]\otimes I$ and the scalar polynomial $q(z)$ is realizable). Then it follows from this and previous steps (i)-(iv) that matrix polynomials are realizable.
    \item[(v)] If $q(z)$ is a scalar polynomial not identically equal to zero, then $q(z)^{-1}$ is realizable. More generally, if $Q(z)$ is a matrix polynomial with $\det Q(z)\not \equiv 0$, then $Q(z)^{-1}$ is realizable (by Proposition \ref{PropInvOfASchurCompl}).
\item[(vi)] Write $f(z)= P(z)(q(z)I_k)^{-1}$, where $P(z)$ is a $k\times k$ polynomial matrix function and $q(z)$ is a scalar polynomial function not identically zero. Use the realization of the matrix polynomials $q(z)I_k$ and $P(z)$ [by (iv)] and then the realization of $[q(z)I_k]^{-1}$ [by (v)] to get a realization for their matrix product
$f(z)= P(z)(q(z)I_k)^{-1}$ by (iv).
\end{itemize}

Our results in Sec. \ref{SecSchurComplAlgebraAndOps} has been developed to treat all the steps for the alternate proof above with the exception of step (iv) on the matrix product of realizable functions is realizable. The only problem here is that Proposition \ref{PropMatrixMultipliationOfTwoSchurComplements} (Matrix multiplication of two Schur complements) needs to be further developed since the product of linear matrix pencils of the same size is in general a (multivariate) quadratic matrix pencil.

We can overcome this problem using our results of Sec. \ref{SecSchurComplAlgebraAndOps} by proceeding in a similar manner as we did for Kronecker products in Subsection \ref{SecOnKroneckerProductsOfLinearMatrixPencils} (cf. Figure \ref{FigFlowDiagramProofKroneckerProdRealiz}): Use Lemma \ref{LemRealizOfSquares} and Lemma \ref{LemRealizProdTwoIndepVar} together with Proposition \ref{PropScalarProdOfASchurCompl} [to realize simple products $x^2B$ and $xyB$ with $x,y$ two independent complex variables and $B$ any square complex matrix; alternatively, once you realize $x^2I$ and $xyI$ with $I$ the identity matrix of the same size as $B$ then we can realize the matrix products $x^2B=(x^2I)B$ and $xyB=(xyI)B$ by Proposition \ref{PropMatrixMultipliationOfTwoSchurComplements} since were not treating symmetries in this remark] followed by Lemma \ref{LemSumSchurComplPlusMatrix} and Proposition \ref{PropSumOfSchurComps} (to realize any quadratic linear matrix pencil), and Proposition \ref{PropComposSchurComplem} (to realize the Schur complement of a quadratic matrix pencil) and then this together with Proposition \ref{PropMatrixMultipliationOfTwoSchurComplements} we prove the desired result -- matrix products of realizable functions is realizable.
\end{remark}

\subsection{Examples}\label{SecExamplesOfBessRealizThm}

\begin{example}\label{Example1st}
To illustrate our approach of the Bessmertny\u{\i} Realization Theorem in the case in which the hypotheses of statements (b), (c), and (e) apply, we will work out the realization of the following rational $\mathbb{C}^{1\times 1}$-valued function of $2$-variables
\begin{align*}
f(z)=\begin{bmatrix}
\frac{z_2}{z_1}
\end{bmatrix}.
\end{align*}
As a first step, we write this in the form of a Kronecker product of matrices
$$f(z)=\frac{1}{q(z)}P(z)=\frac{1}{q(z)}\otimes P(z)=[q(z)]^{-1}\otimes P(z),$$ 
where
$$q(z)=z_1,\;\;P(z)=\begin{bmatrix}
z_2
\end{bmatrix}.$$
Next, we have $\det P(z)=z_2\not \equiv 0$ and $P(z)$ is already in the desired realized form. The next step is to realize $q(z)$, but in this case its already in the desired realized form, so we can realize its inverse,
\begin{eqnarray*}
\frac{1}{q\left( z\right) } =\left[ q\left( z\right) \right] ^{-1}=\left[z_1\right] ^{-1}
=
\left.
\left[\begin{array}{c; {2pt/2pt} c }
0 & 1 \\\hdashline[2pt/2pt]
1 & -z_1 
\end{array}\right]
\right/
\begin{bmatrix}
-z_1
\end{bmatrix}.
\end{eqnarray*}
Finally, we complete this part of the example by realizing the Kronecker product of realizations 
\begin{align*}
f(z)&=[q(z)]^{-1}\otimes P(z)=\left(\left.
\left[\begin{array}{c; {2pt/2pt} c }
0 & 1 \\\hdashline[2pt/2pt]
1 & -z_1
\end{array}\right]
\right/
\begin{bmatrix}
-z_1
\end{bmatrix}\right)
\otimes
[z_2]\\
&=A(z)/A_{22}(z),
\end{align*}
in which $$A(z)=A_0+z_1A_1+z_2A_2$$ is a linear matrix pencil such that the matrices $A_{j}\in \mathbb{C}^{m\times m}$, for some positive integer $m$ (in this example we will have $m=4$), are real and symmetric for $j=0,1,2$. To compute this pencil, we follow Lemma \ref{LemRealizOfKroneckerProdsPart2} and its proof. First, by Lemma \ref{LemKroneckerProdSchurComplWithAMatrix},
\begin{align*}
\left(\left.
\left[\begin{array}{c; {2pt/2pt} c }
0 & 1 \\\hdashline[2pt/2pt]
1 & -z_1
\end{array}\right]
\right/
\begin{bmatrix}
-z_1
\end{bmatrix}\right)
\otimes
[z_2]=
\left.
\left[\begin{array}{c; {2pt/2pt} c }
[0]\otimes [z_2] & [1]\otimes [z_2] \\\hdashline[2pt/2pt]
[1]\otimes [z_2] & [-z_1]\otimes [z_2]
\end{array}\right]
\right/
([-z_1]\otimes [z_2]).
\end{align*}
Second, we compute
\begin{gather*}
\left[\begin{array}{c; {2pt/2pt} c }
[0]\otimes [z_2] & [1]\otimes [z_2] \\\hdashline[2pt/2pt]
[1]\otimes [z_2] & [-z_1]\otimes [z_2]
\end{array}\right]=\left[\begin{array}{c; {2pt/2pt} c }
0 & z_2 \\\hdashline[2pt/2pt]
z_2 & -z_1z_2
\end{array}\right]\\
=
z_2\left[\begin{array}{c; {2pt/2pt} c }
0 & 1 \\\hdashline[2pt/2pt]
1 & 0
\end{array}\right]
+(z_1z_2)\left[\begin{array}{c; {2pt/2pt} c }
0 & 0 \\\hdashline[2pt/2pt]
0 & -1
\end{array}\right].
\end{gather*}
Third, by Lemma \ref{LemShortedMatricesAreSchurCompl} and Lemma \ref{LemRealizProdTwoIndepVar},
\begin{gather*}
(z_1z_2)\left[\begin{array}{c; {2pt/2pt} c }
0 & 0 \\\hdashline[2pt/2pt]
0 & -1
\end{array}\right]
=
\left[\begin{array}{c; {2pt/2pt} c }
[0] & [0] \\\hdashline[2pt/2pt]
[0] & [-z_1z_2]
\end{array}\right]
\\
=\left.
\left[\begin{array}{c c; {2pt/2pt} c c}
0 & 0 & 0 & 0 \\
0 & 0 & -\frac{1}{4}(z_1+z_2) & \frac{1}{4}(z_1-z_2)\\ \hdashline[2pt/2pt]
0 & -\frac{1}{4}(z_1+z_2) & \frac{1}{4} & 0\\
0 & \frac{1}{4}(z_1-z_2) & 0 & -\frac{1}{4}
\end{array}\right]
\right/
\begin{bmatrix}
\frac{1}{4} & 0 \\
0 & -\frac{1}{4}
\end{bmatrix}
\end{gather*}
Fourth, by Lemma \ref{LemSumSchurComplPlusMatrix},
\begin{gather*}
z_2\left[\begin{array}{c; {2pt/2pt} c }
0 & 1 \\\hdashline[2pt/2pt]
1 & 0
\end{array}\right]
+(z_1z_2)\left[\begin{array}{c; {2pt/2pt} c }
0 & 0 \\\hdashline[2pt/2pt]
0 & -1
\end{array}\right]\\
=
\left.
\left[\begin{array}{c c; {2pt/2pt} c c}
0 & z_2 & 0 & 0 \\
z_2 & 0 & -\frac{1}{4}(z_1+z_2) & \frac{1}{4}(z_1-z_2)\\ \hdashline[2pt/2pt]
0 & -\frac{1}{4}(z_1+z_2) & \frac{1}{4} & 0\\
0 & \frac{1}{4}(z_1-z_2) & 0 & -\frac{1}{4}
\end{array}\right]
\right/
\begin{bmatrix}
\frac{1}{4} & 0 \\
0 & -\frac{1}{4}
\end{bmatrix}.
\end{gather*}
Finally, we compute
\begin{gather*}
A(z)=A_0+z_1A_1+z_2A_2
=
\left[\begin{array}{c; {2pt/2pt} c c c}
0 & z_2 & 0 & 0 \\\hdashline[2pt/2pt]
z_2 & 0 & -\frac{1}{4}(z_1+z_2) & \frac{1}{4}(z_1-z_2)\\ 
0 & -\frac{1}{4}(z_1+z_2) & \frac{1}{4} & 0\\
0 & \frac{1}{4}(z_1-z_2) & 0 & -\frac{1}{4}
\end{array}\right],\\
A_0=\left[\begin{array}{c; {2pt/2pt} c c c}
0 & 0 & 0 & 0 \\\hdashline[2pt/2pt]
0 & 0 & 0 & 0\\ 
0 & 0 & \frac{1}{4} & 0\\
0 & 0 & 0 & -\frac{1}{4}
\end{array}\right],\;\;
A_1=\left[\begin{array}{c; {2pt/2pt} c c c}
0 & 0 & 0 & 0 \\\hdashline[2pt/2pt]
0 & 0 & -\frac{1}{4} & \frac{1}{4}\\ 
0 & -\frac{1}{4} & 0 & 0\\
0 & \frac{1}{4} & 0 & 0
\end{array}\right],\\
A_2=\left[\begin{array}{c; {2pt/2pt} c c c}
0 & 1 & 0 & 0 \\\hdashline[2pt/2pt]
1 & 0 & -\frac{1}{4} & -\frac{1}{4}\\ 
0 & -\frac{1}{4} & 0 & 0\\
0 & -\frac{1}{4} & 0 & 0
\end{array}\right].
\end{gather*}
\end{example}
\begin{example}\label{Example2nd}
To illustrate our approach to the Bessmertny\u{\i} Realization Theorem in the case in which the hypotheses of statements (a)-(c) and (e) apply, we will work out the realization of the following rational $\mathbb{C}^{1\times 1}$-valued function of $3$-variables
\begin{align*}
f(z)=\begin{bmatrix}
\frac{z_2z_3}{z_1}
\end{bmatrix}.
\end{align*}
As the function $f(z)$ is homogeneous degree one [i.e., $f(\lambda z)=\lambda f(z)$] then following the proof of part (a) we start by realizing the function:
\begin{align*}
g(w)=f(1,w_2,w_3)=\begin{bmatrix}
w_2w_3
\end{bmatrix}.
\end{align*}
This has the realization
\begin{gather*}
g(w)=B(w)/B_{22}(w),\\
B(w)=A_1+w_2A_2+w_3A_3=
\left[\begin{array}{c; {2pt/2pt} c c}
0 & \frac{1}{4}(w_2+w_3) & -\frac{1}{4}(w_2 - w_3) \vspace{0.1cm} \\ \hdashline[2pt/2pt] 
\frac{1}{4}(w_2+w_3) & -\frac{1}{4} & 0 \\ 
-\frac{1}{4}(w_2 - w_3) & 0 & \frac{1}{4}
\end{array}\right],\\
A_1=
\left[\begin{array}{c; {2pt/2pt} c c}
0 & 0 & 0\vspace{0.1cm} \\ \hdashline[2pt/2pt] 
0 & -\frac{1}{4} & 0 \\ 
0 & 0 & \frac{1}{4}
\end{array}\right],\;\;
A_2=
\left[\begin{array}{c; {2pt/2pt} c c}
0 & \frac{1}{4} & -\frac{1}{4} \vspace{0.1cm} \\ \hdashline[2pt/2pt] 
\frac{1}{4} & 0 & 0 \\ 
-\frac{1}{4} & 0 & 0
\end{array}\right],\;\;
A_3=
\left[\begin{array}{c; {2pt/2pt} c c}
0 & \frac{1}{4} & \frac{1}{4} \vspace{0.1cm} \\ \hdashline[2pt/2pt] 
\frac{1}{4} & 0 & 0 \\ 
\frac{1}{4} & 0 & 0
\end{array}\right].
\end{gather*}
Finally, since $f(z)=z_1g(\frac{z_2}{z_1},\frac{z_3}{z_1})$, we get the realization of $f(z)$ as
\begin{gather*}
f(z)=\begin{bmatrix}
\frac{z_2z_3}{z_1}
\end{bmatrix}=A(z)/A_{22}(z),\\
A(z)=z_1A_1+z_2A_2+z_3A_3=
\left[\begin{array}{c; {2pt/2pt} c c}
0 & \frac{1}{4}(z_2+z_3) & -\frac{1}{4}(z_2 - z_3) \vspace{0.1cm} \\ \hdashline[2pt/2pt] 
\frac{1}{4}(z_2+z_3) & -\frac{1}{4}z_1 & 0 \\ 
-\frac{1}{4}(z_2 - z_3) & 0 & \frac{1}{4}z_1
\end{array}\right].
\end{gather*}
\end{example}

\section{Schur Complements: Algebra and Operations}\label{SecSchurComplAlgebraAndOps}

The goal of this section is to show that elementary operations (whether algebraic like addition, functional like composition, or transformal like the principal pivot transform) when applied to Schur complements of block matrices will be equal to another Schur complement of a block matrix and we provide explicit formulas to compute the resulting block matrix (e.g., for sums of Schur complements, $A/A_{22}+B/B_{22}$, it is equal to a Schur complement $C/C_{22}$ and the formula for $C=[C_{ij}]_{i,j=1,2}$ is given in Proposition \ref{PropSumOfSchurComps}). And after this, use these results to give certain elementary realizations involving linear matrix pencils. 

The main objective in this regard is to prove a Schur complement formula exists and that the resulting block matrix produced inherits the desired functional symmetries. Below is a representative list of our Schur complement/realization algebra:

\begin{itemize}
           \item Scalar multiplication of a Schur complements (Lemma \ref{LemScalarMultiSchurCompl}): $$\lambda (A/A_{22});$$
           \item Sums of a Schur complements (Proposition \ref{PropSumOfSchurComps}): $$A/A_{22}+B/B_{22};$$
            \item Shorted matrices are Schur complements (Lemma \ref{LemShortedMatricesAreSchurCompl}): $$A/A_{22} \oplus 0_l;$$
           \item Direct sum of Schur complements (Proposition \ref{PropDirectSumSchurCompl}): $$A/A_{22}\oplus B/B_{22};$$
           \item Matrix multiplication of two Schur complements (Proposition \ref{PropMatrixMultipliationOfTwoSchurComplements}): $$(A/A_{22})(B/B_{22});$$
           \item Matrix product with a Schur complement (Proposition \ref{PropMatrixMultOfSchurCompl}): $$B(A/A_{22})C;$$
           \item Inverse of Schur complement (Proposition \ref{PropInvOfASchurCompl}): $$(A/A_{22})^{-1};$$
           \item Kronecker product of two Schur complements (Proposition \ref{PropKroneckerProdOfTwoSchurComplements}): $$(A/A_{22})\otimes (B/B_{22});$$
           \item Compositions of Schur complements (Proposition \ref{PropComposSchurComplem}): $$(A/A_{22})/(A/A_{22})_{22};$$
           \item Realization of a simple products (Lemma \ref{LemRealizProdTwoIndepVar}): $$z_1z_2;$$
           \item Kronecker products of realizations (Proposition \ref{PropRealizOfKroneckerProdOfRealiz}): $$A(z)/A_{22}(z)\otimes B(w)/B_{22}(w);$$
           \item Realizability of a monomial (Proposition \ref{PropRealizMonomials}): $$z^{\alpha};$$
            \item Principal pivot transform as a Schur complement (Proposition \ref{PropPPTIsASchurComplement}): $$\operatorname{ppt}_2(A).$$
\end{itemize}

Furthermore, in our approach to elementary operations and realizations, we provided the most basic building blocks for producing more complicated ones. For example, using Proposition \ref{PropSumOfSchurComps} and Lemma \ref{LemShortedMatricesAreSchurCompl} to prove Proposition \ref{PropDirectSumSchurCompl}. Or using Lemma \ref{LemInvIsASchurCompl} together with Proposition \ref{PropMatrixMultOfSchurCompl} to prove Proposition \ref{PropInvOfASchurCompl}. Another example of this is using Lemma \ref{LemKroneckerProdSchurComplWithAMatrix} to prove Corollary \ref{CorLemKroneckerProdSchurComplWithAMatrix} and then to use these together with Proposition \ref{PropComposSchurComplem} to prove Proposition \ref{PropKroneckerProdOfTwoSchurComplements}. Yet another example of this is using Lemma \ref{LemScalarMultiSchurCompl}, Proposition \ref{PropSumOfSchurComps}, and Lemma \ref{LemRealizOfSquares} to prove Lemma \ref{LemRealizProdTwoIndepVar}. This is especially evident in our proof of Proposition \ref{PropRealizOfKroneckerProdOfRealiz} which uses Lemma \ref{LemRealizOfKroneckerProdsPart1}, Lemma \ref{LemRealizOfKroneckerProdsPart2}, and other basic results above (see Figure \ref{FigFlowDiagramProofKroneckerProdRealiz}). This building block approach illustrates how one can attack problems by using our Schur complement algebra and operations in a more ``natural," algorithmic, and potentially computational way.

Moreover, we feel confident that the results in this section and our approach to them will find applications to other areas and problems (especially regarding realization and synthesis) in multidimensional systems theory especially for those linear models associated with electric circuits, networks, and composites. And because of this, we have also included a result on products of Schur complements, (i.e., $A/A_{22} B/B_{22}$), in Proposition \ref{PropMatrixMultipliationOfTwoSchurComplements} and, in Section \ref{SecTransformsForAlternativeRealizations}, results on using the other Schur complement $A/A_{11}$ instead of $A/A_{22}$ as well as the two associated principal pivot transforms, $\operatorname{ppt}_1(A)$ and $\operatorname{ppt}_2(A)$, respectively, in this context of realizability.

\subsection{Sums and scalar multiplication}

This first lemma belongs to the set of results relating to linear combinations involving Schur complements.\ And although elementary, it should give the reader a feel for the style of statements and proofs that we give in the remaining part of this paper which become progressively more difficult.
\begin{lemma}[Scalar multiplication of a Schur complement]\label{LemScalarMultiSchurCompl}
If $A\in
\mathbb{C}
^{m\times m}$ is a $2\times2$ block matrix
\[
A=
\begin{bmatrix}
A_{11} & A_{12}\\
A_{21} & A_{22}
\end{bmatrix},
\]
and $\lambda\in \mathbb{C}\setminus\{0\}$, then
\begin{align}
B/B_{22} = \lambda (A/A_{22}),
\end{align}
where $B\in\mathbb{C}^{m\times m}$ is the $2\times 2$ block matrix
\begin{gather}
B=\begin{bmatrix}
B_{11} & B_{12}\\
B_{21} & B_{22}
\end{bmatrix}=\begin{bmatrix}
\lambda A_{11} & \lambda A_{12}\\
\lambda A_{21} & \lambda A_{22}
\end{bmatrix}=\lambda A,
\end{gather}
and $B_{22}=\lambda A_{22}$ is invertible. Moreover, if $\lambda$ is real and the matrix $A$ is real, symmetric, Hermitian, or real and symmetric then the matrix $B$ is real, symmetric, Hermitian, or real and symmetric, respectively.
\end{lemma}
\begin{proof}
From the block matrix equality
\[
\begin{bmatrix}
(\lambda A)_{11} & (\lambda A)_{12}\\
(\lambda A)_{21} & (\lambda A)_{22}
\end{bmatrix} = \lambda A 
=
\begin{bmatrix}
\lambda (A_{11}) & \lambda (A_{12})\\
\lambda (A_{21}) & \lambda (A_{22})
\end{bmatrix},
\] it follows that
\begin{align*}
 \lambda (A/A_{22}) &= \lambda (A_{11} - A_{12}A_{22}^{-1}A_{21})\\
 &=\lambda (A_{11}) - \lambda (A_{12})[\lambda (A_{22})]^{-1}\lambda (A_{21})\\ &= (\lambda A)_{11} - (\lambda A)_{12}[(\lambda A)_{22}]^{-1}(\lambda A)_{21}\\ &=(\lambda A)/(\lambda A)_{22}.
\end{align*}
The remaining part of the proof is obvious. This proves the lemma.
\end{proof}

\begin{lemma}[Sum of a Schur complement with a matrix]\label{LemSumSchurComplPlusMatrix}
If $A\in
\mathbb{C}
^{m\times m}$ is a $2\times2$ block matrix
\[
A=
\begin{bmatrix}
A_{11} & A_{12}\\
A_{21} & A_{22}
\end{bmatrix},
\]
such that $A_{22}$ is invertible and $A/A_{22}\in
\mathbb{C}
^{k\times k}$ then, for any matrix $B\in
\mathbb{C}
^{k\times k}$,
\begin{align}\label{LemSumSchurComplPlusMatrixABMatrixFormula}
C/C_{22}=A/A_{22}+B,
\end{align}
where $C\in
\mathbb{C}
^{m \times m}$ is the $2\times2$ block
matrix
\begin{equation}
C=\begin{bmatrix}
C_{11} & C_{12}\\
C_{21} & C_{22}
\end{bmatrix}=
\begin{bmatrix}
A_{11}+B & A_{12}\\
A_{21} & A_{22}\\
\end{bmatrix},\label{LemSumSchurComplPlusMatrixCMatrixFormula}
\end{equation}
and $C_{22}=A_{22}$ is invertible. Moreover, if both matrices $A$ and $B$ are real, symmetric, Hermitian, or real and symmetric then the matrix $C$ is real, symmetric, Hermitian, or real and symmetric, respectively.
\end{lemma}
\begin{proof}
The proof is a straightforward calculation using block matrix techniques, to prove the formula  (\ref{LemSumSchurComplPlusMatrixABMatrixFormula}), we compute
\begin{align*}
A/A_{22}+B &=A_{11}+B-A_{12}A_{22}^{-1}A_{21}\\
&= C_{11}-C_{12}C_{22}^{-1}C_{21}\\
&=C/C_{22}.
\end{align*}
The remaining part of the proof follows immediately now from the formula (\ref{LemSumSchurComplPlusMatrixCMatrixFormula}) in terms of the matrices $A$ and $B$. This completes the proof.
\end{proof}

The following proposition is well known (see, for instance, \cite[Fig. 2, Eq. (11), Theorem 3]{63EG} and \cite[p. 1502, Theorem 4.2]{05BGM}), where it is often used in realizing the sum of transfer functions based on the analogy of the parallel connection of electrical networks.

\begin{proposition}[Sum of two Schur complements]\label{PropSumOfSchurComps}If $A\in
\mathbb{C}
^{m\times m}$ and $B\in
\mathbb{C}
^{n\times n}$ are $2\times2$ block matrices
\[
A=
\begin{bmatrix}
A_{11} & A_{12}\\
A_{21} & A_{22}
\end{bmatrix}
,\text{ }B=
\begin{bmatrix}
B_{11} & B_{12}\\
B_{21} & B_{22}
\end{bmatrix}
\]
such that $A_{22}\in
\mathbb{C}
^{p\times p}$, $B_{22}\in
\mathbb{C}
^{q\times q}$ are invertible and $A/A_{22}$, $B/B_{22}\in
\mathbb{C}
^{k\times k}$ then
\begin{align}\label{SumOfSchurComplementsAandB}
C/C_{22}=A/A_{22}+B/B_{22},
\end{align}
where $C\in
\mathbb{C}
^{\left(  k+p+q\right)  \times\left(  k+p+q\right)  }$ is the $3\times3$ block
matrix with the following block partitioned structure $C=[C_{ij}]_{i,j=1,2}$:
\begin{equation}\label{SumOfSchurCompsCMatrix}
C=
\left[\begin{array}{c;{2pt/2pt} c c}
C_{11} & C_{12} \\ \hdashline[2pt/2pt]
C_{21} & C_{22}
\end{array}\right]
=
\left[\begin{array}{c;{2pt/2pt} c c}
A_{11}+B_{11} & A_{12} & B_{12} \\ \hdashline[2pt/2pt]
A_{21} & A_{22} & 0\\
B_{21} & 0 & B_{22}
\end{array}\right],
\end{equation}
and
\begin{equation}\label{SumOfSchurCompsC22Matrix}
C_{22}=
\begin{bmatrix}
A_{22} & 0\\
0 & B_{22}
\end{bmatrix}
\end{equation}
is invertible. Moreover, if both matrices $A$ and $B$ are real, symmetric, Hermitian, or real and symmetric then the matrix $C$ is real, symmetric, Hermitian, or real and symmetric, respectively.
\end{proposition}
\begin{proof}
The proof is a straightforward calculation using block matrix techniques.
First, since $A\in
\mathbb{C}
^{m\times m}$, $B\in
\mathbb{C}
^{n\times n}$, $A_{22}\in
\mathbb{C}
^{p\times p}$, $B_{22}\in
\mathbb{C}
^{q\times q}$, $A/A_{22}$, $B/B_{22}\in
\mathbb{C}
^{k\times k}$ with $A_{22}$ and $B_{22}$ invertible then the $3\times 3$ block matrix $C$ defined in (\ref{SumOfSchurCompsCMatrix}) belongs to $\mathbb{C}
^{\left(  k+p+q\right)  \times\left(  k+p+q\right)  }$. Second, with its partitioned block structure $C=[C_{i,j}]_{i,j=1,2}$,
its $(2,2)$-block $C_{22}$ in (\ref{SumOfSchurCompsC22Matrix}), belongs to $\mathbb{C}^{\left(  p+q\right)  \times\left(  p+q\right)  }$ and is invertible with the inverse
\[
C_{22}^{-1}=
\begin{bmatrix}
A_{22}^{-1} & 0\\
0 & B_{22}^{-1}
\end{bmatrix}.
\]
Therefore, to prove the formula  (\ref{SumOfSchurComplementsAandB}) we compute
\begin{align*}
    A/A_{22}+B/B_{22}&=A_{11}-A_{12}A_{22}^{-1}A_{21}+B_{11}-B_{12}B_{22}^{-1}B_{21}\\&= A_{11}+B_{11}-\left(  A_{12}A_{22}^{-1}A_{21}+B_{12}B_{22}^{-1}%
B_{21}\right) \\
&=A_{11}+B_{11}-
\begin{bmatrix}
A_{12}A_{22}^{-1} & B_{12}B_{22}^{-1}
\end{bmatrix}
\begin{bmatrix}
A_{21}\\
B_{21}
\end{bmatrix}
\\ &=A_{11}+B_{11}-
\begin{bmatrix}
A_{12} & B_{12}
\end{bmatrix}
\begin{bmatrix}
A_{22}^{-1} & 0\\
0 & B_{22}^{-1}
\end{bmatrix}
\begin{bmatrix}
A_{21}\\
B_{21}
\end{bmatrix}
\\&=A_{11}+B_{11}-
\begin{bmatrix}
A_{12} & B_{12}
\end{bmatrix}
\begin{bmatrix}
A_{22} & 0\\
0 & B_{22}
\end{bmatrix}
^{-1}
\begin{bmatrix}
A_{21}\\
B_{21}
\end{bmatrix}
\\&=C_{11}-C_{12}C_{22}^{-1}C_{21} \\&=C/C_{22}.
\end{align*}
The remaining part of the proof follows immediately now from the formula (\ref{SumOfSchurCompsCMatrix}) in terms of the matrices $A$ and $B$. This completes the proof.
\end{proof}

The next lemma is interesting in its own right due to the importance of shorted matrices and operators both in electrical network theory and operator theory, see \cite{71WA, 74AT, 05TA, 15ACM, 47MK, 10MBM, 76NA, 14EP}. Furthermore, it is also an intermediate step in proving Proposition \ref{PropDirectSumSchurCompl} using Proposition \ref{PropSumOfSchurComps}.
\begin{lemma}[Shorted matrices are Schur complements]\label{LemShortedMatricesAreSchurCompl}
If $A\in
\mathbb{C}
^{m\times m}$ and $B\in
\mathbb{C}
^{n\times n}$ are $2\times2$ block matrices
\[
A=
\begin{bmatrix}
A_{11} & A_{12}\\
A_{21} & A_{22}
\end{bmatrix}
,\text{ }B=
\begin{bmatrix}
B_{11} & B_{12}\\
B_{21} & B_{22}
\end{bmatrix}
\]
such that $A_{22}\in
\mathbb{C}
^{p\times p}$, $B_{22}\in
\mathbb{C}
^{q\times q}$ are invertible and $A/A_{22}\in
\mathbb{C}
^{k\times k}$, $B/B_{22}\in
\mathbb{C}
^{l\times l}$ then the direct sums $A/A_{22} \oplus 0_l, 0_k \oplus B/B_{22}\in\mathbb{C}^{\left(  k+l\right)  \times\left(  k+l\right)}$ are Schur complements
\begin{align}\label{ShortedMatricesC}
C/C_{22}&=A/A_{22} \oplus 0_l=
\begin{bmatrix}
A/A_{22} & 0\\
0 & 0_l
\end{bmatrix},\\
\text{ }D/D_{22}&=0_k \oplus B/B_{22}=
\begin{bmatrix}\label{ShortedMatricesD}
0_k & 0\\
0 & B/B_{22}
\end{bmatrix},
\end{align}
where $C\in
\mathbb{C}
^{\left(  k+l+p\right)  \times\left(  k+l+p\right)  }, D\in
\mathbb{C}
^{\left(  k+l+q\right)  \times\left(  k+l+q\right)  }$ are $3\times3$ block matrices with the following block partitioned structure $C=[C_{ij}]_{i,j=1,2}, D=[D_{ij}]_{i,j=1,2}$: 
\begin{align} 
C & =
\left[\begin{array}{c;{2pt/2pt} c c}
C_{11} & C_{12} \\ \hdashline[2pt/2pt]
C_{21} & C_{22}
\end{array}\right]
=
\left[\begin{array}{c c; {2pt/2pt} c}
A_{11} & 0 & A_{12} \\
0 & 0_l & 0 \\ \hdashline[2pt/2pt]
A_{21} & 0 & A_{22}
\end{array}\right],\label{LemShortedMatricesAreSchurComplMatrixC}\\
D & =
\left[\begin{array}{c;{2pt/2pt} c c}
D_{11} & D_{12} \\ \hdashline[2pt/2pt]
D_{21} & D_{22}
\end{array}\right]
=
\left[\begin{array}{c c; {2pt/2pt} c}
0_k & 0 & 0 \\ 
0 & B_{11} & B_{12} \\ \hdashline[2pt/2pt]
0 & B_{21} & B_{22}
\end{array}\right],\label{LemShortedMatricesAreSchurComplMatrixD}
\end{align}
and $C_{22}=A_{22}$, $D_{22}=B_{22}$ are invertible. Moreover, if the matrix $A$ (the matrix $B$) is real, symmetric, Hermitian, or real and symmetric then the matrix $C$ (the matrix $D$) is real, symmetric, Hermitian, or real and symmetric, respectively.
\end{lemma}
\begin{proof}
The proof is again a straightforward calculation using block matrix techniques. First, from the definition of $C=[C_{ij}]_{i,j=1,2}$ in (\ref{LemShortedMatricesAreSchurComplMatrixC}) and starting from the right hand side of (\ref{ShortedMatricesC}) we compute,
\begin{align*}
    \begin{bmatrix}
A/A_{22} & 0\\
0 & 0
\end{bmatrix}&=\begin{bmatrix}
A_{11}-A_{12}A_{22}^{-1}A_{21} & 0\\
0 & 0
\end{bmatrix}\\
&=\begin{bmatrix}
A_{11} & 0\\
0 & 0
\end{bmatrix}
-
\begin{bmatrix}
A_{12}A_{22}^{-1}A_{21} & 0\\
0 & 0
\end{bmatrix}\\&= \begin{bmatrix}
A_{11} & 0\\
0 & 0
\end{bmatrix}
-
\begin{bmatrix}
A_{12}A_{22}^{-1}\\
0
\end{bmatrix}
\begin{bmatrix}
A_{21} & 0
\end{bmatrix}\\&=\begin{bmatrix}
A_{11} & 0\\
0 & 0
\end{bmatrix}
-
\begin{bmatrix}
A_{12}\\
0
\end{bmatrix}
A_{22}^{-1}
\begin{bmatrix}
A_{21} & 0
\end{bmatrix}\\&=C_{11}-C_{12}C_{22}^{-1}C_{21}\\&=C/C_{22}.
\end{align*}
Similarly, from the definition of $D=[D_{ij}]_{i,j=1,2}$ in (\ref{LemShortedMatricesAreSchurComplMatrixD}) and starting from the right hand side of (\ref{ShortedMatricesD}) we compute,
\begin{align*}
    \begin{bmatrix}
0 & 0\\
0 & B/B_{22}
\end{bmatrix}&=\begin{bmatrix}
0 & 0\\
0 & B_{11}-B_{12}B_{22}^{-1}B_{11}
\end{bmatrix}\\&=\begin{bmatrix}
0 & 0\\
0 & B_{11}
\end{bmatrix}
-
\begin{bmatrix}
0 & 0\\
0 & B_{12}B_{22}^{-1}B_{11}
\end{bmatrix}\\&=\begin{bmatrix}
0 & 0\\
0 & B_{11}
\end{bmatrix}
-
\begin{bmatrix}
0\\
B_{12}B_{22}^{-1}
\end{bmatrix}
\begin{bmatrix}
0 & B_{21}
\end{bmatrix}\\&=\begin{bmatrix}
0 & 0\\
0 & B_{11}
\end{bmatrix}
-
\begin{bmatrix}
0\\
B_{12}
\end{bmatrix}
B_{22}^{-1}
\begin{bmatrix}
0 & B_{21}
\end{bmatrix}\\&=D_{11}-D_{12}D_{22}^{-1}D_{21}\\&=D/D_{22}.
\end{align*}
The remaining part of the proof follows immediately now from the Schur complement of $C$ and $D$ in formulas (\ref{LemShortedMatricesAreSchurComplMatrixC}) and (\ref{LemShortedMatricesAreSchurComplMatrixD}) in terms of the matrices $A$ and $B$, respectively. This completes the proof.
\end{proof}

\begin{proposition}[Direct sum of Schur complements]\label{PropDirectSumSchurCompl}
If $A\in
\mathbb{C}
^{m\times m}$ and $B\in
\mathbb{C}
^{n\times n}$ are $2\times2$ block matrices
\[
A=
\begin{bmatrix}
A_{11} & A_{12}\\
A_{21} & A_{22}
\end{bmatrix}
,\text{ }B=
\begin{bmatrix}
B_{11} & B_{12}\\
B_{21} & B_{22}
\end{bmatrix}
\]
such that $A_{22}\in
\mathbb{C}
^{p\times p}$, $B_{22}\in
\mathbb{C}
^{q\times q}$ are invertible and $A/A_{22}\in
\mathbb{C}
^{k\times k}$, $B/B_{22}\in
\mathbb{C}
^{l\times l}$ then
\begin{align}
C/C_{22}=A/A_{22}\oplus B/B_{22}=
\begin{bmatrix}
A/A_{22} & 0\\
0 & B/B_{22}
\end{bmatrix}
\end{align}
where $C\in
\mathbb{C}
^{\left(  k+l+p+q\right)  \times\left(  k+l+p+q\right)  }$ is the $2\times2$
block matrix
\begin{align}
C=
\left[\begin{array}{c;{2pt/2pt} c c}
C_{11} & C_{12} \\ \hdashline[2pt/2pt]
C_{21} & C_{22}
\end{array}\right]
=
\left[\begin{array}{c;{2pt/2pt} c c}
A_{11}\oplus B_{11} & A_{12}\oplus B_{12} \\ \hdashline[2pt/2pt]
A_{21}\oplus B_{21}  & A_{22}\oplus B_{22}
\end{array}\right]
=
\left[\begin{array}{c c; {2pt/2pt} c c}
A_{11} & 0 & A_{12} & 0 \\ 
0 & B_{11} & 0 & B_{12} \\ \hdashline[2pt/2pt]
A_{21} & 0 & A_{22} & 0\\
0 & B_{21} & 0 & B_{22}
\end{array}\right]\label{PropDirectSumSchurComplCMatrix}
\end{align}
and $C_{22}=A_{22}\oplus B_{22}$ is invertible. Moreover, if both matrices $A$ and $B$ are real, symmetric, Hermitian, or real and symmetric then the matrix $C$ is real, symmetric, Hermitian, or real and symmetric, respectively.
\end{proposition}

\begin{proof}
The proof, as we shall see, follows immediately from Lemma \ref{LemShortedMatricesAreSchurCompl} using
Proposition \ref{PropSumOfSchurComps}. From the formulas
\[
A/A_{22}\oplus B/B_{22}=
\begin{bmatrix}
A/A_{22} & 0\\
0 & B/B_{22}
\end{bmatrix}
=
A/A_{22}\oplus 0_l + 0_k\oplus B/B_{22},
\]
and, by Lemma \ref{LemShortedMatricesAreSchurCompl},
\begin{align*}
A/A_{22}\oplus 0_l & = \left[\begin{array}{c c; {2pt/2pt} c}
A_{11} & 0 & A_{12} \\
0 & 0_l & 0 \\ \hdashline[2pt/2pt]
A_{21} & 0 & A_{22}
\end{array}\right]/A_{22},\\
0_k\oplus B/B_{22} & = \left[\begin{array}{c c; {2pt/2pt} c}
0_k & 0 & 0 \\ 
0 & B_{11} & B_{12} \\ \hdashline[2pt/2pt]
0 & B_{21} & B_{22}
\end{array}\right]/B_{22},
\end{align*}
it follows immediately from Proposition \ref{PropSumOfSchurComps} that
\begin{equation*}
C/C_{22}=A/A_{22}\oplus 0_l + 0_k\oplus B/B_{22},  
\end{equation*}
where $C$ is given by the formula (\ref{PropDirectSumSchurComplCMatrix}). The remaining part of the proof follows immediately now from the formula (\ref{PropDirectSumSchurComplCMatrix}) in terms of the matrices $A$ and $B$. This completes the proof.
\end{proof} 
\subsection{Matrix products and inverses}

\begin{remark}
To make the section as complete as possible in its treatment of the ``natural" algebra operation involving Schur complements, we include the next proposition on matrix multiplication of two Schur complements, i.e., $A/A_{22}B/B_{22}$. But it is an operation that we do not use at all in the Bessmertny\u{\i} realization theorem as it turns out that matrix multiplication it is not as ``natural" as the Kronencker product $\otimes$ of matrices is for solving the realization problem with symmetries as explained in Subsection \ref{SubsecRelevantWork}.
\end{remark}

The following proposition is well known (see, for instance, \cite[Fig. 3, Eq. (12), Theorem 4]{63EG}, \cite[p. 6]{79BGK}, \cite[Sec. 2.3]{08BGKR}, and \cite[pp. 1500, 1501, Theorem 4.1]{05BGM}), where it is often used in realizing the product of transfer functions based on the analogy of the cascade connection of electrical networks.
\begin{proposition}[Matrix multiplication of two Schur complements]\label{PropMatrixMultipliationOfTwoSchurComplements}
If $A\in
\mathbb{C}
^{m\times m}$ and $B\in
\mathbb{C}
^{n\times n}$ are $2\times2$ block matrices
\[
A=
\begin{bmatrix}
A_{11} & A_{12}\\
A_{21} & A_{22}
\end{bmatrix},\;
B=
\begin{bmatrix}
B_{11} & B_{12}\\
B_{21} & B_{22}
\end{bmatrix}
\]
such that the matrices $A_{22}$ and $B_{22}$ are invertible and $A/A_{22},B/B_{22}\in\mathbb{C}^{k\times k}$ then 
\begin{equation}\label{MatrixMultiplicationOfTwoSchurComplementsInProp}
C/C_{22}= A/A_{22}B/B_{22},
\end{equation}
where $C\in \mathbb{C}^{(m+n-k)\times (m+n-k)}$ is the $3\times 3$ block matrix with the following block partitioned structure $C=[C_{ij}]_{i,j=1,2}$:
\begin{align} 
C & =
\left[\begin{array}{c;{2pt/2pt} c c}
C_{11} & C_{12} \\ \hdashline[2pt/2pt]
C_{21} & C_{22}
\end{array}\right]
=
\left[\begin{array}{c; {2pt/2pt} c c}
A_{11}B_{11} & A_{12} & A_{11}B_{12} \\\hdashline[2pt/2pt]
A_{21}B_{11} & A_{22} & A_{21}B_{12} \\ 
B_{21} & 0 & B_{22}
\end{array}\right],\label{PropMatrixMultipliationOfTwoSchurComplementsCMatrix}
\end{align}
where the matrix $C_{22}$ is invertible with
\begin{equation}
C_{22}^{-1} = 
\begin{bmatrix}
A_{22} & A_{21}B_{12} \\ 
 0 & B_{22}
\end{bmatrix}^{-1}
=
\begin{bmatrix}
A_{22}^{-1} & -A_{22}^{-1}A_{11}B_{12}B_{22}^{-1}\\
0 & B_{22}^{-1}
\end{bmatrix}.\label{PropMatrixMultipliationOfTwoSchurComplementsInvFormulaC22}
\end{equation}
\end{proposition}
\begin{proof}
The proof is just a straightforward application of block matrix multiplication. First of all, by the hypotheses the matrix products in the statement of the proposition are well-defined and its easy to verify that $C\in \mathbb{C}^{(m+n-k)\times (m+n-k)}$ as well as the inverse formula (\ref{PropMatrixMultipliationOfTwoSchurComplementsInvFormulaC22}) for $C_{22}^{-1}$ is correct. Second, to prove the formula  (\ref{MatrixMultiplicationOfTwoSchurComplementsInProp}), we compute
\begin{align*}
A/A_{22}B/B_{22}&=(A_{11}-A_{12}A_{22}^{-1}A_{21})(B_{11}-B_{12}B_{22}^{-1}B_{21})\\&=(A_{11}-A_{12}A_{22}^{-1}A_{21})B_{11}-(A_{11}-A_{12}A_{22}^{-1}A_{21})B_{12}B_{22}^{-1}B_{21}\\&=A_{11}B_{11}-A_{12}A_{22}^{-1}A_{21}B_{11}+A_{12}A_{22}^{-1}A_{11}B_{12}B_{22}^{-1}B_{21}-A_{11}B_{12}B_{22}^{-1}B_{21}\\&=A_{11}B_{11}-\{
A_{12}A_{22}^{-1}A_{21}B_{11} + [A_{12}(-A_{22}^{-1}A_{11}B_{12}B_{22}^{-1})+A_{11}B_{12}B_{22}^{-1}]B_{21} 
\}\\&=A_{11}B_{11}-\begin{bmatrix}
A_{12}A_{22}^{-1} & A_{12}(-A_{22}^{-1}A_{11}B_{12}B_{22}^{-1})+A_{11}B_{12}B_{22}^{-1}
\end{bmatrix}
\begin{bmatrix}A_{21}B_{11} \\
B_{21} \end{bmatrix}\\&=A_{11}B_{11}-\begin{bmatrix}
A_{12} & A_{11}B_{12}\end{bmatrix}\begin{bmatrix}
A_{22}^{-1} & -A_{22}^{-1}A_{11}B_{12}B_{22}^{-1}\\
0 & B_{22}^{-1}\end{bmatrix}\begin{bmatrix}
A_{21}B_{11} \\B_{21} \end{bmatrix}\\&= C_{11}-C_{12}C_{22}^{-1}C_{21}\\&=C/C_{22}.
\end{align*}
This completes the proof.
\end{proof}

\begin{proposition}[Matrix multiplication of a Schur complement]\label{PropMatrixMultOfSchurCompl}
If $A\in
\mathbb{C}
^{m\times m}$ is a $2\times2$ block matrix
\[
A=
\begin{bmatrix}
A_{11} & A_{12}\\
A_{21} & A_{22}
\end{bmatrix}
\]
such that $A_{22}\in\mathbb{C}^{p\times p}$ is invertible and $A/A_{22}\in\mathbb{C}^{k\times k}$ then, for any matrices $B\in \mathbb{C}^{l\times k}$ and $C\in \mathbb{C}^{k\times l}$,
\begin{align}\label{SchurComplementSandwich}
D/D_{22}=B(A/A_{22})C,
\end{align}
where $D\in \mathbb{C}^{(l+p)\times (l+p)}$ is the $2\times 2$ block matrix
\begin{align}
D=\begin{bmatrix}
D_{11} & D_{12}\\
D_{21} & D_{22}
\end{bmatrix}= \begin{bmatrix}
BA_{11}C & BA_{12}\\
A_{21}C & A_{22}
\end{bmatrix}=\begin{bmatrix}
B & 0\\
0 & I_p
\end{bmatrix}\begin{bmatrix}
A_{11} & A_{12}\\
A_{21} & A_{22}
\end{bmatrix}\begin{bmatrix}
C & 0\\
0 & I_p
\end{bmatrix}\label{PropMatrixMultOfSchurComplCMatrix}
\end{align}
and $D_{22}=A_{22}$ is invertible. Moreover, the following statements are true:
\begin{itemize}
    \item[(a)] If $A,B,C$ are real matrices then $D$ is a real matrix.
    \item[(b)] If $A$ is a symmetric matrix and $C=B^{T}$ then $D$ is a symmetric matrix.
    \item[(c)] If the hypotheses of (a) and (b) are true then $D$ is a real symmetric matrix.
    \item[(d)] If $A$ is a Hermitian matrix and $C=B^*$ then $D$ is a Hermitian matrix.
\end{itemize}
\end{proposition}
\begin{proof}
By block multiplication the result follows immediately from the definition of $D$ in (\ref{PropMatrixMultOfSchurComplCMatrix}), a straightforward computation yields
\begin{align*}
    B(A/A_{22})C&=B(A_{11}-A_{12}A_{22}^{-1}A_{21})C\\&=BA_{11}C-BA_{12}A_{22}^{-1}A_{21}C\\&=D_{11}-D_{12}D_{22}^{-1}D_{21}\\&=D/D_{22}.
\end{align*}
The remaining part of the proof follows immediately now from the formula (\ref{PropMatrixMultOfSchurComplCMatrix}) in terms of the matrices $A,B$ and $C$.
\end{proof}

The following lemma is well known (see, for instance, \cite[pp. 19, 20, Theorem 1.2]{05FZ}).
\begin{lemma}[Inverse is a Schur complement]\label{LemInvIsASchurCompl}
If $A\in
\mathbb{C}
^{m\times m}$ is a $2\times2$ block matrix
\begin{equation}
A=
\begin{bmatrix}
A_{11} & A_{12}\\
A_{21} & A_{22}
\end{bmatrix},\label{BlockFormOfAForFactorization}
\end{equation}
such that $A_{22}\in \mathbb{C}^{p\times p}$ is invertible then
\begin{equation}
A=\begin{bmatrix}
I_{m-p} & A_{12}A_{22}^{-1}\\
0 & I_p
\end{bmatrix}
\begin{bmatrix}
A/A_{22} & 0\\
0 & A_{22}
\end{bmatrix}
\begin{bmatrix}
I_{m-p} & 0\\
A_{22}^{-1}A_{21} & I_p
\end{bmatrix}.\label{BlockFactorizationUsingSchurComplement}
\end{equation}
Furthermore, $A/A_{22}$ is invertible and only if $A$ is invertible, in which case
\begin{equation}
A^{-1}=
\begin{bmatrix}
(A/A_{22})^{-1} & -(A/A_{22})^{-1}A_{12}A_{22}^{-1}\\
-A_{22}^{-1}A_{21}(A/A_{22})^{-1} & A_{22}^{-1}+A_{22}^{-1}A_{21}(A/A_{22})^{-1}A_{12}A_{22}^{-1}
\end{bmatrix}\label{InvBlockFactorizationUsingSchurComplement}
\end{equation}
and
\begin{equation}
(A/A_{22})^{-1}=(A^{-1})_{11}.\label{InvSchurComplIsCornerBlock}
\end{equation}
Moreover, 
\begin{equation}
B/B_{22} = A^{-1}, \label{InvIsASchurComplFormula}
\end{equation}
where $B\in \mathbb{C}^{2m\times 2m}$ is the $2\times 2$ block matrix
\begin{equation}
B=
\begin{bmatrix}
0_m & I_m\\
I_m & -A
\end{bmatrix}\label{InvIsSchurComplBMatrixBlockForm}
\end{equation}
and $B_{22}=-A$ is invertible. In addition, if the matrix $A$ is real, symmetric, Hermitian, or real and symmetric then the matrix $B$ is real, symmetric, Hermitian, or real and symmetric, respectively.
\end{lemma}
\begin{proof}
First, if $A\in
\mathbb{C}
^{m\times m}$ is a $2\times2$ block matrix of the form (\ref{BlockFormOfAForFactorization}) and $A_{22}\in \mathbb{C}^{p\times p}$ is invertible, then the factorization of $A$ in (\ref{BlockFactorizationUsingSchurComplement}) follows immediately from block multiplication as does the inverse formulas
\begin{align*}
\begin{bmatrix}
I_{m-p} & A_{12}A_{22}^{-1}\\
0 & I_p
\end{bmatrix}^{-1}
&=
\begin{bmatrix}
I_{m-p} & -A_{12}A_{22}^{-1}\\
0 & I_p
\end{bmatrix},\\
\begin{bmatrix}
I_{m-p} & 0\\
A_{22}^{-1}A_{21} & I_p
\end{bmatrix}^{-1}
&=
\begin{bmatrix}
I_{m-p} & 0\\
-A_{22}^{-1}A_{21} & I_p
\end{bmatrix}.
\end{align*}
It now follows from this and the factorization (\ref{BlockFactorizationUsingSchurComplement}) that, $A/A_{22}$ is invertible if and only if $A$ is invertible, in which case we have the factorization
\[
A^{-1}=
\begin{bmatrix}
I_{m-p} & 0\\
-A_{22}^{-1}A_{21} & I_p
\end{bmatrix}
\begin{bmatrix}
(A/A_{22})^{-1} & 0\\
0 & A_{22}^{-1}
\end{bmatrix}
\begin{bmatrix}
I_{m-p} & -A_{12}A_{22}^{-1}\\
0 & I_p
\end{bmatrix}.
\]
Hence, by block multiplication it follows immediately that the equality in (\ref{InvBlockFactorizationUsingSchurComplement}) is true which implies the equality (\ref{InvSchurComplIsCornerBlock}) is also true. Finally, for the matrix $B\in \mathbb{C}^{2m\times 2m}$ defined in (\ref{InvIsASchurComplFormula}) we have $B_{22}=-A$ is invertible so that
\begin{align*}
B/B_{22} &= B_{11}-B_{12}B_{22}^{-1}B_{21}\\
&=I_mA^{-1}I_m = A^{-1}.
\end{align*}
The remaining part of the proof follows immediately now from the formula (\ref{InvIsASchurComplFormula}) in terms of the matrix $A$. This completes the proof.
\end{proof}

\begin{proposition}[Inverse of a Schur complement]\label{PropInvOfASchurCompl}
If $A\in
\mathbb{C}
^{m\times m}$ is a $2\times2$ block matrix
\[
A=
\begin{bmatrix}
A_{11} & A_{12}\\
A_{21} & A_{22},
\end{bmatrix}
\]
such that $A_{22}$ is invertible and $A/A_{22}\in\mathbb{C}^{k\times k}$ is invertible then
\begin{align}
C/C_{22}=(A/A_{22})^{-1}
\end{align}
where $C\in \mathbb{C}^{(k+m)\times (k+m)}$ is the $3\times3$ block matrix with the following block partitioned structure $C=[C_{ij}]_{i,j=1,2}$: 
\begin{align} 
C & =
\left[\begin{array}{c;{2pt/2pt} c c}
C_{11} & C_{12} \\ \hdashline[2pt/2pt]
C_{21} & C_{22}
\end{array}\right]
=
\left[\begin{array}{c; {2pt/2pt} c c}
0_k & I_k & 0 \\\hdashline[2pt/2pt]
I_k & -A_{11} & -A_{12} \\ 
0 & -A_{21} & -A_{22}
\end{array}\right]\label{PropInvOfASchurComplCMatrix}
\end{align}
with $C_{22}=-A$ invertible. Moreover, if the matrix $A$ is real, symmetric, Hermitian, or real and symmetric then the matrix $B$ is real, symmetric, Hermitian, or real and symmetric, respectively.
\end{proposition}
\begin{proof}
By Lemma \ref{LemInvIsASchurCompl} and, in particular, formulas (\ref{InvSchurComplIsCornerBlock})-(\ref{InvIsSchurComplBMatrixBlockForm}) we have
\[
(A/A_{22})^{-1}=(A^{-1})_{11}=\begin{bmatrix}
I_k & 0
\end{bmatrix}
A^{-1}
\begin{bmatrix}
I_k \\ 0
\end{bmatrix}
=
\begin{bmatrix}
I_k & 0
\end{bmatrix}
B/B_{22}
\begin{bmatrix}
I_k \\ 0
\end{bmatrix},
\]
where the matrix $B\in \mathbb{C}^{2m\times 2m}$ is defined in terms of $A$ in (\ref{InvIsASchurComplFormula}). Hence, by Proposition \ref{PropMatrixMultOfSchurCompl}
\[
C/C_{22} = \begin{bmatrix}
I_k & 0
\end{bmatrix}
B/B_{22}
\begin{bmatrix}
I_k \\ 0
\end{bmatrix},
\]
where $C=[C_{ij}]_{i,j=1,2}\in \mathbb{C}^{(k+m)\times (k+m)}$ is given by 
\begin{align*}
C_{11}&=\begin{bmatrix}
I_k & 0
\end{bmatrix}
B_{11}
\begin{bmatrix}
I_k \\ 0
\end{bmatrix}
=0_k,\;\;
C_{12}=\begin{bmatrix}
I_k & 0
\end{bmatrix}
B_{12} =\begin{bmatrix}
I_k & 0
\end{bmatrix}\\
&=
B_{21}\begin{bmatrix}
I_k \\ 0
\end{bmatrix}
=\begin{bmatrix}
I_k \\ 0
\end{bmatrix},\;\;
C_{22}=-A=
\begin{bmatrix}
-A_{11} & -A_{12} \\ 
-A_{21} & -A_{22}
\end{bmatrix},
\end{align*}
which yields the formula (\ref{PropInvOfASchurComplCMatrix}) for $C$. The remaining part of the proof follows immediately now from the formula (\ref{PropInvOfASchurComplCMatrix}) in terms of the matrix $A$. This completes the proof.
\end{proof}

\subsubsection{Kronecker products}

The results in this section, on Kronecker products of matrices when one or more of the matrices is a Schur complement, are by far the most technical part of the paper. The following are two major reasons for this. 

First, the technique of finding Schur complement representation for $A/A_{22}\otimes B$ requires that $B$ is invertible by Lemma \ref{LemKroneckerProdSchurComplWithAMatrix}. And from this simple result though we are able to ``easily" build up on it to find Schur complement representations for $A\otimes B/B_{22}$ (Corollary \ref{CorLemKroneckerProdSchurComplWithAMatrix}), but requires $A$ to be invertible, and $A/A_{22}\otimes B/B_{22}$ (Proposition \ref{PropKroneckerProdOfTwoSchurComplements}), but it requires $A/A_{22}$ and $B/B_{22}$  to be invertible. Again, these invertibility requirements are just due to the invertibility hypothesis in Lemma \ref{LemKroneckerProdSchurComplWithAMatrix}. In this paper, there are two ways we treat the case when the invertibility hypotheses are not true. The most general way to do it is to proceed in a similar manner as in our proof of the Bessmertny\u{\i} Realization Theorem (Theorem \ref{ThmBessmRealiz}) (more specifically, that part of the proof corresponding in the flow diagram in Fig. \ref{fig-FlowDiagPrfThmBessmRealiz} to the case $\det P(z)\equiv 0$). The other way to do it, albeit in a less general way, is to use Proposition \ref{PropScalarProdOfASchurCompl} which we find to be more ``natural" in the realization problem when it applies. 

The second reason that this section is more technical is due to the fact that in the proof of Proposition \ref{PropKroneckerProdOfTwoSchurComplements}, where we derive a Schur complement formula for the Kronecker product of two Schur complements, i.e., $A/A_{22}\otimes B/B_{22}$, we must use the result in Section \ref{SecCompositionOfSchurComplement} on composition of Schur complements (Proposition \ref{PropComposSchurComplem}). This result though is slightly more difficult to understand we feel then the rest of the results in this paper, and as such, may not be at first read easy to apply in practice.

Therefore, we give several examples below, namely, Example \ref{ExLemKroneckerProdSchurComplWithAMatrix}, Example \ref{ExCorLemKroneckerProdSchurComplWithAMatrix}, and Example \ref{ExPropKroneckerProductOfTwoSchurComplements} in order to illustrate the statement and proofs of Lemma \ref{LemKroneckerProdSchurComplWithAMatrix}, Corollary \ref{CorLemKroneckerProdSchurComplWithAMatrix}, and Proposition \ref{PropKroneckerProdOfTwoSchurComplements}, respectively.

We follow a similar procedure in Section \ref{SecOnKroneckerProductsOfLinearMatrixPencils} when we consider the Kronecker product of realizations.

\begin{lemma}[Kronecker product of a Schur complement with a matrix]\label{LemKroneckerProdSchurComplWithAMatrix}
If $B\in
\mathbb{C}
^{n\times n}$ and $A\in
\mathbb{C}
^{m\times m}$ is a $2\times2$ block matrix
\[
A=
\begin{bmatrix}
A_{11} & A_{12}\\
A_{21} & A_{22}
\end{bmatrix},
\]
then the Kronecker product of $A$ with $B$, 
\begin{equation}
C=A\otimes B\in\mathbb{C}^{mn\times mn},\label{LemKroneckerProdSchurComplWithAMatrixCBlockForm}
\end{equation}
has following $2\times 2$ block matrix form $C=[C_{i,j}]_{i.j=1,2}$:
\begin{equation}
 C =
\begin{bmatrix}
C_{11} & C_{12}\\
C_{21} & C_{22}
\end{bmatrix}
=
\begin{bmatrix}
A_{11}\otimes B & A_{12}\otimes B\\
A_{21}\otimes B & A_{22}\otimes B
\end{bmatrix}.\label{LemKroneckerProdSchurComplWithAMatrixBlockForm}
\end{equation}
Furthermore, if $A_{22}$ and $B$ are invertible then $C_{22}=A_{22}\otimes B$ is invertible and 
\begin{equation}
C/C_{22}=A/A_{22}\otimes B.
\end{equation}
Moreover, if both matrices $A$ and $B$ are real, symmetric, Hermitian, or real and symmetric then the matrix $C$ is real, symmetric, Hermitian, or real and symmetric, respectively.
\end{lemma}
\begin{proof}
The first part of the proof of this lemma, namely, that $C=A\otimes B\in\mathbb{C}^{mn\times mn}$ has the block form (\ref{LemKroneckerProdSchurComplWithAMatrixCBlockForm}), follows immediately from the definition of the Kronecker product $A\otimes B=[a_{ij}B]_{i,j=1,\ldots,m}$ of the matrices $A$ and $B$ and the $2\times 2$ block form of $A=[A_{ij}]_{i,j=1,2}$. Suppose now that $A_{22}$ and $B$ are invertible. Then it follows by elementary properties of Kronecker products that their Kronecker product $A_{22}\otimes B$ is invertible with
$\left(  A_{22}\otimes B\right)^{-1}=A_{22}^{-1}\otimes
B^{-1}$ and that
\begin{align*}
    A/A_{22}\otimes B&=\left(  A_{11}-A_{12}A_{22}^{-1}A_{21}\right)  \otimes
B\\&=A_{11}\otimes B-\left[  (A_{12}A_{22}^{-1}A_{21})\otimes
B\right] \\&=A_{11}\otimes B-\left(  A_{12}\otimes B\right)  \left(  A_{22}^{-1}\otimes B^{-1}\right)  \left(  A_{21}\otimes B\right) \\&=A_{11}\otimes
B-\left(  A_{12}\otimes B\right)  \left(  A_{22}\otimes B\right)  ^{-1}\left(
A_{21}\otimes B\right) \\&=C/C_{22}.
\end{align*}
The remaining part of the proof follows immediately now by elementary properties of Kronecker products, namely, that $\overline {A\otimes B}=\overline{A}\otimes \overline{B}$, $(A\otimes B)^T=A^T\otimes B^T$, and $(A\otimes B)^*=A^*\otimes B^*$. This completes the proof.
\end{proof}

\begin{example}\label{ExLemKroneckerProdSchurComplWithAMatrix}
To illustrate Lemma \ref{LemKroneckerProdSchurComplWithAMatrix} and the proof, consider the following example. Let
\begin{align*}
A=\left[\begin{array}{c;{2pt/2pt} c c}
A_{11} & A_{12} \\ \hdashline[2pt/2pt]
A_{21} & A_{22}
\end{array}\right]
=
\left[\begin{array}{c;{2pt/2pt} c c}
0 & 2 \\ \hdashline[2pt/2pt]
2 & 4
\end{array}\right],\;\;
B=\begin{bmatrix}
2 & 3\\
3 & 5
\end{bmatrix}.
\end{align*}
Then
\begin{align*}
C=A\otimes B =\begin{bmatrix}
0B & 2B\\
2B & 4B
\end{bmatrix}
=
\begin{bmatrix}
0 & 0 & 4 & 6\\
0 & 0 & 6 & 10\\
4 & 6 & 8 & 12\\
6 & 10 & 12 & 20
\end{bmatrix}
\end{align*}
and has the $2\times 2$ block matrix form
\begin{align*}
C=\left[\begin{array}{c;{2pt/2pt} c c}
C_{11} & C_{12} \\ \hdashline[2pt/2pt]
C_{21} & C_{22}
\end{array}\right]
=
\left[\begin{array}{c;{2pt/2pt} c c}
A_{11}\otimes B & A_{12}\otimes B \\ \hdashline[2pt/2pt]
A_{21}\otimes B & A_{22}\otimes B
\end{array}\right]
=
\left[\begin{array}{c c; {2pt/2pt} c c}
0 & 0 & 4 & 6\\
0 & 0 & 6 & 10\\ \hdashline[2pt/2pt]
4 & 6 & 8 & 12\\
6 & 10 & 12 & 20
\end{array}\right].
\end{align*}
Now $A_{22}=[4]$ and $B$ are invertible which implies, by Lemma \ref{LemKroneckerProdSchurComplWithAMatrix}, that $C_{22}=A_{22}\otimes B$ is invertible with $C_{22}^{-1}=A_{22}^{-1}\otimes B^{-1}$ and $C/C_{22}=A/A_{22}\otimes B$, which we can show in this example by the following direct calculations:
\begin{gather*}
A_{22}^{-1}=\begin{bmatrix}
\frac{1}{4}
\end{bmatrix},\;\;B^{-1}=\begin{bmatrix}
5 & -3\\
-3 & 2
\end{bmatrix},\;\;A_{22}^{-1}\otimes B^{-1}=\frac{1}{4}B^{-1}=\begin{bmatrix}
\frac{5}{4} & -\frac{3}{4}\\
-\frac{3}{4} & \frac{2}{4}
\end{bmatrix}=C_{22}^{-1},\\
A/A_{22}=[0]-[2][4]^{-1}[2]=\begin{bmatrix}
-1
\end{bmatrix},\\
A/A_{22}\otimes B = 
(-1)B
=
\begin{bmatrix}
-2 & -3\\
-3 & -5
\end{bmatrix},\\
C/C_{22}
=\begin{bmatrix}
0 & 0\\
0 & 0
\end{bmatrix}
-
\begin{bmatrix}
4 & 6\\
6 & 10
\end{bmatrix}
\begin{bmatrix}
8 & 12\\
12 & 20
\end{bmatrix}^{-1}
\begin{bmatrix}
4 & 6\\
6 & 10
\end{bmatrix}
=
\begin{bmatrix}
-2 & -3\\
-3 & -5
\end{bmatrix}.
\end{gather*}
\end{example}

The definition below comes from \cite[p.\ 259]{91HJ}.
\begin{definition}\label{DefCommutationMatrices}
For any positive integers $m,n$, the matrix $P(m,n)\in\mathbb{C}^{mn\times mn}$ is defined by
\begin{equation}
P(m,n) = \sum_{i=1}^m \sum_{j=1}^n E_{ij} \otimes E_{ij}^T=[E_{ij}^T]_{i,j=1}^{m,n},
\end{equation}
where $\{E_{ij}:i=1,\ldots, m, j=1,\ldots, n\}$ is the standard basis for $\mathbb{C}^{m\times n}$, i.e., each $E_{ij}\in \mathbb{C}^{m\times n}$ has entry $1$ in the $i$th row, $j$th column and all other entries are zero. The matrix $P(m,n)$ is called the commutation matrix (with respect to $m$ and $n$).
\end{definition}

The following lemma is proven in \cite[Corollary 4.3.10, p.\ 260]{91HJ}.
\begin{lemma}[Main properties of commutation matrices]\label{LemElemPropertiesCommMatrices}
Let positive integers $m,n,p,$ and $q$ be given and let $P(p,m)\in \mathbb{C}^{pm\times pm}$ and $P(n,q)\in \mathbb{C}^{nq\times nq}$ denote the commutation matrices (as defined in Def. \ref{DefCommutationMatrices}). Then $P(p,m)$ is a permutation matrix and $P(p,m)=\overline{P(p,m)}=P(m,p)^T=P(m,p)^{-1}$. Furthermore, for all $A\in \mathbb{C}^{m\times n}$ and $B\in \mathbb{C}^{p\times q}$,
\begin{equation}
B\otimes A = P(m,p)^{T}(A\otimes B)P(n,q).
\end{equation}
\end{lemma}

\begin{corollary}[Part 2 of Lemma \ref{LemKroneckerProdSchurComplWithAMatrix}]\label{CorLemKroneckerProdSchurComplWithAMatrix}
If $A\in
\mathbb{C}
^{m\times m}$ and $B\in
\mathbb{C}
^{n\times n}$ is a $2\times2$ block matrix
\[
B=
\begin{bmatrix}
B_{11} & B_{12}\\
B_{21} & B_{22}
\end{bmatrix},
\]
such that $A$ and $B_{22}$ are invertible with $B/B_{22}\in \mathbb{C}^{l\times l}$ then 
\begin{equation}
D/D_{22}=A\otimes B/B_{22},
\end{equation}
where 
\begin{equation}
D=Q^T(B\otimes A)Q,\; Q\in\mathbb{C}^{mn\times mn}\label{CorLemKroneckerProdSchurComplWithAMatrixDMatrix}
\end{equation}
have the following $2\times 2$ block matrix forms $Q=[Q_{i,j}]_{i.j=1,2}$, $D=[D_{i,j}]_{i.j=1,2}$:
\begin{align}
Q&=\begin{bmatrix}
Q_{11} & Q_{12}\\
Q_{21} & Q_{22}
\end{bmatrix}
=\begin{bmatrix}
P(l,m) & 0\\
0 & I_{mn-lm}
\end{bmatrix},\label{CorLemKroneckerProdSchurComplWithAMatrixQMatrixBlkForm}
\\
D &=
\begin{bmatrix}
D_{11} & D_{12}\\
D_{21} & D_{22}
\end{bmatrix}
=
\begin{bmatrix}
A\otimes B_{11} & P(l,m)^T(B_{12}\otimes A)\\
(B_{21}\otimes A)P(l,m) & B_{22}\otimes A
\end{bmatrix},\label{CorLemKroneckerProdSchurComplWithAMatrixDMatrixBlkForm}
\end{align}
in which $P(l,m)\in \mathbb{C}^{lm\times lm}$ is the commutation matrix (with respect to $l$ and $m$ as defined in Def.\ \ref{DefCommutationMatrices}) and $D_{22}=B_{22}\otimes A$ is invertible. Furthermore, $Q$ is a permutation matrix and $Q=\overline{Q}$, $Q^T=Q^{-1}$. Moreover, if both matrices $A$ and $B$ are real, symmetric, Hermitian, or real and symmetric then the matrix $D$ is real, symmetric, Hermitian, or real and symmetric, respectively.
\end{corollary}
\begin{proof}
By the hypotheses, Lemma \ref{LemElemPropertiesCommMatrices}, and Lemma \ref{LemKroneckerProdSchurComplWithAMatrix} we have
\[
P(m,l)^T(A\otimes B/B_{22})P(m,l)=B/B_{22}\otimes A = C/C_{22},
\]
where $C=B\otimes A\in \mathbb{C}^{nm\times nm}$ has the $2\times 2$ block matrix form:
\begin{equation*}
C =
\begin{bmatrix}
C_{11} & C_{12}\\
C_{21} & C_{22}
\end{bmatrix}
=
\begin{bmatrix}
B_{11}\otimes A & B_{12}\otimes A\\
B_{21}\otimes A & B_{22}\otimes A
\end{bmatrix},
\end{equation*}
in which $C_{22}=B_{22}\otimes A$ is invertible. By Lemma \ref{LemElemPropertiesCommMatrices} and Proposition \ref{PropMatrixMultOfSchurCompl} it follows that
\begin{equation*}
A\otimes B/B_{22} = [P(m,l)^T]^{-1}C/C_{22}P(m,l)^{-1}=P(l,m)^TC/C_{22}P(l,m)=D/D_{22},
\end{equation*}
where $D=[D_{ij}]_{i,j=1,2}\in \mathbb{C}^{mn\times mn}$ is the $2\times 2$ block matrix defined by (\ref{CorLemKroneckerProdSchurComplWithAMatrixDMatrixBlkForm}), where we have used the fact that
\begin{equation}
P(l,m)^T(B_{11}\otimes A)P(l,m)=A\otimes B_{11},    
\end{equation}
which follows from Lemma \ref{LemElemPropertiesCommMatrices}, 
from which it follows immediately from block multiplication that $D=Q^T(B\otimes A)Q$, where $Q\in\mathbb{C}^{mn\times mn}$ is the matrix defined in (\ref{CorLemKroneckerProdSchurComplWithAMatrixQMatrixBlkForm}). The fact that $Q$ is a permutation matrix satisfying $Q=\overline{Q}$, $Q^T=Q^{-1}$ follows from its definition (\ref{CorLemKroneckerProdSchurComplWithAMatrixQMatrixBlkForm}) and the corresponding properties of $P(l,m)$ in Lemma \ref{LemElemPropertiesCommMatrices}. The remaining part of the proof follows immediately now formula for $D$ in (\ref{CorLemKroneckerProdSchurComplWithAMatrixDMatrixBlkForm}) and the properties of $Q$. This completes the proof.
\end{proof}

\begin{example}\label{ExCorLemKroneckerProdSchurComplWithAMatrix}
To illustrate Corollary \ref{CorLemKroneckerProdSchurComplWithAMatrix} and the proof, consider the following example. Let
\begin{align*}
A=
\begin{bmatrix}
0 & 2 \\
2 & 4
\end{bmatrix},\;\;
B=\left[\begin{array}{c;{2pt/2pt} c c}
B_{11} & B_{12} \\ \hdashline[2pt/2pt]
B_{21} & B_{22}
\end{array}\right]
=
\left[\begin{array}{c;{2pt/2pt} c c}
2 & 3 \\ \hdashline[2pt/2pt]
3 & 5
\end{array}\right].
\end{align*}
Then, for this example in the notation of Corollary \ref{CorLemKroneckerProdSchurComplWithAMatrix}, we have $m=n=2, l=1$ and
\begin{gather*}
P(1,2)=\begin{bmatrix}
E_{ij}^T
\end{bmatrix}_{i,j=1}^{1,2}
=\begin{bmatrix}
E_{11}^T & E_{12}^T
\end{bmatrix}
=\begin{bmatrix}
\begin{bmatrix}
1 & 0
\end{bmatrix}^T & \begin{bmatrix}
0 & 1
\end{bmatrix}^T
\end{bmatrix}
=\begin{bmatrix}
1 & 0\\
0 & 1
\end{bmatrix}
=
I_2,\\
Q=\begin{bmatrix}
P(1,2) & 0\\
0 & I_2
\end{bmatrix}
=
\begin{bmatrix}
I_2 & 0\\
0 & I_2
\end{bmatrix}
=
\begin{bmatrix}
1 & 0 & 0 & 0\\
0 & 1 & 0 & 0\\
0 & 0 & 1 & 0\\
0 & 0 & 0 & 1\\
\end{bmatrix}
=I_4,\\
D=Q^T(B\otimes A)Q
=B\otimes A
=
\begin{bmatrix}
2A & 3A\\
3A & 5A
\end{bmatrix}
=
\begin{bmatrix}
0 & 4 & 0 & 6\\
4 & 8 & 6 & 12\\
0 & 6 & 0 & 10\\
6 & 12 & 10 & 20\\
\end{bmatrix}
\end{gather*}
and $D$ has the $2\times 2$ block matrix form
\begin{gather*}
D=\left[\begin{array}{c;{2pt/2pt} c c}
D_{11} & D_{12} \\ \hdashline[2pt/2pt]
D_{21} & D_{22}
\end{array}\right]
=
\left[\begin{array}{c;{2pt/2pt} c c}
A\otimes B_{11} & P(1,2)^T(B_{12}\otimes A) \\ \hdashline[2pt/2pt]
(B_{21}\otimes A)P(1,2) & B_{22}\otimes A
\end{array}\right]
\\
=
\left[\begin{array}{c;{2pt/2pt} c c}
A\otimes B_{11} & B_{12}\otimes A \\ \hdashline[2pt/2pt]
B_{21}\otimes A & B_{22}\otimes A
\end{array}\right]
=
\left[\begin{array}{c c; {2pt/2pt} c c}
0 & 4 & 0 & 6\\
4 & 8 & 6 & 12\\ \hdashline[2pt/2pt]
0 & 6 & 0 & 10\\
6 & 12 & 10 & 20\\
\end{array}\right].
\end{gather*}
Now $A$ and $B_{22}=[5]$ are invertible which implies, by Corollary \ref{CorLemKroneckerProdSchurComplWithAMatrix}, that $D_{22}=B_{22}\otimes A$ is invertible with $D_{22}^{-1}=B_{22}^{-1}\otimes A^{-1}$ and $D/D_{22}=A\otimes B/B_{22}$, which we can show in this example by the following direct calculations:
\begin{gather*}
B_{22}^{-1}=\begin{bmatrix}
\frac{1}{5}
\end{bmatrix},\;\;A^{-1}=\begin{bmatrix}
-1 & \frac{1}{2}\\
\frac{1}{2} & 0
\end{bmatrix},\;\;B_{22}^{-1}\otimes A^{-1}=\frac{1}{5}A^{-1}=\begin{bmatrix}
-\frac{1}{5} & \frac{1}{10}\\
\frac{1}{10} & 0
\end{bmatrix}=D_{22}^{-1},\\
B/B_{22}=[2]-[3][5]^{-1}[3]=\begin{bmatrix}
\frac{1}{5}
\end{bmatrix},\\
A\otimes B/B_{22}
=
\begin{bmatrix}
0\begin{bmatrix}
\frac{1}{5}
\end{bmatrix} & 2\begin{bmatrix}
\frac{1}{5}
\end{bmatrix}   \vspace{0.1cm}\\
2\begin{bmatrix}
\frac{1}{5}
\end{bmatrix} & 4\begin{bmatrix}
\frac{1}{5}
\end{bmatrix}
\end{bmatrix}
=
\begin{bmatrix}
0 & \frac{2}{5}   \vspace{0.1cm}\\
\frac{2}{5} & \frac{4}{5}
\end{bmatrix}
,\\
D/D_{22}
=\begin{bmatrix}
0 & 4\\
4 & 8
\end{bmatrix}
-
\begin{bmatrix}
0 & 6\\
6 & 12
\end{bmatrix}
\begin{bmatrix}
0 & 10\\
10 & 20
\end{bmatrix}^{-1}
\begin{bmatrix}
0 & 6\\
6 & 12
\end{bmatrix}
=
\begin{bmatrix}
0 & \frac{2}{5}   \vspace{0.1cm}\\
\frac{2}{5} & \frac{4}{5}
\end{bmatrix}.
\end{gather*}
\end{example}

\begin{proposition}[Kronecker product of two Schur complements]\label{PropKroneckerProdOfTwoSchurComplements}
If $A\in
\mathbb{C}
^{m\times m}$ and $B\in
\mathbb{C}
^{n\times n}$ are $2\times2$ block matrices
\[
A=
\begin{bmatrix}
A_{11} & A_{12}\\
A_{21} & A_{22}
\end{bmatrix},\;
B=
\begin{bmatrix}
B_{11} & B_{12}\\
B_{21} & B_{22}
\end{bmatrix}
\]
such that the matrices $A_{22}$, $B_{22}$, $A/A_{22}$, and $B/B_{22}$ are invertible with $B/B_{22}\in \mathbb{C}^{l\times l}$ then 
\begin{equation}
M/M_{22} = A/A_{22}\otimes B/B_{22},
\end{equation}
where $M\in\mathbb{C}^{mn\times mn}$ is the invertible matrix
\begin{equation}
M = Q^{T} (B\otimes A) Q,\label{PropKroneckerProdOfTwoSchurComplementsMMatrix}
\end{equation}
$Q\in \mathbb{C}^{mn\times mn}$ is the permutation matrix defined by (\ref{CorLemKroneckerProdSchurComplWithAMatrixQMatrixBlkForm}), and $M=[M_{ij}]_{i,j=1,2}$ is the $2\times 2$ matrix with the block partitioned structure: 
\begin{align} 
M & =
\left[\begin{array}{c;{2pt/2pt} c c}
M_{11} & M_{12} \\ \hdashline[2pt/2pt]
M_{21} & M_{22}
\end{array}\right]
=
\left[\begin{array}{c; {2pt/2pt} c c}
D_{33} & D_{34} & D_{32} \\\hdashline[2pt/2pt]
D_{43} & D_{44} & D_{42} \\ 
D_{23} & D_{24} & D_{22}
\end{array}\right],
\end{align}
where
\begin{equation}
M_{22}
=
\begin{bmatrix}
D_{44} & D_{42}\\
D_{24} & D_{22}
\end{bmatrix},
\end{equation}
is invertible, and
\begin{align}
\left[\begin{array}{c;{2pt/2pt} c c}
D_{33} & D_{34} \\ \hdashline[2pt/2pt]
D_{43} & D_{44}
\end{array}\right]
=
\left[\begin{array}{c;{2pt/2pt} c c}
A_{11}\otimes B_{11} & A_{12}\otimes B_{11} \\ \hdashline[2pt/2pt]
A_{21}\otimes B_{11} & A_{22}\otimes B_{11}
\end{array}\right]
=A\otimes B_{11},
\end{align}
\begin{align}
D_{22} = B_{22}\otimes A,
\end{align}
\begin{align}
\begin{bmatrix}
D_{32}\\
D_{42}
\end{bmatrix}
=
P(l, m)^{T}(B_{12}\otimes A),
\end{align}
\begin{align}
\begin{bmatrix}
D_{23} & D_{24}
\end{bmatrix}
=
(B_{21}\otimes A)P(l,m),
\end{align}
in which $P(l,m)\in \mathbb{C}^{lm\times lm}$ is the commutation matrix (with respect to $l$ and $m$ as defined in Def.\ \ref{DefCommutationMatrices}). Moreover, if both matrices $A$ and $B$ are real, symmetric, Hermitian, or real and symmetric then the matrix $M$ is a real, symmetric, Hermitian, or real and symmetric, respectively.
\end{proposition}
\begin{proof}
As we shall see, the proof of this proposition will follow immediately from Lemma \ref{LemKroneckerProdSchurComplWithAMatrix} and Corollary \ref{CorLemKroneckerProdSchurComplWithAMatrix} on the Kronecker product of a Schur complement with a matrix, and Proposition \ref{PropComposSchurComplem} on the composition of Schur complements. 

Let $A=[A_{ij}]_{i,j=1,2}\in
\mathbb{C}
^{m\times m}$ and $B=[B_{ij}]_{i,j=1,2}\in
\mathbb{C}
^{n\times n}$ be $2\times 2$ block matrices such that the matrices $A_{22}$, $B_{22}$, $A/A_{22}$, and $B/B_{22}$ are invertible with $B/B_{22}\in \mathbb{C}^{l\times l}$. Then by Lemma \ref{LemKroneckerProdSchurComplWithAMatrix} we have 
\begin{equation}
C/C_{22} = A/A_{22}\otimes B/B_{22},
\end{equation}
where $C=A\otimes (B/B_{22})\in \mathbb{C}^{ml\times ml}$ with the $2\times 2$ block matrix form
\begin{equation}
 C =
\begin{bmatrix}
C_{11} & C_{12}\\
C_{21} & C_{22}
\end{bmatrix}
=
\begin{bmatrix}
A_{11}\otimes (B/B_{22}) & A_{12}\otimes (B/B_{22})\\
A_{21}\otimes (B/B_{22}) & A_{22}\otimes (B/B_{22})
\end{bmatrix}
=
A\otimes (B/B_{22})
\end{equation}
with $C_{22}=A_{22}\otimes (B/B_{22})$ invertible. By Corollary \ref{CorLemKroneckerProdSchurComplWithAMatrix} we have 
\begin{align}
D/D_{22} = A\otimes (B/B_{22})=C,
\end{align}
where $D=Q^T(B\otimes A)Q,\; Q \in \mathbb{C}^{mn\times mn}$ have the $2\times 2$ block matrix forms (\ref{CorLemKroneckerProdSchurComplWithAMatrixDMatrixBlkForm}) and (\ref{CorLemKroneckerProdSchurComplWithAMatrixQMatrixBlkForm}), respectively, and $D_{22}=B_{22}\otimes A$ is invertible. We now write this all in another way in order to make it clear how we will use Proposition \ref{PropComposSchurComplem}. First, we write
\begin{equation}
 \hspace{-0.25cm} D/D_{22}=
\begin{bmatrix}
(D/D_{22})_{33} & (D/D_{22})_{34}\\
(D/D_{22})_{43} & (D/D_{22})_{44}
\end{bmatrix}
=
\begin{bmatrix}
A_{11}\otimes (B/B_{22}) & A_{12}\otimes (B/B_{22})\\
A_{21}\otimes (B/B_{22}) & A_{22}\otimes (B/B_{22})\label{PropKroneckerProdOfTwoSchurComplsBlockFormSchurComplD}
\end{bmatrix}
\end{equation}
in which $(D/D_{22})_{44}=A_{22}\otimes (B/B_{22})$ is invertible. Now consider the $2\times 2$ block matrix form $D=[D_{ij}]_{i,j=1,2}$ in (\ref{CorLemKroneckerProdSchurComplWithAMatrixDMatrixBlkForm}). Then it follows that 
\begin{align}
D_{11}=A\otimes B_{11}=\begin{bmatrix}
A_{11}\otimes B_{11} & A_{12}\otimes B_{11} \\
A_{21}\otimes B_{11} & A_{22}\otimes B_{11}
\end{bmatrix}\label{PropKroneckerProdOfTwoSchurComplsBlockFormSchurComplD11}
\end{align}
(where the last equality follows from the first part of Lemma \ref{LemKroneckerProdSchurComplWithAMatrix}) which is conformal to the block structure of $D/D_{22}$ in (\ref{PropKroneckerProdOfTwoSchurComplsBlockFormSchurComplD}). Let us write now this $2\times 2$ block form of $D_{11}$ in (\ref{PropKroneckerProdOfTwoSchurComplsBlockFormSchurComplD11}) as 
\begin{equation}
D_{11}=\begin{bmatrix}
D_{33} & D_{34}\\
D_{43} & D_{44}
\end{bmatrix}.
\end{equation}
This yields a subpartitioning of the $2\times 2$ block matrix $D$ from (\ref{CorLemKroneckerProdSchurComplWithAMatrixDMatrixBlkForm}) into a $3\times 3$ block matrix as 
\begin{align} 
D & =
\left[\begin{array}{c;{2pt/2pt} c c}
D_{11} & D_{12} \\ \hdashline[2pt/2pt]
D_{21} & D_{22}
\end{array}\right]
=
\left[\begin{array}{c c; {2pt/2pt} c}
D_{33} & D_{34} & D_{32} \\
D_{43} & D_{44} & D_{42} \\ \hdashline[2pt/2pt]
D_{23} & D_{24} & D_{22}
\end{array}\right].
\end{align}
On the other hand, we can repartition the matrix $D$ in the following $2\times 2$ block partitioned structure $D=M=[M_{ij}]_{i,j=1,2}$: 
\begin{align} 
M & =
\left[\begin{array}{c;{2pt/2pt} c c}
M_{11} & M_{12} \\ \hdashline[2pt/2pt]
M_{21} & M_{22}
\end{array}\right]
=
\left[\begin{array}{c; {2pt/2pt} c c}
D_{33} & D_{34} & D_{32} \\\hdashline[2pt/2pt]
D_{43} & D_{44} & D_{42} \\ 
D_{23} & D_{24} & D_{22}
\end{array}\right],
\end{align}
where, in particular,
\begin{equation}
M_{22}
=
\begin{bmatrix}
D_{44} & D_{42}\\
D_{24} & D_{22}
\end{bmatrix}.
\end{equation}
Therefore, Proposition \ref{PropComposSchurComplem} applies here since $D_{22}$ is invertible as is $(D/D_{22})_{44}$, which together with the above implies
\begin{align}
M/M_{22}=(D/D_{22})/(D/D_{22})_{44}=C/C_{22}=A/A_{22}\otimes B/B_{22},
\end{align}
as desired. The remaining part of the proof follows immediately now formula for $M$ in (\ref{PropKroneckerProdOfTwoSchurComplementsMMatrix}) and the properties of $Q$.
\end{proof}

\begin{remark}\label{RemPMatrix}
In Corollary \ref{CorLemKroneckerProdSchurComplWithAMatrix} and Proposition \ref{PropKroneckerProdOfTwoSchurComplements}, we can write $D$ and $M$ in terms of $A\otimes B$ instead of $B\otimes A$ using the formula:
\begin{equation}
Q^T(B\otimes A)Q=P^T(A\otimes B)P,
\end{equation}
where $P$ is the permutation matrix with $P=\overline{P}$, $P^T=P^{-1}$ is given by
\begin{equation}
P=P(m,n)Q,\label{RemPMatrixDefInTermsOfQ}
\end{equation}
which follows from Lemma \ref{LemElemPropertiesCommMatrices} with the commutation matrix $P(m,n)$ since 
\begin{equation}
B\otimes A=P(m,n)^T(A\otimes B)P(m,n).
\end{equation}
\end{remark}

\begin{example}\label{ExPropKroneckerProductOfTwoSchurComplements}
To illustrate Proposition \ref{PropKroneckerProdOfTwoSchurComplements} and the proof, consider the following example which builds off the previous examples \ref{ExLemKroneckerProdSchurComplWithAMatrix} and \ref{ExCorLemKroneckerProdSchurComplWithAMatrix} (just as Proposition \ref{PropKroneckerProdOfTwoSchurComplements} and its proof builds off both Lemma \ref{LemKroneckerProdSchurComplWithAMatrix}, Corollary \ref{CorLemKroneckerProdSchurComplWithAMatrix}, and their proofs). Let
\begin{align*}
A=\left[\begin{array}{c;{2pt/2pt} c c}
A_{11} & A_{12} \\ \hdashline[2pt/2pt]
A_{21} & A_{22}
\end{array}\right]
=
\left[\begin{array}{c;{2pt/2pt} c c}
0 & 2 \\ \hdashline[2pt/2pt]
2 & 4
\end{array}\right],\;\;
B=\left[\begin{array}{c;{2pt/2pt} c c}
B_{11} & B_{12} \\ \hdashline[2pt/2pt]
B_{21} & B_{22}
\end{array}\right]
=
\left[\begin{array}{c;{2pt/2pt} c c}
2 & 3 \\ \hdashline[2pt/2pt]
3 & 5
\end{array}\right].
\end{align*}
Then, as calculated in Example \ref{ExCorLemKroneckerProdSchurComplWithAMatrix}, the matrix
$D=Q^T(B\otimes A)Q$ has the $2\times 2$ block matrix form
\begin{gather*}
D=\left[\begin{array}{c;{2pt/2pt} c c}
D_{11} & D_{12} \\ \hdashline[2pt/2pt]
D_{21} & D_{22}
\end{array}\right]
=
\left[\begin{array}{c;{2pt/2pt} c c}
A\otimes B_{11} & P(1,2)^T(B_{12}\otimes A) \\ \hdashline[2pt/2pt]
(B_{21}\otimes A)P(1,2) & B_{22}\otimes A
\end{array}\right]
\\
=
\left[\begin{array}{c c; {2pt/2pt} c c}
0 & 4 & 0 & 6\\
4 & 8 & 6 & 12\\ \hdashline[2pt/2pt]
0 & 6 & 0 & 10\\
6 & 12 & 10 & 20\\
\end{array}\right].
\end{gather*}
Now $A_{22}=[4]$, $B_{22}=[5]$, $A/A_{22}=[-1]$, and $B/B_{22}=\begin{bmatrix}
\frac{1}{5}
\end{bmatrix}$ are invertible, which implies by Proposition \ref{PropKroneckerProdOfTwoSchurComplements} that $M=D$ is an invertible matrix and it has a $2\times 2$ block form $M=[M_{ij}]_{i,j=1,2}$ such that $M_{22}$ is invertible and $M/M_{22}=A/A_{22}\otimes B/B_{22}$. According to this Proposition \ref{PropKroneckerProdOfTwoSchurComplements} and its proof, we form this block structure $M=[M_{ij}]_{i,j=1,2}$ in the following manner: We first partition the matrix $D_{11}=A\otimes B_{11}$ into a $2\times 2$ block matrix $D_{11}=[D_{ij}]_{i,j=3,4}$ as
\begin{gather*}
D_{11}=\left[\begin{array}{c;{2pt/2pt} c c}
D_{33} & D_{34} \\ \hdashline[2pt/2pt]
D_{44} & D_{44}
\end{array}\right]
=
\left[\begin{array}{c;{2pt/2pt} c c}
A_{11}\otimes B_{11} & A_{12}\otimes B_{11} \\ \hdashline[2pt/2pt]
A_{21}\otimes B_{11} & A_{22}\otimes B_{11}
\end{array}\right]
=A\otimes B_{11}\\
=
\left[\begin{array}{c;{2pt/2pt} c c}
0 & 4 \\ \hdashline[2pt/2pt]
4 & 8
\end{array}\right].
\end{gather*}
This yields a subpartitioning of the $2\times 2$ block matrix $D$ into a $3\times 3$ block matrix as
\begin{gather*}
D =
\left[\begin{array}{c;{2pt/2pt} c c}
D_{11} & D_{12} \\ \hdashline[2pt/2pt]
D_{21} & D_{22}
\end{array}\right]
=
\left[\begin{array}{c c; {2pt/2pt} c}
D_{33} & D_{34} & D_{32} \\
D_{43} & D_{44} & D_{42} \\ \hdashline[2pt/2pt]
D_{23} & D_{24} & D_{22}
\end{array}\right]
\\
=
\left[\begin{array}{c; {2pt/2pt} c}
\begin{matrix} [0] & [4]\\ [4] & [8] \end{matrix} & \begin{matrix} \begin{bmatrix} 0 & 6 \end{bmatrix} \\ \begin{bmatrix} 6 & 12 \end{bmatrix} \end{matrix}   \vspace{0.1cm}\\ \hdashline[2pt/2pt]
\begin{matrix}
\begin{bmatrix}  0 \\ 6 \end{bmatrix} & \begin{bmatrix}  6 \\ 12 \end{bmatrix}
\end{matrix}
&
\begin{bmatrix}  
0 & 10\\
10 & 20
\end{bmatrix}
\end{array}\right]
\end{gather*}
Next, we repartition $D=M=[M_{ij}]_{i,j=1,2}$ as
\begin{gather*}
M=\left[\begin{array}{c;{2pt/2pt} c c}
M_{11} & M_{12} \\ \hdashline[2pt/2pt]
M_{21} & M_{22}
\end{array}\right]
=
\left[\begin{array}{c; {2pt/2pt} c c}
D_{33} & D_{34} & D_{32} \\ \hdashline[2pt/2pt]
D_{43} & D_{44} & D_{42} \\ 
D_{23} & D_{24} & D_{22}
\end{array}\right]
=
\left[\begin{array}{c; {2pt/2pt} c}
\begin{bmatrix}
0
\end{bmatrix} &
\begin{matrix}
\begin{bmatrix}
4
\end{bmatrix} &
\begin{bmatrix}
0 & 6
\end{bmatrix}
\end{matrix}   \vspace{0.1cm}\\ \hdashline[2pt/2pt]
\begin{matrix}
\begin{bmatrix}
4
\end{bmatrix}\\
\begin{bmatrix}
0\\
6
\end{bmatrix}
\end{matrix} &
\begin{matrix}
\begin{bmatrix}
8
\end{bmatrix} & \begin{bmatrix}
6 & 12
\end{bmatrix} \\
\begin{bmatrix}
6 \\
12
\end{bmatrix} & \begin{bmatrix}
0 & 10\\
10 & 12
\end{bmatrix}
\end{matrix}
\end{array}\right]\\
=
\left[\begin{array}{c; {2pt/2pt} c c c}
0 & 4 & 0 & 6\\ \hdashline[2pt/2pt]
4 & 8 & 6 & 12\\ 
0 & 6 & 0 & 10\\
6 & 12 & 10 & 20
\end{array}\right] = D.
\end{gather*}
Then by direct calculation we find that
\begin{gather*}
M/M_{22}=M_{11}-M_{12}M_{22}^{-1}M_{21}=\begin{bmatrix}
0
\end{bmatrix}
-
\begin{bmatrix}
4 & 0 & 6
\end{bmatrix}
\begin{bmatrix}
8 & 6 & 12\\ 
6 & 0 & 10\\
12 & 10 & 20
\end{bmatrix}^{-1}
\begin{bmatrix}
4\\
0\\
6
\end{bmatrix}
=
\begin{bmatrix}
-\frac{1}{5}
\end{bmatrix},\\
A/A_{22}\otimes B/B_{22}=\begin{bmatrix}
-1
\end{bmatrix}
\otimes 
\begin{bmatrix}
\frac{1}{5}
\end{bmatrix}
=
\begin{bmatrix}
-\frac{1}{5}
\end{bmatrix}.
\end{gather*}
We will now conclude this example by considering Remark \ref{RemPMatrix}. From this remark we know that
\begin{gather*}
D=M=Q^T(B\otimes A)Q=P^T(A\otimes B)P,
\end{gather*}
where in this example we have $m=n=2$, $\{E_{ij}:i=1,2,j=1,2\}$ is the standard basis for $C^{2\times 2}$ (as defined in Def.\ \ref{DefCommutationMatrices}), and
\begin{gather*}
P=P(m,n)Q=P(2,2)I_4=P(2,2)=[E_{ij}^T]_{i,j=1}^{2,2}
=
\left[\begin{array}{c;{2pt/2pt} c c}
E_{11}^T & E_{12}^T \\ \hdashline[2pt/2pt]
E_{21}^T & E_{22}^T
\end{array}\right]\\
=
\left[\begin{array}{c c; {2pt/2pt} c c}
1 & 0 & 0 & 0\\
0 & 0 & 1 & 0\\ \hdashline[2pt/2pt]
0 & 1 & 0 & 0\\
0 & 0 & 0 & 1
\end{array}\right].
\end{gather*}
By a direct calculation we find that
\begin{gather*}
P^T(A\otimes B)P=\begin{bmatrix}
1 & 0 & 0 & 0\\
0 & 0 & 1 & 0\\
0 & 1 & 0 & 0\\
0 & 0 & 0 & 1
\end{bmatrix}
\begin{bmatrix}
0 & 0 & 4 & 6\\
0 & 0 & 6 & 10\\
4 & 6 & 8 & 12\\
6 & 10 & 12 & 20
\end{bmatrix}
\begin{bmatrix}
1 & 0 & 0 & 0\\
0 & 0 & 1 & 0\\
0 & 1 & 0 & 0\\
0 & 0 & 0 & 1
\end{bmatrix}\\
=
\begin{bmatrix}
0 & 4 & 0 & 6\\
4 & 8 & 6 & 12\\ 
0 & 6 & 0 & 10\\
6 & 12 & 10 & 20
\end{bmatrix} = Q^T(B\otimes A)Q = D = M.
\end{gather*}.
\end{example}

\begin{lemma}[Matrix factorizations into block canonical forms]\label{LemMatrFactorizBlockForms}
If $B\in
\mathbb{C}
^{n\times n}$ then
\begin{equation}
B=EDF
\end{equation}
for some invertible matrices $E,F\in \mathbb{C}
^{n\times n}$ and block matrix $D\in \mathbb{C}
^{n\times n}$ with one of the following $2\times 2$ block matrix forms:
\begin{align}
D&=
\begin{bmatrix}
D_{11} & 0\\
0 & 0
\end{bmatrix},\;D_{11}=I_r,\; r=\operatorname{rank}(B)\label{LemMatrFactorizBlockForms1stForm};\\
D&=
\begin{bmatrix}
D_{11} & 0 \\
0 & 0
\end{bmatrix},\;D_{11}=\begin{bmatrix}
I_{i_+} & 0 \\
0 & -I_{i_-}
\end{bmatrix},\;i_+,i_-\in \mathbb{N}\cup \{0\},\;i_++i_-=\operatorname{rank}(B)\label{LemMatrFactorizBlockForms2ndForm}.
\end{align}
Moreover, the following statements are true:
\begin{itemize}
    \item[(a)] If $B$ is a real matrix then both matrices $E,F$ can be chosen to be real and the matrix $D$ can also be chosen to have the form (\ref{LemMatrFactorizBlockForms1stForm}). 
    \item[(b)] If $B$ is a symmetric matrix then the matrix $E$ can be chosen so that $E=F^T$ and the matrix $D$ can also be chosen to have the form (\ref{LemMatrFactorizBlockForms1stForm}). 
    \item[(c)] If the hypotheses of (a) and (b) are true then both matrices $E,F$ can chosen to be real and to satisfy $E=F^T$ and the matrix $D$ can also be chosen to have the form (\ref{LemMatrFactorizBlockForms2ndForm}). 
    \item[(d)] If $B$ is a Hermitian matrix then the matrix $E$ can be chosen so that $E=F^*$ and the matrix $D$ can also be chosen to have the form (\ref{LemMatrFactorizBlockForms2ndForm}).
\end{itemize}
\end{lemma}
\begin{proof}
Let $B\in
\mathbb{C}
^{n\times n}$. First, by elementary results in linear algebra on rank, we know that there exists invertible matrices $E,F\in
\mathbb{C}
^{n\times n}$
such that $B=EDF$, where $D\in
\mathbb{C}
^{n\times n}$ has the block form  (\ref{LemMatrFactorizBlockForms1stForm}) such that these matrices are all real if $B$ is
real. On the other hand, if $B$ is symmetric then we take
$E=F^{T}$ in this case (see \cite[Corollary
4.4.4.(c), p. 263 in Sec. 4.4: Unitary congruence and complex symmetric
matrices]{13HJ}, \cite[Sec. 3.0: Introduction and historical remarks in Chap. 3: Singular value inequalities]{91HJ}, \cite[Theorem 3.9 and comments in Sec. 3.4]{14GPP}). If $B$ is real symmetric or Hermitian, then we instead have a
factorization $C=EDF$ with $E,F\in
\mathbb{C}
^{n\times n}$, $E=F^T$ and $F$
real, if $B$ is real symmetric or $E=F^{\ast}$ if $B$ is Hermitian, and $D\in
\mathbb{C}
^{n\times n}$ has the block form (\ref{LemMatrFactorizBlockForms2ndForm}), where $i_+$ and $i_-$ are the number of positive and negative eigenvalues of
$B$, respectively (where its possible to have $i_+=0$ or $i_-=0$ or no zero on the main diagonal in the case $i_++i_-=n$). This completes the proof.
\end{proof}

\begin{proposition}[Scalar product of a Schur complement]\label{PropScalarProdOfASchurCompl}
If $B\in
\mathbb{C}
^{n\times n}$ and $A\in
\mathbb{C}
^{m\times m}$ is a $2\times2$ block matrix
\[
A=
\begin{bmatrix}
A_{11} & A_{12}\\
A_{21} & A_{22}
\end{bmatrix},
\]
such that $A_{22}$ is invertible and $A/A_{22}\in \mathbb{C}^{1\times 1}$
then there exists a $2\times 2$ block matrix $C=[C_{ij}]_{i,j=1,2}$ with $C_{22}$ invertible such that
\begin{equation}
C/C_{22}=A/A_{22}\otimes B=\det(A/A_{22}) B.
\end{equation}
Moreover, if both matrices $A$ and $B$ are real, symmetric, Hermitian, or real and symmetric then the matrix $C$ is a real, symmetric, Hermitian, or real and symmetric, respectively. In addition, if $r=\operatorname{rank}(B)$ and using the factorization of $B$ in Lemma \ref{LemMatrFactorizBlockForms}, i.e., $B=EDF$, where $E,F\in\mathbb{C}^{n\times n}$ are invertible, $D=D_{11}\oplus 0_{n-r}$, and $D_{11}\in \mathbb{C}^{r\times r}$ is invertible with the form (\ref{LemMatrFactorizBlockForms1stForm}) or (\ref{LemMatrFactorizBlockForms2ndForm}), then we can take the matrix $C\in \mathbb{C}^{(mr+n-r)\times (mr+n-r)}$ to be
\begin{align}
C  &=
\left[\begin{array}{c;{2pt/2pt} c c}
C_{11} & C_{12} \\ \hdashline[2pt/2pt]
C_{21} & C_{22}
\end{array}\right] \\
&=
\left[\begin{array}{c;{2pt/2pt} c c}
E & 0\\ \hdashline[2pt/2pt]
0 & I_{(m-1)r}
\end{array}\right]
\left[\begin{array}{c l; {2pt/2pt} c}
A_{11}\otimes D_{11} & 0 & A_{12}\otimes D_{11} \\
0 & 0_{n-r} & 0 \\ \hdashline[2pt/2pt]
A_{21}\otimes D_{11} & 0 & A_{22}\otimes D_{11}
\end{array}\right]
\left[\begin{array}{c;{2pt/2pt} c c}
F & 0\\ \hdashline[2pt/2pt]
0 & I_{(m-1)r}
\end{array}\right].
\end{align}
\end{proposition}
\begin{proof}
Suppose $B\in
\mathbb{C}
^{n\times n}$ and $A\in
\mathbb{C}
^{m\times m}$ is a $2\times2$ block matrix $A=[A_{ij}]_{i,j=1,2}$ such that $A_{22}$ is invertible and $A/A_{22}\in \mathbb{C}^{1\times 1}$. Then, $\det (A/A_{22}) B$ [i.e., the scalar multiplication of the scalar $\det (A/A_{22})$ with the matrix $B$] is equal to the Kronecker product of the $1\times 1$ matrix $A/A_{22}$ with the matrix $B$, that is,
\begin{equation}
A/A_{22}=\begin{bmatrix}
\det (A/A_{22})
\end{bmatrix},\;\;A/A_{22}\otimes B=\det (A/A_{22}) B.
\end{equation}
Now, we would like to apply Lemma \ref{LemKroneckerProdSchurComplWithAMatrix} to this Kronecker product $A/A_{22}\otimes B$, but the hypothesis that $B$ is invertible need not be satisfied. On the other hand, by Lemma \ref{LemMatrFactorizBlockForms} we have the factorization $B=EDF$, where $E,F\in\mathbb{C}^{n\times n}$ are invertible matrices and $D=D_{11}\oplus 0_{n-r}$ is the matrix direct sum of an invertible matrix $D_{11}$ [where $D_{11}=I_r$ or $D_{11}=I_{i_+}\oplus (-I_{i-})$ are both $r\times r$ matrix and $r=\operatorname{rank}(B)$] and the $(n-r)\times (n-r)$ zero matrix $0_{n-r}$ (with no zero matrix present if $B$ is invertible) in which statements (a)-(d) in that lemma are true. Thus, we have
\begin{gather*}
A/A_{22}\otimes B=\det (A/A_{22})B=E\det (A/A_{22})DF=E[(\det (A/A_{22}) D_{11})\oplus 0_{n-r}]F\\
=E[(A/A_{22}\otimes D_{11})\oplus 0_{n-r}]F.
\end{gather*}
By Lemma \ref{LemKroneckerProdSchurComplWithAMatrix},
\begin{equation}
G/G_{22} = A/A_{22}\otimes D_{11},
\end{equation}
where
\begin{equation}
G =
\begin{bmatrix}
G_{11} & G_{12}\\
G_{21} & G_{22}
\end{bmatrix}
=
\begin{bmatrix}
A_{11}\otimes D_{11} & A_{12}\otimes D_{11}\\
A_{21}\otimes D_{11} & A_{22}\otimes D_{11}
\end{bmatrix}.
\end{equation}
By Lemma \ref{LemShortedMatricesAreSchurCompl},
\begin{equation}
H/H_{22}=G/G_{22}\oplus 0_{n-r},
\end{equation}
where
\begin{align} 
H & =
\left[\begin{array}{c;{2pt/2pt} c c}
H_{11} & H_{12} \\ \hdashline[2pt/2pt]
H_{21} & H_{22}
\end{array}\right]
=
\left[\begin{array}{c c; {2pt/2pt} c}
G_{11} & 0 & G_{12} \\
0 & 0_{n-r} & 0 \\ \hdashline[2pt/2pt]
G_{21} & 0 & G_{22}
\end{array}\right].
\end{align}
By Proposition \ref{PropMatrixMultOfSchurCompl},
\begin{equation}
C/C_{22}=E(H/H_{22})F 
\end{equation}
where
\begin{equation}
C = \begin{bmatrix}
EH_{11}F & EH_{12}\\
H_{21}F & H_{22}
\end{bmatrix}
=
\begin{bmatrix}
E & 0\\
0 & I_{(m-1)r}
\end{bmatrix}
\begin{bmatrix}
H_{11} & H_{12}\\
H_{21} & H_{22}
\end{bmatrix}
\begin{bmatrix}
F & 0\\
0 & I_{(m-1)r}
\end{bmatrix}.
\end{equation}
Therefore, putting this all together we have proven
\begin{align*}
\det(A/A_{22}) B&=A/A_{22}\otimes B\\
&=E[(A/A_{22}\otimes D_{11})\oplus 0_{n-r}]F\\
&=E(G/G_{22}\oplus 0_{n-r})F\\
&=E(H/H_{22})F=C/C_{22},
\end{align*}
and $C$ has the desired properties. This completes the proof.
\end{proof}

\paragraph{On Kronecker products of linear matrix pencils}\label{SecOnKroneckerProductsOfLinearMatrixPencils}

This section contains the main technical portion of the statements (and their proofs) needed in the proof of the Bessmertny\u{\i} Realization Theorem (Theorem \ref{ThmBessmRealiz}) that pertain to Kronecker products of linear matrix pencils and realizations.

The main result in this section is Proposition \ref{PropRealizOfKroneckerProdOfRealiz}. As the proof of it is rather technical (as it builds on many other basic building blocks in this section and the previous sections), we provide a flow diagram for the proof in Figure \ref{FigFlowDiagramProofKroneckerProdRealiz}.

\begin{lemma}[Realization of squares]\label{LemRealizOfSquares}
The $\mathbb{C}^{1\times 1}$-valued function $f(z_1)=\begin{bmatrix}
z_1^2
\end{bmatrix}$ of an independent variable $z_1$ has a Bessmertny\u{\i} realization, i.e.,
\begin{equation}
f(z_1)=[z_1^2]=A(z_1)/A_{22}(z_1),
\end{equation}
with linear matrix pencil
\begin{equation}
A(z_1)=A_0+z_1A_1=\begin{bmatrix}
A_{11}(z_1) & A_{12}(z_1)\\
A_{21}(z_1) & A_{22}(z_1)
\end{bmatrix},
\end{equation}
where
\begin{equation}
A_0=\begin{bmatrix}
0 & 0\\
0 & -1
\end{bmatrix},\;\;
A_1=\begin{bmatrix}
0 & 1\\
1 & 0
\end{bmatrix},
\end{equation}
and
\begin{align} 
\left[\begin{array}{c;{2pt/2pt} c}
A_{11}(z_1) & A_{12}(z_1) \\ \hdashline[2pt/2pt]
A_{21}(z_1) & A_{22}(z_1)
\end{array}\right]
=
\left[\begin{array}{c; {2pt/2pt} c}
0 & z_1 \\ \hdashline[2pt/2pt]
z_1 & -1
\end{array}\right]
\end{align}
such that
\begin{equation}
\det A_{22}(z_1)=-1\not \equiv 0.
\end{equation}
Moreover, the matrices $A_j$ are real and symmetric, i.e., $A_j=\overline{A_j}=A_j^{T}$ for $j=0,1$.
\end{lemma}
\begin{proof}
The proof follows immediately from the definition of $A(z_1), A_0, A_1$ and the calculation
\begin{align*}
A(z_1)/A_{22}(z_1)=[0]-[z_1][-1]^{-1}[z_1]=-[z_1][-1][z_1]=-[-z_1^2]=[z_1^2].
\end{align*}
\end{proof}

\begin{lemma}[Realization of the product of two independent variables]\label{LemRealizProdTwoIndepVar}
The $\mathbb{C}^{1\times 1}$-valued function $f(z_1,z_2)=\begin{bmatrix}
z_1z_2
\end{bmatrix}$ of two independent variables $z_1$ and $z_2$ has a Bessmertny\u{\i} realization, i.e.,
\begin{equation}
f(z_1,z_2)=[z_1z_2]=A(z_1,z_2)/A_{22}(z_1,z_2),
\end{equation}
with linear matrix pencil
\begin{equation}
A(z_1,z_2)=A_0+z_1A_1+z_2A_2=\begin{bmatrix}
A_{11}(z_1,z_2) & A_{12}(z_1,z_2)\\
A_{21}(z_1,z_2) & A_{22}(z_1,z_2)
\end{bmatrix},
\end{equation}
where
\begin{equation}
A_0=\begin{bmatrix}
0 & 0 & 0\\
0 & -\frac{1}{4} & 0\\
0 & 0 & \frac{1}{4}
\end{bmatrix},\;
A_1=\begin{bmatrix}
0 & \frac{1}{4} & -\frac{1}{4}\\
\frac{1}{4} & 0 & 0\\
-\frac{1}{4} & 0 & 0
\end{bmatrix},\;
A_2=\begin{bmatrix}
0 & \frac{1}{4} & \frac{1}{4}\\
\frac{1}{4} & 0 & 0\\
\frac{1}{4} & 0 & 0
\end{bmatrix},
\end{equation}
and
\begin{align} 
\left[\begin{array}{c;{2pt/2pt} c c}
A_{11}(z_1,z_2) & A_{12}(z_1,z_2) \\ \hdashline[2pt/2pt]
A_{21}(z_1,z_2) & A_{22}(z_1,z_2)
\end{array}\right]
=
\left[\begin{array}{c; {2pt/2pt} c c}
0 & \frac{1}{4}(z_1+z_2) & -\frac{1}{4}(z_1-z_2) \\\hdashline[2pt/2pt]
\frac{1}{4}(z_1+z_2) & -\frac{1}{4} & 0 \\ 
-\frac{1}{4}(z_1-z_2) & 0 & \frac{1}{4}
\end{array}\right]
\end{align}
such that
\begin{equation}
\det A_{22}(z_1,z_2)=\frac{-1}{16}\not \equiv 0.
\end{equation}
Moreover, the matrices $A_j$ are real and symmetric, i.e., $A_j=\overline{A_j}=A_j^{T}$ for $j=0,1,2$.
\end{lemma}
\begin{proof}
The statement follows by applying Lemma \ref{LemRealizOfSquares}, Lemma \ref{LemScalarMultiSchurCompl}, and Proposition \ref{PropSumOfSchurComps} in succession:
\begin{align*}
[z_1z_2]&=\left[\frac{1}{4}(z_1+z_2)^2-\frac{1}{4}(z_1-z_2)^2\right]=\frac{1}{4}\left[(z_1+z_2)^2\right]+\frac{-1}{4}\left[(z_1-z_2)^2\right]\\
&=\frac{1}{4}\left(\left.\left[\begin{array}{c; {2pt/2pt} c}
0 & z_1+z_2 \\ \hdashline[2pt/2pt]
z_1+z_2 & -1
\end{array}\right]
\right/
\begin{bmatrix}
 -1
\end{bmatrix}\right)\\
&+\frac{-1}{4}\left(\left.\left[\begin{array}{c; {2pt/2pt} c}
0 & z_1-z_2 \\ \hdashline[2pt/2pt]
z_1-z_2 & -1
\end{array}\right]
\right/
\begin{bmatrix}
 -1
\end{bmatrix}\right)
\\
&=\left.\left[\begin{array}{c; {2pt/2pt} c}
0 & \frac{1}{4}(z_1+z_2) \\ \hdashline[2pt/2pt]
\frac{1}{4}(z_1+z_2) & -\frac{1}{4}
\end{array}\right]
\right/
\begin{bmatrix}
 -\frac{1}{4}
\end{bmatrix}\\
&+\left.\left[\begin{array}{c; {2pt/2pt} c}
0 & -\frac{1}{4}(z_1-z_2) \\ \hdashline[2pt/2pt]
-\frac{1}{4}(z_1-z_2) & \frac{1}{4}
\end{array}\right]
\right/
\begin{bmatrix}
 \frac{1}{4}
\end{bmatrix}\\
&=
\left.\left[\begin{array}{c; {2pt/2pt} c c}
0 & \frac{1}{4}(z_1+z_2) & -\frac{1}{4}(z_1-z_2) \\\hdashline[2pt/2pt]
\frac{1}{4}(z_1+z_2) & -\frac{1}{4} & 0 \\ 
-\frac{1}{4}(z_1-z_2) & 0 & \frac{1}{4}
\end{array}\right]
\right/
\begin{bmatrix}
-\frac{1}{4} & 0 \\ 
0 & \frac{1}{4}
\end{bmatrix}.
\end{align*}
The proof follows immediately from this representation.
\end{proof}

\begin{lemma}[Realization of Kronecker products: Part I]\label{LemRealizOfKroneckerProdsPart1}
If $A(z)$ and $B(w)$ are two linear matrix pencils
\begin{align}
A\left(  z\right)   &  = A_0+\sum_{i=1}^{s}
z_{i}A_{i},\;\;B\left(  w\right)  = B_0 + \sum_{j=1}^{t}
w_{j}B_{j}
\end{align}
where $A_{i}\in
\mathbb{C}^{m\times m}$ (for $i=0,\ldots, s$) and $B_{j}\in
\mathbb{C}^{n\times n}$ (for $j=0,\ldots, t$) then there exists a linear matrix pencil $C(z,w)$ in $2\times 2$ block form
\begin{align}\label{LemRealizOfKroneckerProdsPart1CPencil}
C\left(  z,w\right)   &  = C_0+\sum_{i=1}^{s}
z_{i}C_{i}+\sum_{j=1}^{t}
w_{j}C_{s+j}
=
\begin{bmatrix}
C_{11}(z,w) & C_{12}(z,w)\\
C_{21}(z,w) & C_{22}(z,w)
\end{bmatrix},
\end{align}
with $\det C_{22}(z,w) \not \equiv 0$, such that
\begin{align}
C(z,w)/C_{22}(z,w) = A(z)\otimes B(w).
\end{align}
Moreover, the following statements are true:
\begin{itemize}
\item[(a)] If all the matrices $A_i$ and $B_j$ (for $i=0,\ldots, s$ and $j=0,\ldots, t$) are real then one can choose all the matrices $C_k$ (for $k=0,\ldots, s+t$) to be real.
\item[(b)] If all the matrices $A_i$ and $B_j$ (for $i=0,\ldots, s$ and $j=0,\ldots, t$) are symmetric then one can choose all the matrices $C_k$ (for $k=0,\ldots, s+t$) to be symmetric.
\item[(c)] If all the matrices $A_i$ and $B_j$ (for $i=0,\ldots, s$ and $j=0,\ldots, t$) are Hermitian then one can choose all the matrices $C_k$ (for $k=0,\ldots, s+t$) to be Hermitian.
\item[(d)] If any combination of the (a)-(d) hypotheses are true then the matrices $C_k$ can be chosen to satisfy the same combination of conclusions.
\end{itemize}
\end{lemma}
\begin{proof}
It follows from Lemma \ref{LemKroneckerProdSchurComplWithAMatrix} and linearity properties of the Kronecker product $\otimes$ that
\begin{gather}
M(z,w)=A(z)\otimes B(w)\\
= A_0\otimes B_0+\sum_{i=1}^sz_i(A_i\otimes B_0)+\sum_{j=1}^tw_j(A_0\otimes B_j)+\sum_{j=1}^t\sum_{i=1}^s
\left(  z_{i}w_{j}\right) \left(  A_{i}\otimes B_{j}\right).
\end{gather}
The first part of the sum for $M(z,w)$, i.e., $A_0\otimes B_0+\sum_{i=1}^sz_i(A_i\otimes B_0)+\sum_{j=1}^tw_j(A_0\otimes B_j)$ is a linear matrix pencil and is already realized. The second part of the sum, i.e., $\sum_{j=1}^t\sum_{i=1}^s
\left(  z_{i}w_{j}\right) \left(  A_{i}\otimes B_{j}\right)$ is realizable by Proposition \ref{PropSumOfSchurComps}, Proposition \ref{PropScalarProdOfASchurCompl}, and Lemma \ref{LemRealizProdTwoIndepVar}. Hence, the sum of these two parts, which is $M(z,w)$, is realizable by Lemma \ref{LemSumSchurComplPlusMatrix}. This proves that $C(z,w)/C_{22}(z,w)=M(z,w)=A(z)\otimes B(w)$ for some linear matrix pencil $C(z,w)$ in the form (\ref{LemRealizOfKroneckerProdsPart1CPencil}) with $\det C_{22}(z,w)\not\equiv 0$. This completes the first part of the proof. The rest of the proof of statements (a)-(d) follow immediately from these results and the elementary properties of the Kronecker product $\otimes$, namely, that $\overline {A_i\otimes B_j}=\overline{A_i}\otimes \overline{B_j}$, $(A_i\otimes B_j)^T=A_i^T\otimes B_j^T$, and $(A_i\otimes B_j)^*=A_i^*\otimes B_j^*$. This completes the proof.
\end{proof}

\begin{lemma}[Realization of Kronecker products: Part II]\label{LemRealizOfKroneckerProdsPart2}
If $B(w)$ is a linear matrix pencil
\begin{align}
B\left(  w\right)  = B_0 + \sum_{j=1}^{t}
w_{j}B_{j}
\end{align}
and $A(z)$ is a linear matrix pencil in $2\times 2$ block form
\begin{align}
A\left(  z\right)   &  = A_0+\sum_{i=1}^{s}
z_{i}A_{i}
=
\begin{bmatrix}
A_{11}(z) & A_{12}(z)\\
A_{21}(z) & A_{22}(z)
\end{bmatrix}
\end{align}
where $A_{i}\in
\mathbb{C}^{m\times m}$ (for $i=0,\ldots, s$) and $B_{j}\in
\mathbb{C}^{n\times n}$ (for $j=0,\ldots, t$) such that $\det A_{22}\left(z\right) \not \equiv 0$ and $\det B\left(  w\right)  \not \equiv 0$, then there exists a linear matrix pencil $D(z,w)$ in $2\times 2$ block form
\begin{align}\label{LemRealizOfKroneckerProdsPart2DPencil}
D\left(  z,w\right)   &  = D_0+\sum_{i=1}^{s}
z_{i}D_{i}+\sum_{j=1}^{t}
w_{j}D_{s+j}
=
\begin{bmatrix}
D_{11}(z) & D_{12}(z)\\
D_{21}(z) & D_{22}(z)
\end{bmatrix},
\end{align}
with $\det D_{22}(z,w) \not \equiv 0$, such that
\begin{align}
D(z,w)/D_{22}(z,w) = A(z)/A_{22}(z)\otimes B(w).
\end{align}
Moreover, the following statements are true:
\begin{itemize}
\item[(a)] If all the matrices $A_i$ and $B_j$ (for $i=0,\ldots, s$ and $j=0,\ldots, t$) are real then one can choose all the matrices $D_k$ (for $k=0,\ldots, s+t$) to be real.
\item[(b)] If all the matrices $A_i$ and $B_j$ (for $i=0,\ldots, s$ and $j=0,\ldots, t$) are symmetric then one can choose all the matrices $D_k$ (for $k=0,\ldots, s+t$) to be symmetric.
\item[(c)] If all the matrices $A_i$ and $B_j$ (for $i=0,\ldots, s$ and $j=0,\ldots, t$) are Hermitian then one can choose all the matrices $D_k$ (for $k=0,\ldots, s+t$) to be Hermitian.
\item[(d)] If any combination of the (a)-(d) hypotheses are true then the matrices $D_k$ can be chosen to satisfy the same combination of conclusions.
\end{itemize}
\end{lemma}
\begin{proof}
By Lemma \ref{LemKroneckerProdSchurComplWithAMatrix} we know that
\begin{align}
M(z,w)/M_{22}(z,w)=A(z)/A_{22}(z)\otimes B(w),
\end{align}
where
\begin{align}
M(z,w)&=\begin{bmatrix}
M_{11}(z,w) & M_{12}(z,w)\\
M_{21}(z,w) & M_{22}(z,w)
\end{bmatrix}\\
&=
\begin{bmatrix}
A_{11}(z)\otimes B(w) & A_{12}(z)\otimes B(w)\\
A_{21}(z)\otimes B(w) & A_{22}(z)\otimes B(w)
\end{bmatrix}
=A(z)\otimes B(w)
\end{align}
 and $\det M_{22}(z,w)\not \equiv 0.$ By Lemma \ref{LemRealizOfKroneckerProdsPart1} we know that there exists a linear matrix pencil $C(z,w)$ in the $2\times 2$ block form (\ref{LemRealizOfKroneckerProdsPart1CPencil}) with $\det C_{22}(z,w)\not \equiv 0$ such that
 \begin{align}
C(z,w)/C_{22}(z,w)=A(z)\otimes B(w)=M(z,w).
 \end{align}
It now follows from this that
\begin{align}
[C(z,w)/C_{22}(z,w)]/[C(z,w)/C_{22}(z,w)]_{22}&=M(z,w)/M_{22}(z,w)\\
&=A(z)/A_{22}(z)\otimes B(w),
\end{align}
and thus by Proposition \ref{PropComposSchurComplem} the statement is proven since by this proposition we can take $D(z,w)=C(z,w)$ [although with a possibly different $2\times 2$ block form described in that proposition in which $\det D_{22}(z,w)\not \equiv 0$]. This proves the first part of the lemma and statements (a)-(d) of this lemma follow immediately from this representation of $D(z,w)$ and Lemma \ref{LemRealizOfKroneckerProdsPart1}.
\end{proof}

\begin{example}
We will now work out a concrete example to illustrate our approach to the realization problem which uses Lemma \ref{LemRealizOfKroneckerProdsPart2} and its proof. Consider the rational $\mathbb{C}^{1\times 1}$-valued function of the two independent variables $z_1,w_1$ defined by
\begin{align*}
f(z_1,w_1)=\left[ \frac{1}{3+3z_{1}}\left( 9+55w_{1}\right) \right].
\end{align*}
Then this can be written in terms of a Kronecker product as
\begin{align*}
\left[ \frac{1}{3+3z_{1}}\left( 9+55w_{1}\right) \right]=[3+3z_{1}]^{-1}\otimes [9+55w_{1}]=A(z)/A_{22}(z)\otimes B(w),
\end{align*}
where
\begin{align*}
A(z)&=A_0+z_1A_1=
\begin{bmatrix}
A_{11}(z) & A_{12}(z)\\
A_{21}(z) & A_{22}(z)\\
\end{bmatrix}
=
\left[\begin{array}{c; {2pt/2pt} c}
0 & 1 \\\hdashline[2pt/2pt]
1 & -(3+3z_{1})
\end{array}\right],\\
A_0&=\begin{bmatrix}
0 & 1\\
1 & -3
\end{bmatrix},\;\;
A_1=\begin{bmatrix}
0 & 0\\
0 & -3
\end{bmatrix},\\
B(w)&=B_0+w_1B_1=[9+55w_{1}],\;\;B_0=[9],\;\;B_1=[55].
\end{align*}
By Lemma \ref{LemKroneckerProdSchurComplWithAMatrix} we know that
\begin{align*}
M(z,w)/M_{22}(z,w)=A(z)/A_{22}(z)\otimes B(w),
\end{align*}
where 
\begin{align*}
M(z,w)&=\begin{bmatrix}
M_{11}(z,w) & M_{12}(z,w)\\
M_{21}(z,w) & M_{22}(z,w)
\end{bmatrix}
=
\begin{bmatrix}
A_{11}(z)\otimes B(w) & A_{12}(z)\otimes B(w)\\
A_{21}(z)\otimes B(w) & A_{22}(z)\otimes B(w)
\end{bmatrix}\\
&=A(z)\otimes B(w),\\
M_{11}(z,w)&=A_{11}(z)\otimes B(w)=[0],\\
M_{12}(z,w)&=A_{12}(z)\otimes B(w)=[9+55w_1]=A_{21}(z)\otimes B(w)=M_{21}(z,w),\\
M_{22}(z,w)&=A_{22}(z)\otimes B(w)=[-(3+3z_{1})(9+55w_1)]
\end{align*}
 and $$\det M_{22}(z,w)=\det[A_{22}(z)\otimes B(w)]=-(3+3z_{1})(9+55w_1) \not \equiv 0.$$
By Lemma \ref{LemRealizOfKroneckerProdsPart1} we know that there exists a linear matrix pencil $C(z,w)$ in the $2\times 2$ block form (\ref{LemRealizOfKroneckerProdsPart1CPencil}) with $\det C_{22}(z,w)\not \equiv 0$ such that
 \begin{align}
C(z,w)/C_{22}(z,w)=A(z)\otimes B(w)=M(z,w).
 \end{align}
Let us now calculate this $C(z,w)$ using the method described in the proof of Lemma \ref{LemRealizOfKroneckerProdsPart1}. First, we have
\begin{align*}
M(z,w)&=A(z)\otimes B(w)\\
&= A_0\otimes B_0+z_1(A_1\otimes B_0)+w_1(A_0\otimes B_1)+
\left(  z_{1}w_{1}\right) \left(  A_{1}\otimes B_{1}\right),\\
A_0\otimes B_0&=\begin{bmatrix}
0 & 1\\
1 & -3
\end{bmatrix}
\otimes
[9]=\begin{bmatrix}
0 & 9\\
9 & -27
\end{bmatrix},\\
A_1\otimes B_0&=\begin{bmatrix}
0 & 0\\
0 & -3
\end{bmatrix}
\otimes
[9]=\begin{bmatrix}
0 & 0\\
0 & -27
\end{bmatrix},\\
A_0\otimes B_1&=\begin{bmatrix}
0 & 1\\
1 & -3
\end{bmatrix}
\otimes
[55]=\begin{bmatrix}
0 & 55\\
55 & -165
\end{bmatrix},\\
A_1\otimes B_1&=\begin{bmatrix}
0 & 0\\
0 & -3
\end{bmatrix}
\otimes
[55]=\begin{bmatrix}
0 & 0\\
0 & -165
\end{bmatrix}.
\end{align*}
The first part of the sum for $M(z,w)$, i.e., 
\begin{gather*}
A_0\otimes B_0+z_1(A_1\otimes B_0)+w_1(A_0\otimes B_1)\\
=\begin{bmatrix}
0 & 9\\
9 & -27
\end{bmatrix}+z_1\begin{bmatrix}
0 & 0\\
0 & -27
\end{bmatrix}+w_1\begin{bmatrix}
0 & 55\\
55 & -165
\end{bmatrix}\\
=\begin{bmatrix}
0 & 9+55w_1\\
9+55w_1 & -27-27z_1-165w_1
\end{bmatrix},
\end{gather*}
is a linear matrix pencil and is already realized. The second part of the sum, i.e., $\left(  z_{1}w_{1}\right) \left(  A_{1}\otimes B_{1}\right)$, is realizable by Lemma \ref{LemRealizProdTwoIndepVar} and Proposition \ref{PropScalarProdOfASchurCompl} (in fact, for this example Lemma \ref{LemShortedMatricesAreSchurCompl} could be used instead of the latter proposition to speed up the calculation), which we can calculate as
\begin{gather*}
\left(  z_{1}w_{1}\right) \left(  A_{1}\otimes B_{1}\right)=\left(  z_{1}w_{1}\right) \left(  A_{1}\otimes B_{1}\right)=\left(  z_{1}w_{1}\right)\begin{bmatrix}
0 & 0\\
0 & -165
\end{bmatrix}\\
=\left(\left.\left[\begin{array}{c; {2pt/2pt} c c}
0 & \frac{1}{4}(z_1+w_1) & -\frac{1}{4}(z_1-w_1) \\\hdashline[2pt/2pt]
\frac{1}{4}(z_1+w_1) & -\frac{1}{4} & 0 \\ 
-\frac{1}{4}(z_1-w_1) & 0 & \frac{1}{4}
\end{array}\right]
\right/
\begin{bmatrix}
-\frac{1}{4} & 0 \\ 
0 & \frac{1}{4}
\end{bmatrix}\right)\otimes \begin{bmatrix}
0 & 0\\
0 & -165
\end{bmatrix}\\
=
\left.\left[\begin{array}{c c; {2pt/2pt} c c}
0 & 0 & 0 & 0 \\ 
0 & 0 & \frac{-165}{4}\left( z_{1}+w_{1}\right)  & \frac{165}{4}\left(
z_{1}-w_{1}\right)  \\ \hdashline[2pt/2pt]
0 & \frac{-165}{4}\left( z_{1}+w_{1}\right)  & \frac{165}{4} & 0 \\ 
0 & \frac{165}{4}\left( z_{1}-w_{1}\right)  & 0 & \frac{-165}{4}%
\end{array}\right]
\right/ 
\begin{bmatrix}
\frac{165}{4} & 0 \\ 
0 & \frac{-165}{4}%
\end{bmatrix}.
\end{gather*}
Hence, the sum of these two parts, which is $M(z,w)$, is realizable by Lemma \ref{LemSumSchurComplPlusMatrix}, which we can calculate as
\begin{gather*}
C(z,w)/C_{22}(z,w)=A(z)\otimes B(w)=M(z,w),\\
C(z,w)=\begin{bmatrix}
C_{11}(z,w) & C_{12}(z,w)\\
C_{21}(z,w) & C_{22}(z,w)
\end{bmatrix}\\
=\left[\begin{array}{c c; {2pt/2pt} c c}
0 & 9+55w_1 & 0 & 0 \\ 
9+55w_1 & -27-27z_1-165w_1 & \frac{-165}{4}\left( z_{1}+w_{1}\right)  & \frac{165}{4}\left(
z_{1}-w_{1}\right)  \\ \hdashline[2pt/2pt]
0 & \frac{-165}{4}\left( z_{1}+w_{1}\right)  & \frac{165}{4} & 0 \\ 
0 & \frac{165}{4}\left( z_{1}-w_{1}\right)  & 0 & \frac{-165}{4}%
\end{array}\right],\\
\det C_{22}(z,w)=\left(\frac{165}{4}\right)\left(\frac{-165}{4}\right)\not \equiv 0.
\end{gather*}
It now follows that we have
\begin{align*}
[C(z,w)/C_{22}(z,w)]/[C(z,w)/C_{22}(z,w)]_{22}=A(z)/A_{22}(z)\otimes B(w)
\end{align*}
and, by Proposition \ref{PropComposSchurComplem},
\begin{align*}
D(z,w)/D_{22}(z,w)=[C(z,w)/C_{22}(z,w)]/[C(z,w)/C_{22}(z,w)]_{22},
\end{align*}
where
\begin{align*}
D(z,w)=C(z,w),
\end{align*}
with the $2\times 2$ block form
\begin{align*}
D(z,w)&=\left[\begin{array}{c;{2pt/2pt} c}
D_{11}(z_1,w_1) & D_{12}(z_1,w_1) \\ \hdashline[2pt/2pt]
D_{21}(z_1,w_1) & D_{22}(z_1,w_1)
\end{array}\right]\\
&=
\left[\begin{array}{c; {2pt/2pt} c c c}
0 & 9+55w_1 & 0 & 0 \\ \hdashline[2pt/2pt]
9+55w_1 & -27-27z_1-165w_1 & \frac{-165}{4}\left( z_{1}+w_{1}\right)  & \frac{165}{4}\left(
z_{1}-w_{1}\right)  \\ 
0 & \frac{-165}{4}\left( z_{1}+w_{1}\right)  & \frac{165}{4} & 0 \\ 
0 & \frac{165}{4}\left( z_{1}-w_{1}\right)  & 0 & \frac{-165}{4}%
\end{array}\right]
\end{align*}
and
\begin{align*}
\det D_{22}(z_1,w_1)=\left(\frac{165}{4}\right)^2(3+3z_1)(9+55w_1)\not \equiv 0.
\end{align*}
Therefore, $f(z,w)=\left[ \frac{1}{3+3z_{1}}\left( 9+55w_{1}\right) \right]=D(z,w)/D_{22}(z,w)$ has the desired Bessmertny\u{\i} realization with the linear matrix pencil $D(z,w)$ given by
\begin{align*}
D(z,w)&=D_0+z_1D_1+w_1D_2,\\
D_0&= \begin{bmatrix}
0 & 9 & 0 & 0 \\
9 & -27 & 0  & 0  \\ 
0 & 0  & \frac{165}{4} & 0 \\ 
0 & 0  & 0 & \frac{-165}{4}
\end{bmatrix},\\
D_1&=\begin{bmatrix}
0 & 0 & 0 & 0 \\
0 & -27 & \frac{-165}{4}  & \frac{165}{4}  \\ 
0 & \frac{-165}{4}  & 0 & 0 \\ 
0 & \frac{165}{4}  & 0 & 0
\end{bmatrix},\\
D_2&=\begin{bmatrix}
0 & 0 & 0 & 0 \\
0 & -165 & \frac{-165}{4}  & \frac{-165}{4}  \\ 
0 & \frac{-165}{4}  & 0 & 0 \\ 
0 & \frac{-165}{4}  & 0 & 0
\end{bmatrix},
\end{align*}
in which $D_0,D_1,D_2$ are all real and symmetric matrices.
\end{example}

\begin{remark}
Lemma \ref{LemRealizOfKroneckerProdsPart2} can also be extended to the realization $A(z)\otimes B(w)/B_{22}(w)$ in a similar fashion with a slight nuance. First we have $A(z)\otimes B(w)/B_{22}(w)=P^T(B(w)/B_{22}(w)\otimes A(z))P$, with permutation matrix $P$ independent of the variables $z$ and $w$. Then the desired result follows from Lemma \ref{LemRealizOfKroneckerProdsPart2} together with Proposition \ref{PropMatrixMultOfSchurCompl}. This is similar to Corollary \ref{CorLemKroneckerProdSchurComplWithAMatrix} with Remark \ref{RemPMatrix}.
\end{remark}

\begin{proposition}[Realization of the Kronecker product of realizations]\label{PropRealizOfKroneckerProdOfRealiz}
If $A(z)$ and $B(w)$ are two linear matrix pencils in $2\times 2$ block form
\begin{align*}
A\left(  z\right)   &  = A_0+\sum_{i=1}^{s}
z_{i}A_{i}
=
\begin{bmatrix}
A_{11}(z) & A_{12}(z)\\
A_{21}(z) & A_{22}(z)
\end{bmatrix},\\
B\left(  w\right)   &  = B_0 + \sum_{j=1}^{t}
w_{j}B_{j}
=
\begin{bmatrix}
B_{11}(z) & B_{12}(z)\\
B_{21}(z) & B_{22}(z)
\end{bmatrix},
\end{align*}
where $A_{i}\in
\mathbb{C}^{m\times m}$ (for $i=0,\ldots, s$), $B_{j}\in
\mathbb{C}^{n\times n}$ (for $j=0,\ldots, t$) such that $\det A_{22}\left(z\right) \not \equiv 0$, $\det B_{22}\left(  w\right)  \not \equiv 0$, $\det A(z)/A_{22}(z) \not \equiv 0$, and $\det B(w)/B_{22}(w) \not \equiv 0$ then there exists a linear matrix pencil
\begin{align}
D\left(  z,w\right)   &  = D_0+\sum_{i=1}^{s}
z_{i}D_{i}+\sum_{j=1}^{t}
w_{j}D_{s+j}
=
\begin{bmatrix}
D_{11}(z) & D_{12}(z)\\
D_{21}(z) & D_{22}(z)
\end{bmatrix},
\end{align}
with $\det D_{22}(z,w) \not \equiv 0$, such that
\begin{align}
D(z,w)/D_{22}(z,w) = A(z)/A_{22}(z)\otimes B(w)/B_{22}(w).
\end{align}
Moreover, the following statements are true:
\begin{itemize} 
\item[(a)] If all the matrices $A_i$ and $B_j$ (for $i=0,\ldots, s$ and $j=0,\ldots, t$) are real then one can choose all the matrices $D_k$ (for $k=0,\ldots, s+t$) to be real.
\item[(b)] If all the matrices $A_i$ and $B_j$ (for $i=0,\ldots, s$ and $j=0,\ldots, t$) are symmetric then one can choose all the matrices $D_k$ (for $k=0,\ldots, s+t$) to be symmetric.
\item[(c)] If all the matrices $A_i$ and $B_j$ (for $i=0,\ldots, s$ and $j=0,\ldots, t$) are Hermitian then one can choose all the matrices $D_k$ (for $k=0,\ldots, s+t$) to be Hermitian.
\item[(d)] If any combination of the (a)-(d) hypotheses are true then the matrices $D_k$ can be chosen to satisfy the same combination of conclusions.
\end{itemize}
\end{proposition}
\begin{figure}[H]
\caption{\;\;Flow  diagram  for  the  proof of the realization of the Kronecker product of realizations (see Proposition \ref{PropRealizOfKroneckerProdOfRealiz}).}\label{FigFlowDiagramProofKroneckerProdRealiz}
\begin{tikzpicture}
[scale=0.68, 
transform shape,
blockone/.style ={rectangle, draw, text width=15em,align=center, minimum height=1em},
blocktwo/.style ={rectangle, draw, text width=25em,align=center, minimum height=1em},
blockthree/.style ={rectangle, draw, text width=10em,align=center, minimum height=1em},
blockfour/.style ={rectangle, draw, text width=25em,align=center, minimum height=4.5em},
R/.style ={draw, circle, inner sep=5pt},
Rk/.style ={draw, circle, inner sep=2pt},
tpn/.style ={draw, circle, inner sep=6pt}
]
\draw (-4.5,-3.1) node[]
{$\boldsymbol{ A(z) \otimes B(w)}$};
\draw (-4.5, -3.8) node[] {
\scriptsize $\boldsymbol{ = A_0 \otimes B_0
+\sum\limits_{i=1}^sz_i(A_i\otimes B_0)+\sum\limits_{j=1}^tw_j(A_0 \otimes B_j)
+\sum\limits_{j=1}^t\sum\limits_{i=1}^s
\left(  z_{i}w_{j}\right) A_{i}\otimes B_{j}}$};
\draw (-9.7,-2.56) -| (0.61,-4.26);
\draw (0.61,-4.26) -| (-9.7,-2.56);

\draw (-4.5,-4.27) coordinate (14);
\draw (-4.5,-5.75) coordinate (15);
\draw[->,>=stealth] (14) -- node[left] {(ii)} (15);

\draw (-4.5,-6.3) node[blockone] (m) {Simple Matrix Product $\boldsymbol{z_1z_2C = [z_1z_2] \otimes C}$};

\draw (-4.5,-6.85) coordinate (16);
\draw (-4.5,-7.78) coordinate (17);
\draw[->,>=stealth] (16) -- node[left] {(iii.a)} (17);

\draw (-4.5,-8.1) node[blockthree] (z1z2) {Simple Product 
$\boldsymbol{z_1z_2}$};

\draw (6,-3.4) node[blockfour] (sp) {$\boldsymbol{A(z)/A_{22}(z) \otimes B(w)/B_{22}(w)}$};

\draw (6,-4.27) coordinate (14);
\draw (6,-5.68) coordinate (15);
\draw[->,>=stealth] (14) -- node[right] {(iii.b)} (15);

\draw (6,-6.3) node[blockone] (m) {$\boldsymbol{\begin{array}{c}M(z,w)/M_{22}(z,w),\\ \hspace{-0.4em} M(z,w)=P^T[A(z) \otimes B(w)]P\end{array}}$};

\draw (6,-6.88) coordinate (16);
\draw (6,-7.85) coordinate (17);
\draw[->,>=stealth] (16) -- node[right] {(iv)} node[Rk, left, xshift=-3] {$1$} (17);

\draw (6,-8.1) node[blockthree] (z1z2) {Matrix Pencil};
\draw (-0.7,-8.45) coordinate (20);
\draw (0.5,-8.45) coordinate (21);
\draw[->,>=stealth] (20) -- (21);
\draw (-2.485,-8.1) -| (-0.7,-8.45);
\draw (4,-8.1) -| (2.7,-8.45);
\draw (2.7,-8.45) coordinate (22);
\draw (1.6,-8.45) coordinate (23);
\draw[->,>=stealth] (22) -- (23);
\draw (1.05,-8.5) node[R] (R) {R};
\draw (1.05,-7.6) -- (1.05, -2);
\matrix [right] at (-8,-12) {
  \node [label=right:(ii) Linear Combinations ] {}; \\
  \node [label=right:(iii.a) Scalar Multiplication (scalar is a Schur complement)] {}; \\
  \node [label=right:(iii.b) Kronecker Product of Two Schur Complements] {}; \\
  \node [label=right:(iv) Composition of a Schur Complements] {}; \\
};
\draw (-4.6,-1.8) node[tpn] {\large{$1$}};
\draw (5.9,-1.8) node[tpn] {\large{$2$}};
\end{tikzpicture}
\end{figure}
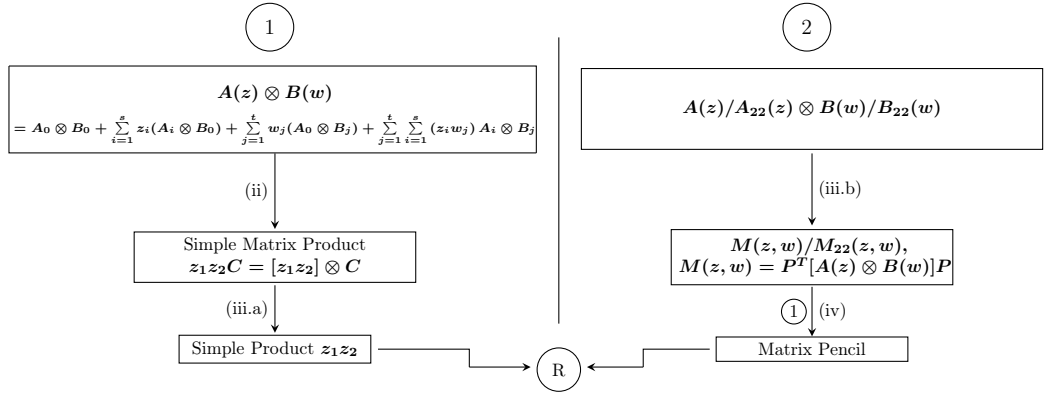

\begin{proof}
The proof of this statement is very similar to the proof of Lemma \ref{LemRealizOfKroneckerProdsPart2}, but is slightly more technical because of Proposition \ref{PropKroneckerProdOfTwoSchurComplements}. By the hypotheses, Proposition \ref{PropKroneckerProdOfTwoSchurComplements}, and the linearity properties of the Kronecker product $\otimes$, it follows that
\begin{align}
M(z,w)/M_{22}(z,w) = A(z)/A_{22}(z)\otimes B(w)/B_{22}(w),
\end{align}
where 
\begin{align}
M(z,w)&=Q^T[B(w)\otimes A(z)]Q=P^T[A(z)\otimes B(w)]P
\end{align}
in which $M=M(z,w)=[M_{ij}(z,w)]_{i,j=1,2}$ is the $2\times 2$ block matrix in terms of $A=A(z)=[A_{ij}(z)]_{i,j=1,2}$ and $B=B(w)=[B_{ij}(w)]_{i,j=1,2}$ in Proposition \ref{PropKroneckerProdOfTwoSchurComplements}, $Q\in\mathbb{C}^{mn\times mn}$ is the constant (independent of $z,w$) permutation matrix given by the formula (\ref{CorLemKroneckerProdSchurComplWithAMatrixQMatrixBlkForm}) satisfying $Q=\overline{Q}, Q^T=Q^{-1}$ (similarly for the permutation matrix $P$ defined in terms of $Q$ in (\ref{RemPMatrixDefInTermsOfQ}) discussed in Remark \ref{RemPMatrix}), and
\begin{equation}
\det M_{22}(z,w)\not \equiv 0.
\end{equation}
The proof of this proposition now follows immediately from this (in a similar manner as the proof of Lemma \ref{LemRealizOfKroneckerProdsPart2}) by Lemma \ref{LemRealizOfKroneckerProdsPart1}, Proposition \ref{PropMatrixMultOfSchurCompl}, and Proposition \ref{PropComposSchurComplem}. This completes the proof.
\end{proof}
\begin{proposition}[Realizability of Monomial $z^{\alpha}$]\label{PropRealizMonomials}
An arbitrary monomial $z^{\alpha}=z_{1}^{\alpha\left(  1\right)  }\cdots
z_{n}^{\alpha\left(  n\right)  }$, where $\alpha=\left(  \alpha\left(
1\right)  ,\ldots,\alpha\left(  n\right)  \right)  \in\left[
\mathbb{N}
\cup\left\{  0\right\}  \right]  ^{n}$, $n\in
\mathbb{N}
$, is Bessmertny\u{\i} realizable by a linear matrix pencil $A\left(  z\right)
=A_{0}+z_{1}A_{1}+\cdots+z_{n}A_{n}$ with matrices
\[
A_{j}=A_{j}^{T}=\overline{A_{j}},\text{ }j=0,\ldots,n\text{.}
\]
\end{proposition}

\begin{proof}
From Lemma \ref{LemRealizProdTwoIndepVar}, the result is true for the product
$w_{1}w_{2}$ of two independent variables $w_1$ and $w_2$. By Proposition \ref{PropRealizOfKroneckerProdOfRealiz}, the result is true
for the product $w_{1}w_{2}\cdots w_{2m-1}w_{2m}$ for any $m\in
\mathbb{N}
$. Hence by taking $m$ large enough, changing variables to the $z_{j}$'s, and
possibly setting some of the $w_{l}$ equal to $1$, it follow that the
monomial $z^{\alpha}=z_{1}^{\alpha\left(  1\right)  }\cdots z_{n}%
^{\alpha\left(  n\right)  }$ has the desired Bessmertny\u{\i} realization. This
completes the proof.
\end{proof}

\subsection{Compositions}\label{SecCompositionOfSchurComplement}

This section is on compositions of Schur complements. For us, our result on compositions  (Proposition \ref{PropComposSchurComplem}) represents a fundamental result in Section \ref{SecSchurComplAlgebraAndOps} on algebra and operations on Schur complements and realizations which allow for producing from basic building blocks more complicated ones.

In order to understand the notion of compositions of Schur complements and the proposition that follows, we introduce first some notation. We also have provided a concrete example below (Example \ref{ExPropComposSchurComplem}) that uses the notation and applies Proposition \ref{PropComposSchurComplem}.

\begin{definition}[The Schur complement function]
Suppose $m,k$ are positive integers such that $m>k.$ Then the Schur complement function with respect to the pair $(m,k)$ is the function $f=f_{m,k}:D_{m,k}\rightarrow\mathbb{C}^{k\times k}$ defined by
\begin{equation*}
f(A)=f_{m,k}(A)=A/A_{22},   
\end{equation*} whose domain $D_{m,k}$ consists of all matrices $A\in\mathbb{C}^{m\times m}$ with a $2\times 2$ block matrix form
\[
A=\begin{bmatrix}
A_{11} & A_{12}\\
A_{21} & A_{22}
\end{bmatrix},
\]
such that $A_{22}\in\mathbb{C}^{(m-k)\times (m-k)}$ is invertible.
\end{definition}
In this subsection we are interested in the composition of Schur complement functions, that is, using the definition above, the composition $h=g\circ f$ of the function $g=f_{k,l}$ with $f=f_{m,k}$. More precisely, let $l$ be any positive integer such that $k>l$, then $g=f_{k,l}:D_{k,l}\rightarrow\mathbb{C}^{l\times l}$ is the function
\begin{equation*}
g(B)=f_{k,l}(B)=B/B_{44},   
\end{equation*}
whose domain $D_{k,l}$ consists of all matrices $B\in\mathbb{C}^{k\times k}$ with a $2\times 2$ block matrix form
\[
B=\begin{bmatrix}
B_{33} & B_{34}\\
B_{43} & B_{44}
\end{bmatrix},
\]
such that $B_{44}\in\mathbb{C}^{(k-l)\times (k-l)}$ is invertible. Therefore, the composition function $h=g\circ f:\{A\in D_{m,k}:f(A)\in D_{k,l}\}\rightarrow\mathbb{C}^{l\times l}$ is defined by
\begin{equation*}
h(A)=g(B),\; B=f(A),  
\end{equation*}
that is,
\begin{equation*}
h(A)=g( f (A) )=(A/A_{22})/(A/A_{22})_{44}. 
\end{equation*}

The main question we address in this subsection is whether or not $h(A)$ is a Schur complement, i.e., for each $A\in D_{m,k}$ with $f(A)\in D_{k,l}$, does there exist a $2\times 2$ block matrix $C=[C_{ij}]_{i,j=1,2}$ with $C_{22}$ invertible such that $h(A)=C/C_{22}$? The next proposition tells us that the answer is yes and gives a formula for this matrix $C$ in terms of $A$.

In order to state the next proposition and give a proof, we need to give some notation first. Begin by partitioning the $k\times k$ matrix $A_{11}=[A_{ij}]_{i,j=3,4}$ conformal to the block structure of the $k\times k$ matrix $A/A_{22}=B=[B_{ij}]_{i,j=3,4}$ so that
\begin{equation}
A_{11}=\begin{bmatrix}
A_{33} & A_{34}\\
A_{43} & A_{44}
\end{bmatrix}.\label{DefCMatrixA11PartitInPropComposSchurComplem}
\end{equation}
This yields a subpartitioning of the matrix $A$ into a $3\times 3$ block matrix as
\begin{align} 
A & =
\left[\begin{array}{c;{2pt/2pt} c c}
A_{11} & A_{12} \\ \hdashline[2pt/2pt]
A_{21} & A_{22}
\end{array}\right]
=
\left[\begin{array}{c c; {2pt/2pt} c}
A_{33} & A_{34} & A_{32} \\
A_{43} & A_{44} & A_{42} \\ \hdashline[2pt/2pt]
A_{23} & A_{24} & A_{22}
\end{array}\right].\label{DefCMatrixAPartitInPropComposSchurComplem}
\end{align}
On the other hand, we can repartition the matrix $A$ in the following $2\times 2$ block partitioned structure $A=C=[C_{ij}]_{i,j=1,2}$: 
\begin{align} 
A = C & =
\left[\begin{array}{c;{2pt/2pt} c c}
C_{11} & C_{12} \\ \hdashline[2pt/2pt]
C_{21} & C_{22}
\end{array}\right]
=
\left[\begin{array}{c; {2pt/2pt} c c}
A_{33} & A_{34} & A_{32} \\\hdashline[2pt/2pt]
A_{43} & A_{44} & A_{42} \\ 
A_{23} & A_{24} & A_{22}
\end{array}\right],\label{DefCMatrixCPartitInPropComposSchurComplem}
\end{align}
where, in particular,
\begin{equation}
C_{22}
=
\begin{bmatrix}
A_{44} & A_{42}\\
A_{24} & A_{22}
\end{bmatrix}.\label{DefCMatrixC22PartitInPropComposSchurComplem}
\end{equation}
Our question is then answered with the following proposition since it tells us that $C_{22}$ is invertible and 
\begin{equation*}
C/C_{22}=h(A)=g( f (A) ).
\end{equation*}

\begin{proposition}[Composition of Schur complements]\label{PropComposSchurComplem}
If $A\in
\mathbb{C}
^{m\times m}$ is a $2\times2$ block matrix
\[
A=
\begin{bmatrix}
A_{11} & A_{12}\\
A_{21} & A_{22}
\end{bmatrix},
\]
such that $A_{22}$ is invertible and suppose that $A/A_{22}\in
\mathbb{C}
^{k\times k}$ is also a $2\times2$ block matrix
\[
A/A_{22}=
\begin{bmatrix}
(A/A_{22})_{33} & (A/A_{22})_{34}\\
(A/A_{22})_{43} & (A/A_{22})_{44}
\end{bmatrix},
\]
such that $(A/A_{22})_{44}\in\mathbb{C}^{l\times l}$ is invertible then  
\begin{equation*}
C/C_{22}=(A/A_{22})/(A/A_{22})_{44},
\end{equation*}
where $A=C=[C_{ij}]_{i,j=1,2}$ is the $2\times 2$ block matrix in (\ref{DefCMatrixCPartitInPropComposSchurComplem}) [defined in terms of the subpartitioning of $A_{11}$ and $A$ in (\ref{DefCMatrixA11PartitInPropComposSchurComplem}) and (\ref{DefCMatrixAPartitInPropComposSchurComplem})] with $C_{22}$ [in (\ref{DefCMatrixC22PartitInPropComposSchurComplem})] invertible and
\begin{align*}
(A/A_{22})_{44}=C_{22}/A_{22}.
\end{align*}
\end{proposition}
\begin{proof}
This proposition is essentially just the well-known Crabtree-Haynsworth quotient formula for Schur complements \cite[Theorem 1.2 (Quotient Formula), p. 25]{05FZ}. First, using the $2\times 2$ block matrix form (\ref{DefCMatrixA11PartitInPropComposSchurComplem}) for $A_{11}$ and the $3\times 3$ block matrix form for $A$ in (\ref{DefCMatrixAPartitInPropComposSchurComplem}) we find that
\begin{align}
A/A_{22} &= A_{11}-A_{12}A_{22}^{-1}A_{21} \nonumber\\
&= \begin{bmatrix}
A_{33} & A_{34}\\
A_{43} & A_{44}
\end{bmatrix}
-
\begin{bmatrix}
A_{32}\\
A_{42}
\end{bmatrix}
A_{22}^{-1}
\begin{bmatrix}
A_{23} & A_{24}
\end{bmatrix}\nonumber
\\
&=
\begin{bmatrix}
A_{33}-A_{32}A_{22}^{-1}A_{23} & A_{34}-A_{32}A_{22}^{-1}A_{24}\\
A_{43}-A_{42}A_{22}^{-1}A_{23} & A_{44}-A_{42}A_{22}^{-1}A_{24}%
\end{bmatrix}.\label{ComposSchurComplemFormulasBlocksOfSchurCompl}
\end{align}
From this, it follows that
\begin{align*}
(A/A_{22})_{44}&=A_{44}-A_{42}A_{22}^{-1}A_{24}=C_{22}/A_{22},
\end{align*}
where $C_{22}$ is the $2\times 2$ block matrix in (\ref{DefCMatrixC22PartitInPropComposSchurComplem}). Now by hypotheses, both $A_{22}$ and $(A/A_{22})_{44}$ are invertible, and we just proved $C_{22}/A_{22}=(A/A_{22})_{44}$ so that by Lemma \ref{LemInvIsASchurCompl} it follows that $C_{22}$ is also invertible and using inverse formula in (\ref{InvBlockFactorizationUsingSchurComplement}) for $C_{22}^{-1}$ we have
\begin{align}
\hspace{-0.25cm} C_{22}^{-1} &= \begin{bmatrix}
A_{44} & A_{42}\\
A_{24} & A_{22}
\end{bmatrix}
^{-1}\nonumber\\
&= \begin{bmatrix}
((A/A_{22})_{44})^{-1} & -((A/A_{22})_{44})^{-1}A_{42}A_{22}^{-1}\\
-A_{22}^{-1}A_{24}((A/A_{22})_{44})^{-1} & A_{22}^{-1}+A_{22}^{-1}A_{24}((A/A_{22})_{44})^{-1}A_{42}A_{22}^{-1}
\end{bmatrix}.\label{InvBlockFormulaForC11Inv}
\end{align}
Thus, by block multiplication, it follows from the formulas  (\ref{ComposSchurComplemFormulasBlocksOfSchurCompl}) and (\ref{InvBlockFormulaForC11Inv}) that
\begin{gather*}
C/C_{22} = A_{33}-
\begin{bmatrix}
A_{34} & A_{32}
\end{bmatrix}
\begin{bmatrix}
A_{44} & A_{42}\\
A_{24} & A_{22}
\end{bmatrix}
^{-1}
\begin{bmatrix}
A_{43}\\
A_{23}
\end{bmatrix}\\
\hspace{-3em}=A_{33}-
\begin{bmatrix}
A_{34} & A_{32}
\end{bmatrix}
\begin{bmatrix}
((A/A_{22})_{44})^{-1} & -((A/A_{22})_{44})^{-1}A_{42}A_{22}^{-1}\\
-A_{22}^{-1}A_{24}((A/A_{22})_{44})^{-1} & A_{22}^{-1}+A_{22}^{-1}A_{24}((A/A_{22})_{44})^{-1}A_{42}A_{22}^{-1}
\end{bmatrix}
\begin{bmatrix}
A_{43}\\
A_{23}
\end{bmatrix}\\
=A_{33}-[A_{34}((A/A_{22})_{44})^{-1}-A_{32}A_{22}^{-1}A_{24}((A/A_{22})_{44})^{-1}]A_{43}\\
-[-A_{34}((A/A_{22})_{44})^{-1}A_{42}A_{22}^{-1}+A_{32}(A_{22}^{-1}+A_{22}^{-1}A_{24}((A/A_{22})_{44})^{-1}A_{42}A_{22}^{-1})]A_{23}
\\
=A_{33}-A_{34}((A/A_{22})_{44})^{-1}A_{43}+A_{32}A_{22}^{-1}A_{24}((A/A_{22})_{44})^{-1}A_{43}\\
+A_{34}((A/A_{22})_{44})^{-1}A_{42}A_{22}^{-1}A_{23}-A_{32}A_{22}^{-1}A_{23}-A_{32}A_{22}^{-1}A_{24}((A/A_{22})_{44})^{-1}A_{42}A_{22}^{-1}A_{23}
\\
=A_{33}-A_{32}A_{22}^{-1}A_{23}-(A_{34}-A_{32}A_{22}^{-1}A_{24})((A/A_{22})_{44})^{-1}A_{43}\\
+(A_{34}-A_{32}A_{22}^{-1}A_{24})((A/A_{22})_{44})^{-1}A_{42}A_{22}^{-1}A_{23}
\\
=(A_{33}-A_{32}A_{22}^{-1}A_{23})-(A_{34}-A_{32}A_{22}^{-1}A_{24})((A/A_{22})_{44})^{-1}(A_{43}-A_{42}A_{22}^{-1}A_{23})
\\
=(A/A_{22})_{33}-(A/A_{22})_{34}(A/A_{22})_{44}^{-1}(A/A_{22})_{43}\\
=(A/A_{22})/(A/A_{22})_{44}.
\end{gather*}
This completes the proof.
\end{proof}

\begin{example}\label{ExPropComposSchurComplem}
We will now work out a concrete example to to demonstrate the notation and Proposition \ref{PropComposSchurComplem} above. Consider the following $2\times 2$ block matrix $A=[A_{ij}]_{i,j=1,2}$ and its Schur complement $A/A_{22}$,
$$A =
\left[\begin{array}{c;{2pt/2pt} c c}
A_{11} & A_{12} \\ \hdashline[2pt/2pt]
A_{21} & A_{22}
\end{array}\right] =
\left[\begin{array}{c c;{2pt/2pt} c c}
4 & 3 & 1 & 1 \\
4 & 2 & 2 & 1\\  \hdashline[2pt/2pt]
1 & 1 & 1 & 1\\  
2 & 1 & 2 & 1\\
\end{array}\right],\;\;A/A_{22}=\begin{bmatrix}
3 & 2 \\
2 & 1
\end{bmatrix}.$$
Suppose now we interested in taking the Schur complement of $A / A_{22}$ with respect to its lower left corner block (i.e., the invertible matrix $[1]\in \mathbb{C}^{1\times 1}$). Then we block partition $A / A_{22}$ into the $2\times 2$ block matrix $A / A_{22}=[(A / A_{22})_{ij}]_{i,j=3,4}$ and compute the desired Schur complement $(A / A_{22})/(A / A_{22})_{44}$ as
\begin{gather*}
A / A_{22}
= \left[\begin{array}{c;{2pt/2pt}c}
(A / A_{22})_{33} & (A / A_{22})_{34} \\ \hdashline[2pt/2pt]
(A / A_{22})_{43} & (A / A_{22})_{44}
\end{array}\right] = \left[\begin{array}{c;{2pt/2pt}c}
3 & 2 \\ \hdashline[2pt/2pt]
2 & 1
\end{array}\right],\\
(A / A_{22})/(A / A_{22})_{44}=[3]-[2][1]^{-1}[2]=[-1].
\end{gather*}
According to Proposition \ref{PropComposSchurComplem},
\begin{align*}
C/C_{22}=(A / A_{22})/(A / A_{22})_{44},
\end{align*}
where $A=C=[C_{ij}]_{i,j=1,2}$ is the $2\times 2$ block matrix in (\ref{DefCMatrixCPartitInPropComposSchurComplem}) [defined in terms of the subpartitioning of $A_{11}$ and $A$ in (\ref{DefCMatrixA11PartitInPropComposSchurComplem}) and (\ref{DefCMatrixAPartitInPropComposSchurComplem})] with $C_{22}$ [in (\ref{DefCMatrixC22PartitInPropComposSchurComplem})] invertible and $(A/A_{22})_{44}=C_{22}/A_{22}.$
Lets work this all out explicitly now for this example. We begin by partitioning $A_{11}=[A_{ij}]_{i,j=3,4}$ conformal to the block structure of $A / A_{22}=[(A / A_{22})_{ij}]_{i,j=3,4}$ so that
\begin{align*}
A_{11}=\left[\begin{array}{c;{2pt/2pt}c}
A_{33} & A_{34} \\ \hdashline[2pt/2pt]
A_{43} & A_{44}
\end{array}\right]=\left[\begin{array}{c;{2pt/2pt}c}
4 & 3 \\ \hdashline[2pt/2pt]
4 & 2
\end{array}\right]
\end{align*}
This yields a subpartitioning of the matrix $A$ into the $3\times 3$ block matrix as
\begin{align*}
A=\left[\begin{array}{c;{2pt/2pt} c c}
A_{11} & A_{12} \\ \hdashline[2pt/2pt]
A_{21} & A_{22}
\end{array}\right]=\left[\begin{array}{c c; {2pt/2pt} c}
A_{33} & A_{34} & A_{32} \\
A_{43} & A_{44} & A_{42} \\ \hdashline[2pt/2pt]
A_{23} & A_{24} & A_{22}
\end{array}\right]
=
\left[\begin{array}{c c; {2pt/2pt} c c}
4 & 3 & 1 & 1 \\
4 & 2 & 2 & 1 \\ \hdashline[2pt/2pt]
1 & 1 & 1 & 1\\
2 & 1 & 2 & 1
\end{array}\right]
\end{align*}
and from this we repartition $A$ to get the $2\times 2$ block matrix
\begin{align*} 
A = C  =
\left[\begin{array}{c;{2pt/2pt} c c}
C_{11} & C_{12} \\ \hdashline[2pt/2pt]
C_{21} & C_{22}
\end{array}\right]
=
\left[\begin{array}{c; {2pt/2pt} c c}
A_{33} & A_{34} & A_{32} \\\hdashline[2pt/2pt]
A_{43} & A_{44} & A_{42} \\ 
A_{23} & A_{24} & A_{22}
\end{array}\right]=\left[\begin{array}{c;{2pt/2pt} c c c}
4 & 3 & 1 & 1\\ \hdashline[2pt/2pt]
4 & 2 & 2 & 1\\ 
1 & 1 & 1 & 1\\
2 & 1 & 2 & 1\\
\end{array}\right].
\end{align*}
We now verify that
\begin{align*}
C / C_{22}
&= [4] - \begin{bmatrix}
3 & 1 & 1
\end{bmatrix} \begin{bmatrix}
2 & 2 & 1 \\
1 & 1 & 1 \\
1 & 2 & 1
\end{bmatrix}^{-1} \begin{bmatrix}
4\\
1\\
2
\end{bmatrix}
= [4] - [5]\\
&= [-1]=(A / A_{22})/(A / A_{22})_{44},\\
C_{22}/A_{22}&=\left.\left[\begin{array}{c; {2pt/2pt} c c}
2 & 2 & 1 \\\hdashline[2pt/2pt]
1 & 1 & 1 \\ 
1 & 2 &1
\end{array}\right]
\right /
\begin{bmatrix}
1 & 1\\
2 & 1
\end{bmatrix}
=
[2]-\begin{bmatrix}
2 & 1
\end{bmatrix}
\begin{bmatrix}
1 & 1\\
2 & 1
\end{bmatrix}^{-1}
\begin{bmatrix}
1 \\
1 
\end{bmatrix}
=[2]-[1]
\\
&=[1]=(A/A_{22})_{44}.
\end{align*}
\end{example}

\subsection{Transforms (leading to alternative realizations)}\label{SecTransformsForAlternativeRealizations}

In this subsection we discuss additional transformations of matrices associated with the Schur complement and how they can be used to give alternative realization theorems for rational matrix functions, i.e., instead of the Bessmernty\u{\i} Realization Theorem (Theorem \ref{ThmBessmRealiz}) or in conjunction with it. Our main focus will be on the \textit{principal pivot transform} (PPT), see Definition \ref{DefPPTPrincipalPivotTransform} below, which can be considered as a matrix partial inverse. We begin by introducing the reader to PPT in the context of network synthesis problems. After this we will give our results.

The PPT was introduced by A. W. Tucker in \cite{60AT} (see also \cite{00MT}) as ``an attempt to study for a general field, rather than an ordered field, the linear algebraic structure underlying the `simplex method' of G. B. Dantzig, so remarkably effective in Linear Programming." 

Later, R. J. Duffin, D. Hazony, and N. Morrison in \cite{65DHM, 66DHM} studied the PPT for the purposes of solving certain network synthesis problems, although in the latter its called the \textit{gyration} and they denote it by $\Gamma$. More generally, if you compare our Definition \ref{DefPPTPrincipalPivotTransform} of $\operatorname{ppt}_1(A)$ below to the definition of the $r$-fold gyration $\Gamma_{1,\ldots, r}(A)$ in \cite[Sec. 3.2, pp. 54-55, especially (11)]{65DHM} of a $2\times 2$ block matrix $A=[A_{ij}]_{i,j=1,2}\in \mathbb{C}^{n\times n}$ with $A_{11}\in \mathbb{C}^{r\times r}$ invertible, you will see that $\operatorname{ppt}_1(A)=\Gamma_{1,\ldots, r}(A)$. 
As quoted in \cite[Sec. 1.1, p. 1]{66DHM}, if $\Gamma(A)=B$ for two matrices $A$ and $B$ that ``This relationship is sufficient to make the matrices $A$ and $B$ combinatorially equivalent,`` a term they say was coined by A. W. Tucker in \cite{60AT} and ``The impedance, admittance, chain, and hybrid matrices of network theory are all combinatorially equivalent. The work of A. W. Tucker emerged from the linear programming field and is applied here to network theory." They elaborate on this in \cite[Sec. 1.3, p. 394]{66DHM} by saying, ``We wish to show in this paper that combinatorial equivalence has application in the entirely different field of network synthesis. It is worth noting that ideas similar to combinatorial equivalence have been applied to network algebra problems by Bott and Duffin..." and they cite \cite{53BD, 59RD} (see also \cite{78DM}). They further elaborate on their synthesis procedure in \cite[Sec. 2.3, p. 402]{66DHM} saying, ``In what follows, we shall use the $\Gamma$ operator to give a new extension of the Brune synthesis to $n$-port..." and that R. J. Duffin in \cite{55RD} had showed how such a network Brune-type synthesis could be viewed as a purely algebraic process. The key point though that \cite{66DHM} makes is that ``Our extension of the Brune method differs from the above in that it is not necessary at any state to invert a matrix."  In these regard, it is not surprising that R. J. Duffin and his colleagues would be interest in such network synthesis problems that don't require certain algebraic operations given his famous result with R. Bott in \cite{49BD} on synthesis of network impedance functions of one-variable without the use of transformers (for R. Bott's perspective on this, see his interview \cite{01AJ}).

The PPT and variations of its form also appear in other contexts such as in the work of M. G. Krein and I. E. Ovcharenko \cite{94KO} on inverse problems for canonical differential equations or in the study of the analytic properties of the Dirichlet-to-Neumann (DtN) map in electromagnetism \cite{16CWM} for layered media, for instance.

The above serves to give perspective and motivate our consideration of the PPT below in connection to the realization problem of this paper using the Schur complement and the relationship with the PPT.

\subsubsection{Principal pivot transform}\label{SecPrincipalPivotTransform}
There are two forms of the ppt of a $2\times 2$ block matrix $A=[A_{ij}]_{i,j=1,2}\in\mathbb{C}^{m\times m}$ that we will discuss, denoted by $\operatorname{ppt}_1$ and $\operatorname{ppt}_2$, which can be written in terms of the two Schur complements either $A/A_{11}$ or $A/A_{22}$ of $A$ with respect to the $(1,1)$-block $A_{11}$ (if $A_{11}$ is invertible) or the $(2,2)$-block $A_{22}$ (if $A_{22}$ is invertible), where by definition
\begin{align}
A /A_{11} &= A_{22} - A_{21} A_{11}^{-1} A_{12},\\
A/A_{22} &= A_{11}-A_{12}A_{22}^{-1}A_{21}.
\end{align}

In fact, the results of this paper could have been posed in terms of the first Schur complement $A/A_{11}$ instead of second one $A/A_{22}$. The relationship between these two versions of the Schur complement is described in the next lemma. This lemma, together with Proposition \ref{PropMatrixMultOfSchurCompl}, gives a means to easily transform our results, which are stated in terms of second form of the Schur complements, into similar statements in terms of the first form (or vice versa). As an example of this, compare Proposition \ref{PropPPTIsASchurComplement} to its corollary (Corollary \ref{CorOtherPPTAsASchurCompl}) by considering the proof of the latter.


\begin{lemma}[Relationship between the two Schur complements]\label{LemRelBetwTheTwoSchuCompls}
If $A\in
\mathbb{C}
^{m\times m}$ is a $2\times2$ block matrix
\[
A=
\begin{bmatrix}
A_{11} & A_{12}\\
A_{21} & A_{22}
\end{bmatrix},
\]
such that $A_{11}\in\mathbb{C}^{k\times k}$ is invertible then
\begin{equation}
B/B_{22} = A/A_{11},\label{SchuComp1InTermsOf2InLemRelBetwTheTwoSchuCompls}
\end{equation}
where $B=[B_{ij}]_{i,j=1,2}\in
\mathbb{C}
^{m\times m}$ is the $2\times 2$ block matrix
\begin{equation}
B=\begin{bmatrix}
B_{11} & B_{12}\\
B_{21} & B_{22}
\end{bmatrix}
=
\begin{bmatrix}
A_{22} & A_{21}\\
A_{12} & A_{11}
\end{bmatrix}
=
U^TAU,\label{BMatrixInLemRelBetwTheTwoSchuCompls}
\end{equation}
such that $B_{22}=A_{11}$ is invertible and $U\in
\mathbb{C}
^{m\times m}$ is the $2\times 2$ block matrix
\begin{equation}
U=
\begin{bmatrix}
0 & I_k\\
I_{m-k} & 0
\end{bmatrix}.\label{UMatrixInLemRelBetwTheTwoSchuCompls}
\end{equation}
Moreover, if $A$ is real, symmetric, Hermitian, or real and symmetric then the matrix $B$ is real, symmetric, Hermitian, or real and symmetric, respectively.
\end{lemma}
\begin{proof}
The proof is obvious from the definitions of $A/A_{11}$, $B$, and $U$.
\end{proof}

\begin{definition}\label{DefPPTPrincipalPivotTransform}
The principal pivot transform (PPT) of a matrix $A\in\mathbb{C}^{m\times m}$ in $2\times 2$ block matrix form
\begin{equation}
A=\begin{bmatrix}
A_{11} & A_{12}\\
A_{21} & A_{22}
\end{bmatrix},
\end{equation}
with respect to an invertible $A_{22}$ is defined to be the matrix $\operatorname{ppt}_2(A)\in\mathbb{C}^{m\times m}$ with the $2\times 2$ block  matrix form 
\begin{equation}
\operatorname{ppt}_2(A) =
\begin{bmatrix}
A/A_{22} & A_{12}A_{22}^{-1} \\ 
-A_{22}^{-1}A_{21} & A_{22}^{-1}
\end{bmatrix}.
\end{equation}
Similarly, the PPT of $A$ with respect to an invertible $A_{11}$ is defined to be the matrix $\operatorname{ppt}_1(A)\in\mathbb{C}^{m\times m}$ with the $2\times 2$ block matrix form 
\begin{equation}
\operatorname{ppt}_1(A) =
\begin{bmatrix}
A_{11}^{-1} & -A_{11}^{-1}A_{12} \\ 
A_{21}A_{11}^{-1} & A/A_{11}
\end{bmatrix}.
\end{equation}
\end{definition}

The relationship between these two versions of the PPT is described in the next lemma.
\begin{lemma}[Relationship between the two PPTs]\label{LemRelBetwTheTwoPPTs}
If $A\in\mathbb{C}^{m\times m}$ is a $2\times 2$ block matrix
\begin{equation}
A=\begin{bmatrix}
A_{11} & A_{12}\\
A_{21} & A_{22}
\end{bmatrix},
\end{equation}
such that $A_{11}\in\mathbb{C}^{k\times k}$ is invertible then
\begin{equation}
\operatorname{ppt}_1(A)=U\operatorname{ppt}_2(B)U^T,
\end{equation}
where $B,U\in\mathbb{C}^{m\times m}$ are the $2\times 2$ block matrices given in terms of $A$ by (\ref{BMatrixInLemRelBetwTheTwoSchuCompls}) and (\ref{UMatrixInLemRelBetwTheTwoSchuCompls}), respectively. 
\end{lemma}
\begin{proof}
By the definitions of the two PPTs (i.e., $\operatorname{ppt}_1$ and $\operatorname{ppt}_2$), the definitions of the $2\times 2$ block matrices $B,U$ in (\ref{BMatrixInLemRelBetwTheTwoSchuCompls}) and (\ref{UMatrixInLemRelBetwTheTwoSchuCompls}), respectively, in terms of $A$ and the formula (\ref{SchuComp1InTermsOf2InLemRelBetwTheTwoSchuCompls}) from Lemma \ref{LemRelBetwTheTwoSchuCompls} we have
\begin{align*}
    \operatorname{ppt}_1(A)&=\begin{bmatrix}
A_{11}^{-1} & -A_{11}^{-1}A_{12} \\ 
A_{21}A_{11}^{-1} & A/A_{11}
\end{bmatrix}\\&= \begin{bmatrix}
0 & I_k\\
I_{m-k} & 0
\end{bmatrix}\begin{bmatrix}
A/A_{11} & A_{21}A_{11}^{-1}\\
-A_{11}^{-1}A_{12} & A_{11}^{-1}
\end{bmatrix}\begin{bmatrix}
0 & I_{m-k}\\
I_k & 0
\end{bmatrix}\\&=\begin{bmatrix}
0 & I_k\\
I_{m-k} & 0
\end{bmatrix}\begin{bmatrix}
B/B_{22} & B_{12}B_{22}^{-1}\\
-B_{22}^{-1}B_{21} & B_{22}^{-1}
\end{bmatrix}\begin{bmatrix}
0 & I_{m-k}\\
I_k & 0
\end{bmatrix}\\&=U\operatorname{ppt}_2(B)U^T,
\end{align*}
which proves the lemma.
\end{proof}

\begin{remark}
There are two important remarks that need to be made regarding other realizations that are possible instead of the Bessmertny\u{\i} Realization Theorem (i.e., Theorem \ref{ThmBessmRealiz}).
\begin{itemize}
    \item[i)] From the definition of the PPT above, we have the simple relationship between the PPT and the Schur complement: If $A\in\mathbb{C}^{m\times m}$ is a $2\times 2$ block matrix
\begin{equation}
A=\begin{bmatrix}
A_{11} & A_{12}\\
A_{21} & A_{22}
\end{bmatrix},
\end{equation}
such that $A_{11}\in \mathbb{C}^{k\times k}$ is invertible and $A_{22}\in\mathbb{C}^{p\times p}$ then
\begin{align}
A/A_{22}=V^T\operatorname{ppt}_2(A)V,
\end{align}
where $V\in \mathbb{C}^{(k+p)\times k}$ is the $2\times 1$ block matrix
\begin{align}
V=\begin{bmatrix}
I_k \\ 
0
\end{bmatrix}.
\end{align}

    \item[ii)]It follows from this and the next proposition that one could have instead stated a realization theorem for rational matrix functions $f(z)$ similar to the Bessmertny\u{\i} Realization Theorem (i.e., Theorem \ref{ThmBessmRealiz}), but in terms of the principal pivot transform $\operatorname{ppt}_2(\cdot)$ instead of Schur complement $(\cdot)/(\cdot)_{22}$. And then using Lemma \ref{LemRelBetwTheTwoPPTs} or Corollary \ref{CorOtherPPTAsASchurCompl}, this could be done in terms of the other principal pivot transform $\operatorname{ppt}_1(\cdot)$ instead.
\end{itemize}
\end{remark}

\begin{proposition}[Principal pivot transform as a Schur complement]\label{PropPPTIsASchurComplement}
If $A\in\mathbb{C}^{m\times m}$ is a $2\times 2$ block matrix
\begin{equation}
A=\begin{bmatrix}
A_{11} & A_{12}\\
A_{21} & A_{22}
\end{bmatrix}
\end{equation}
such that $A_{22}\in\mathbb{C}^{p\times p}$ is invertible and $A_{11}\in \mathbb{C}^{k\times k}$ then
\begin{equation}
C/C_{22} = \operatorname{ppt}_2(A),
\end{equation}
where $C\in \mathbb{C}^{(k+p+p)\times (k+p+p)}$ is the $3\times 3$ block matrix with the following block partititioned structure $C=[C_{ij}]_{i,j=1,2}$:
\begin{align} 
C & =
\left[\begin{array}{c;{2pt/2pt} c c}
C_{11} & C_{12} \\ \hdashline[2pt/2pt]
C_{21} & C_{22}
\end{array}\right]
=
\left[\begin{array}{c c; {2pt/2pt} c}
A_{11} & 0 & A_{12} \\
0 & 0_p & I_p \\ \hdashline[2pt/2pt]
A_{21} & -I_p & A_{22}
\end{array}\right]
\end{align}
with $C_{22}=A_{22}$ invertible. Furthermore,
\begin{gather}
(JC)/(JC)_{22}=J_{11}\operatorname{ppt}_2(A)=\begin{bmatrix}
A/A_{22} & A_{12}A_{22}^{-1}\\
A_{22}^{-1}A_{21} & -A_{22}^{-1}
\end{bmatrix},\label{PropPPTIsASchurComplementJ11PPT2AMatrixFormula}\\
JC=\left[\begin{array}{c;{2pt/2pt} c c}
(JC)_{11} & (JC)_{12} \\ \hdashline[2pt/2pt]
(JC)_{21} & (JC)_{22}
\end{array}\right]=\left[\begin{array}{c;{2pt/2pt} c c}
J_{11}C_{11} & J_{11}C_{12} \\ \hdashline[2pt/2pt]
C_{21} & C_{22}
\end{array}\right]
=
\left[\begin{array}{c c; {2pt/2pt} c}
A_{11} & 0 & A_{12} \\
0 & 0_p & -I_p \\ \hdashline[2pt/2pt]
A_{21} & -I_p & A_{22}
\end{array}\right],\label{PropPPTIsASchurComplementJCMatrixFormula}
\end{gather}
where $J_{11}\in\mathbb{C}^{m\times m}$ and $J\in \mathbb{C}^{(k+p+p)\times (k+p+p)}$ are the $2\times 2$ block matrices
\begin{align}
J_{11}=\begin{bmatrix}
I_k & 0 \\
0 & -I_p
\end{bmatrix},\;\;
J=\begin{bmatrix}
J_{11} & 0\\
0 & I_p
\end{bmatrix}.
\end{align}
Moreover, if $A$ is real, symmetric, Hermitian, or real and symmetric then $J_{11}\operatorname{ppt}_2(A)$ and $JC$ are both real, symmetric, Hermitian, or real and symmetric, respectively.
\end{proposition}
\begin{proof}
The proof is straightforward via block multiplication. First,
\begin{align*}\operatorname{ppt}_2(A)&=\begin{bmatrix}
A/A_{22} & A_{12}A_{22}^{-1} \\ 
-A_{22}^{-1}A_{21} & A_{22}^{-1}
\end{bmatrix}\\
&=
\begin{bmatrix}
A_{11}-A_{12}A_{22}^{-1}A_{21} & A_{12}A_{22}^{-1}\\
-A_{22}^{-1}A_{21}  & A_{22}^{-1}
\end{bmatrix}
\\
&=
\begin{bmatrix}
A_{11} & 0\\
0 & 0_p
\end{bmatrix}
-
\begin{bmatrix}
A_{12}A_{22}^{-1}A_{21} & -A_{12}A_{22}^{-1}\\
A_{22}^{-1}A_{21} & -A_{22}^{-1}
\end{bmatrix}\\
&=
\begin{bmatrix}
A_{11} & 0\\
0 & 0_p
\end{bmatrix}
-
\begin{bmatrix}
A_{12}A_{22}^{-1}\\
A_{22}^{-1}
\end{bmatrix}
\begin{bmatrix}
A_{21} & -I_p
\end{bmatrix}\\
&= 
\begin{bmatrix}
A_{11} & 0\\
0 & 0_p
\end{bmatrix}
-
\begin{bmatrix}
A_{12}\\
I_p
\end{bmatrix}
A_{22}^{-1}
\begin{bmatrix}
A_{21} & -I_p
\end{bmatrix}\\&= C_{11}-C_{12}C_{22}^{-1}C_{21}\\&=C/C_{22}
\end{align*}
and
\begin{align*}
J_{11}(C/C_{22})=J_{11}\operatorname{ppt}_2(A)=\begin{bmatrix}
A/A_{22} & A_{12}A_{22}^{-1} \\ 
A_{22}^{-1}A_{21} & -A_{22}^{-1}
\end{bmatrix}.
\end{align*}
By Proposition \ref{PropMatrixMultOfSchurCompl}, it follows that
\begin{align*}
M/M_{22}=J_{11}(C/C_{22}),
\end{align*}
where
\begin{align*} 
M =\left[\begin{array}{c;{2pt/2pt} c c}
M_{11} & M_{12} \\ \hdashline[2pt/2pt]
M_{21} & M_{22}
\end{array}\right]
=
\left[\begin{array}{c;{2pt/2pt} c c}
J_{11}C_{11} & J_{11}C_{12} \\ \hdashline[2pt/2pt]
C_{21} & C_{22}
\end{array}\right].
\end{align*}
Finally, by block multiplication, we verify that
\begin{align*} 
\left[\begin{array}{c;{2pt/2pt} c c}
J_{11}C_{11} & J_{11}C_{12} \\ \hdashline[2pt/2pt]
C_{21} & C_{22}
\end{array}\right]
=
\left[\begin{array}{c c; {2pt/2pt} c}
A_{11} & 0 & A_{12} \\
0 & 0_p & -I_p \\ \hdashline[2pt/2pt]
A_{21} & -I_p & A_{22}
\end{array}\right]
=
JC.
\end{align*}
The remaining part of the proof follows immediately now formulas (\ref{PropPPTIsASchurComplementJ11PPT2AMatrixFormula}) and (\ref{PropPPTIsASchurComplementJCMatrixFormula}). This completes the proof.
\end{proof}

\begin{corollary}[The other PPT as a Schur complement]\label{CorOtherPPTAsASchurCompl}
If $A\in\mathbb{C}^{m\times m}$ is a $2\times 2$ block matrix
\begin{equation}
A=\begin{bmatrix}
A_{11} & A_{12}\\
A_{21} & A_{22}
\end{bmatrix},
\end{equation}
such that $A_{11}\in\mathbb{C}^{k\times k}$ is invertible and $A_{22}\in \mathbb{C}^{p\times p}$ then
\begin{equation}
D/D_{22} = \operatorname{ppt}_1(A),
\end{equation}
where $D\in \mathbb{C}^{(k+p+k)\times (k+p+k)}$ is the $3\times 3$ block matrix with the following block partititioned structure $D=[D_{ij}]_{i,j=1,2}$:
\begin{align} 
D & =
\left[\begin{array}{c;{2pt/2pt} c c}
D_{11} & D_{12} \\ \hdashline[2pt/2pt]
D_{21} & D_{22}
\end{array}\right]
=
\left[\begin{array}{c c; {2pt/2pt} c}
0_k & 0 & I_k \\
0 & A_{22} & A_{21} \\ \hdashline[2pt/2pt]
-I_k & A_{12} & A_{11}
\end{array}\right]
\end{align}
with $D_{22}=A_{11}$ invertible. Furthermore,
\begin{gather}
(KD)/(KD)_{22}=K_{11}\operatorname{ppt}_1(A)=\begin{bmatrix}
-A_{11}^{-1} & A_{11}^{-1}A_{12}\\
A_{21}A_{11}^{-1} & A/A_{11}
\end{bmatrix},\\
KD=\left[\begin{array}{c;{2pt/2pt} c c}
(KD)_{11} & (KD)_{12} \\ \hdashline[2pt/2pt]
(KD)_{21} & (KD)_{22}
\end{array}\right]=\left[\begin{array}{c;{2pt/2pt} c c}
K_{11}D_{11} & K_{11}D_{12} \\ \hdashline[2pt/2pt]
D_{21} & D_{22}
\end{array}\right]
=
\left[\begin{array}{c c; {2pt/2pt} c}
0_k & 0 & -I_k \\
0 & A_{22} & A_{21} \\ \hdashline[2pt/2pt]
-I_k & A_{12} & A_{11}
\end{array}\right],
\end{gather}
where $K_{11}\in\mathbb{C}^{m\times m}$ and $K\in \mathbb{C}^{(k+p+k)\times (k+p+k)}$ are the $2\times 2$ block matrices
\begin{align}
K_{11}=\begin{bmatrix}
-I_k & 0 \\
0 & I_p
\end{bmatrix},\;\;
K=\begin{bmatrix}
K_{11} & 0\\
0 & I_k
\end{bmatrix}.
\end{align}
Moreover, if $A$ is real, symmetric, Hermitian, or real and symmetric then $K_{11}\operatorname{ppt}_1(A)$ and $KD$ are both real, symmetric, Hermitian, or real and symmetric, respectively.
\end{corollary}
\begin{proof}
Although the proof of this corollary could be proved directly by verifying via block matrix methods the statements, our goal here is to give a proof based on the discussion in the introduction of Section \ref{SecPrincipalPivotTransform}, namely, to prove the corollary using Proposition \ref{PropMatrixMultOfSchurCompl}, Lemma \ref{LemRelBetwTheTwoPPTs}, and Proposition \ref{PropPPTIsASchurComplement}. First, by Lemma \ref{LemRelBetwTheTwoPPTs} we have
\begin{align*}
B=U^TAU,\;\;\operatorname{ppt}_1(A) &= U\operatorname{ppt}_2(B)U^T,
\end{align*}
where $B,U\in\mathbb{C}^{m\times m}$ have the $2\times 2$ block matrix forms in (\ref{BMatrixInLemRelBetwTheTwoSchuCompls}) and (\ref{UMatrixInLemRelBetwTheTwoSchuCompls}), respectively. Next, by Proposition \ref{PropPPTIsASchurComplement} we know that
\begin{align*}
C/C_{22} = \operatorname{ppt}_2(B),
\end{align*}
where, by the definition of $B$ and the hypotheses that $A_{11}\in\mathbb{C}^{k\times k}$ and $A_{22}\in \mathbb{C}^{p\times p}$, the matrix $C\in \mathbb{C}^{(p+k+k)\times (p+k+k)}$ is the $3\times 3$ block matrix with the following block partititioned structure $C=[C_{ij}]_{i,j=1,2}$:
\begin{align*} 
C & =
\left[\begin{array}{c;{2pt/2pt} c c}
C_{11} & C_{12} \\ \hdashline[2pt/2pt]
C_{21} & C_{22}
\end{array}\right]
=
\left[\begin{array}{c c; {2pt/2pt} c}
B_{11} & 0 & B_{12} \\
0 & 0_k & I_k \\ \hdashline[2pt/2pt]
B_{21} & -I_k & B_{22}
\end{array}\right]
=
\left[\begin{array}{c c; {2pt/2pt} c}
A_{22} & 0 & A_{21} \\
0 & 0_k & I_k \\ \hdashline[2pt/2pt]
A_{12} & -I_k & A_{11}
\end{array}\right]
\end{align*}
with $C_{22}=B_{22}=A_{11}$ invertible. Thus, it follows by this and Proposition \ref{PropMatrixMultOfSchurCompl} that
\begin{align*}
D/D_{22} = UC/C_{22}U^T,
\end{align*}
where $D\in \mathbb{C}^{(k+p+k)\times (k+p+k)}$ (with $m=k+p$) is the $2\times 2$ block matrix
\begin{align*}
D&=\begin{bmatrix}
D_{11} & D_{12}\\
D_{21} & D_{22}
\end{bmatrix}
=
\left[\begin{array}{c;{2pt/2pt} c c}
UC_{11}U^T & UC_{12} \\ \hdashline[2pt/2pt]
C_{21}U^T & C_{22}
\end{array}\right]
=
\left[\begin{array}{c c; {2pt/2pt} c}
0_k & 0 & I_k \\
0 & A_{22} & A_{21} \\ \hdashline[2pt/2pt]
-I_k & A_{12} & A_{11}
\end{array}\right]\\
&=
\left[\begin{array}{c;{2pt/2pt} c c}
U & 0 \\ \hdashline[2pt/2pt]
0 & I_k
\end{array}\right]
\left[\begin{array}{c;{2pt/2pt} c c}
C_{11} & C_{12} \\ \hdashline[2pt/2pt]
C_{21} & C_{22}
\end{array}\right]
\left[\begin{array}{c;{2pt/2pt} c c}
U & 0 \\ \hdashline[2pt/2pt]
0 & I_k
\end{array}\right]^T.
\end{align*}
Furthermore, by Proposition \ref{PropPPTIsASchurComplement} and the relation of $B$ to $A$, we have
\begin{gather*}
(JC)/(JC)_{22}=J_{11}\operatorname{ppt}_2(B)=\begin{bmatrix}
B/B_{22} & B_{12}B_{22}^{-1}\\
B_{22}^{-1}B_{21} & -B_{22}^{-1}
\end{bmatrix}\\
=\begin{bmatrix}
A/A_{11} & A_{21}A_{11}^{-1}\\
A_{11}^{-1}A_{12} & -A_{11}^{-1}
\end{bmatrix},\\
JC=\left[\begin{array}{c;{2pt/2pt} c c}
(JC)_{11} & (JC)_{12} \\ \hdashline[2pt/2pt]
(JC)_{21} & (JC)_{22}
\end{array}\right]=\left[\begin{array}{c;{2pt/2pt} c c}
J_{11}C_{11} & J_{11}C_{12} \\ \hdashline[2pt/2pt]
C_{21} & C_{22}
\end{array}\right]
\\=
\left[\begin{array}{c c; {2pt/2pt} c}
B_{11} & 0 & B_{12} \\
0 & 0_k & -I_k \\ \hdashline[2pt/2pt]
B_{21} & -I_k & B_{22}
\end{array}\right]
=
\left[\begin{array}{c c; {2pt/2pt} c}
A_{22} & 0 & A_{21} \\
0 & 0_k & -I_k \\ \hdashline[2pt/2pt]
A_{12} & -I_k & A_{11}
\end{array}\right],
\end{gather*}
where $J_{11}\in\mathbb{C}^{m\times m}$ and $J\in \mathbb{C}^{(p+k+k)\times (p+k+k)}$ are the $2\times 2$ block matrices
\begin{align*}
J_{11}=\begin{bmatrix}
I_p & 0 \\
0 & -I_k
\end{bmatrix},\;\;
J=\begin{bmatrix}
J_{11} & 0\\
0 & I_k
\end{bmatrix}.
\end{align*}
Now, using block multiplication, it follows that
\begin{gather*}
UJ_{11}U^T=K_{11},\;\;
\left[\begin{array}{c;{2pt/2pt} c c}
U & 0 \\ \hdashline[2pt/2pt]
0 & I_k
\end{array}\right]
J
\left[\begin{array}{c;{2pt/2pt} c c}
U & 0 \\ \hdashline[2pt/2pt]
0 & I_k
\end{array}\right]^T=\begin{bmatrix}
K_{11} & 0\\
0 & I_k
\end{bmatrix}=K,
\end{gather*}
which implies 
\begin{gather*}
KD=\left[\begin{array}{c;{2pt/2pt} c c}
U & 0 \\ \hdashline[2pt/2pt]
0 & I_k
\end{array}\right]
J
\left[\begin{array}{c;{2pt/2pt} c c}
U & 0 \\ \hdashline[2pt/2pt]
0 & I_k
\end{array}\right]^T
\left[\begin{array}{c;{2pt/2pt} c c}
U & 0 \\ \hdashline[2pt/2pt]
0 & I_k
\end{array}\right]
\left[\begin{array}{c;{2pt/2pt} c c}
C_{11} & C_{12} \\ \hdashline[2pt/2pt]
C_{21} & C_{22}
\end{array}\right]
\left[\begin{array}{c;{2pt/2pt} c c}
U & 0 \\ \hdashline[2pt/2pt]
0 & I_k
\end{array}\right]^T\\
=
\left[\begin{array}{c;{2pt/2pt} c c}
U & 0 \\ \hdashline[2pt/2pt]
0 & I_k
\end{array}\right]
JC
\left[\begin{array}{c;{2pt/2pt} c c}
U & 0 \\ \hdashline[2pt/2pt]
0 & I_k
\end{array}\right]^T
=
\left[\begin{array}{c;{2pt/2pt} c c}
U(JC)_{11}U^T & U(JC)_{12} \\ \hdashline[2pt/2pt]
(JC)_{21}U^T & (JC)_{22}
\end{array}\right].
\end{gather*}
Also, using block multiplication and the block forms for $D$ and $K$, it follows that
\begin{gather*}
KD=\left[\begin{array}{c;{2pt/2pt} c c}
K_{11}D_{11} & K_{11}D_{12} \\ \hdashline[2pt/2pt]
D_{21} & D_{22}
\end{array}\right]
=
\left[\begin{array}{c c; {2pt/2pt} c}
0_k & 0 & -I_k \\
0 & A_{22} & A_{21} \\ \hdashline[2pt/2pt]
-I_k & A_{12} & A_{11}
\end{array}\right].
\end{gather*}
Now by Proposition \ref{PropMatrixMultOfSchurCompl} we know that
\begin{gather*}
U(JC)/(JC)_{22}U^T=(KD)/(KD)_{22},
\end{gather*}
where $KD=[(KD)_{ij}]_{i,j=1,2}$ is given the $2\times 2$ block matrix form
\begin{gather*}
KD=\left[\begin{array}{c;{2pt/2pt} c c}
(KD)_{11} & (KD)_{12} \\ \hdashline[2pt/2pt]
(KD)_{21} & (KD)_{22}
\end{array}\right]=
\left[\begin{array}{c;{2pt/2pt} c c}
U(JC)_{11}U^T & U(JC)_{12} \\ \hdashline[2pt/2pt]
(JC)_{21}U^T & (JC)_{22}
\end{array}\right].
\end{gather*}
From these facts we conclude that
\begin{gather*}
\begin{bmatrix}
-A_{11}^{-1} & A_{11}^{-1}A_{12}\\
A_{21}A_{11}^{-1} & A/A_{11}
\end{bmatrix}=K_{11}\operatorname{ppt}_1(A)=K_{11}U\operatorname{ppt}_2(B)U^T\\ =UJ_{11}\operatorname{ppt}_2(B)U^T
=U^T(JC)/(JC)_{22}U=(KD)/(KD)_{22}.
\end{gather*}
Finally, if $A$ is real, symmetric, Hermitian, or real and symmetric then $B=U^TAU$ is real, symmetric, Hermitian, or real and symmetric, respectively, so by Proposition \ref{PropPPTIsASchurComplement} it follows that $J_{11}\operatorname{ppt}_2(B)$ and $JC$ are both real, symmetric, Hermitian, or real and symmetric, respectively, which implies $K_{11}\operatorname{ppt}_1(A)=UJ_{11}\operatorname{ppt}_2(B)U^T$ and $KD=\left[\begin{array}{c;{2pt/2pt} c c}
U & 0 \\ \hdashline[2pt/2pt]
0 & I_k
\end{array}\right]
JC
\left[\begin{array}{c;{2pt/2pt} c c}
U & 0 \\ \hdashline[2pt/2pt]
0 & I_k
\end{array}\right]^T$ are both real, symmetric, Hermitian, or real and symmetric, respectively. This completes the proof.
\end{proof}

\section*{Acknowledgments}

The authors would like to thank Graeme W.\ Milton, Mihai Putinar, Joseph A. Ball, and Victor Vinnikov for all the helpful conversations that made this paper possible.\ Both authors are indebted to the reviewer for their valuable comments and suggestions on our original paper which helped improve the presentation. We are especially appreciative of the reviewer for outlining the alternative approach in Remark \ref{RemAltApproach}, commentary on relevant work, and bringing to our attention the references \cite{05BGM,08BGKR, 79BGK, 18HMS}.

\section*{Declarations}

\subsection*{Funding} Not applicable.
\subsection*{Conflicts of interest/Competing interests} Not applicable.
\subsection*{Availability of data and material} Not applicable.
\subsection*{Code availability} Not applicable.

\end{document}